\documentclass[11pt]{amsart}

\usepackage{
    amsmath,
    amsfonts,
    amssymb,
    amsthm,
    amscd,
    comment,
    enumitem,
    etoolbox,
    textcomp,   
    gensymb,    
    mathtools,
    mathdots
}
\usepackage[usenames,dvipsnames]{xcolor}
\usepackage[all]{xy}

\makeatletter
\@namedef{subjclassname@2020}{%
  \textup{2020} Mathematics Subject Classification}
\makeatother

\usepackage[T1]{fontenc}
\usepackage{bbm}                     
\usepackage[colorlinks=true, linkcolor=blue, citecolor=blue, urlcolor=blue, breaklinks=true]{hyperref}

\DeclareFontFamily{OT1}{pzc}{}
\DeclareFontShape{OT1}{pzc}{m}{it}{<-> s * [1.10] pzcmi7t}{}
\DeclareMathAlphabet{\mathpzc}{OT1}{pzc}{m}{it}

\usepackage[capitalize]{cleveref}   

\crefname{cor}{Corollary}{Corollaries}
\crefname{defin}{Definition}{Definitions}
\crefname{eg}{Example}{Examples}
\crefname{lem}{Lemma}{Lemmas}
\crefname{prop}{Proposition}{Propositions}
\crefname{theo}{Theorem}{Theorems}
\crefname{equation}{}{}
\crefname{enumi}{}{}

\newcommand\C{\mathbb{C}}
\newcommand\N{\mathbb{N}}

\newcommand\Z{\mathbb{Z}}
\newcommand\kk{\Bbbk}
\newcommand\one{\mathbbm{1}}

\newcommand\bC{\mathbf{C}}
\newcommand\bFh{\widehat{\bF}}

\newcommand\cA{\mathcal{A}}
\newcommand\cB{\mathcal{B}}
\newcommand\cC{\mathcal{C}}

\newcommand\fq{\mathfrak{q}}
\newcommand\gl{\mathfrak{gl}}

\newcommand\op{\mathrm{op}}
\newcommand\Rep{\mathrm{\underline{Re}p}}
\newcommand\rev{\mathrm{rev}}

\newcommand\tI{\mathtt{I}}

\newcommand{\md}{\textup{-mod}}
\newcommand\Sym{\textup{Sym}}                       
\newcommand{\smod}{\textup{-smod}}

\newcommand\odd{\textup{odd}}

\newcommand\cg{\pi}                                 

\newcommand{\xrightarrowdbl}[2][]{
  \xrightarrow[#1]{#2}\mathrel{\mkern-14mu}\rightarrow
}

\newcommand\AOB{\mathpzc{AOB}}          
\newcommand\AOBC{\mathpzc{AOBC}}        
\newcommand\AOS{\mathpzc{AOS}}          
\newcommand\AQcat{\mathpzc{AQ}}         
\newcommand\cEnd{\mathpzc{End}}         
\newcommand\OB{\mathpzc{OB}}            
\newcommand\OBC{\mathpzc{OBC}}          
\newcommand\OS{\mathpzc{OS}}            
\newcommand\Qcat{\mathpzc{Q}}           
\newcommand\SEnd{\mathpzc{SEnd}}        
\newcommand\SVec{\mathpzc{SVec}}        
\newcommand\Wcat{\mathpzc{W}}

\newcommand\bF{\mathbf{F}}
\newcommand\redFunc{\mathbf{F}^\bullet}

\newcommand\AHC{\mathrm{AHC}}           
\newcommand\BC{\mathrm{BC}}             
\newcommand\HC{\mathrm{HC}}             

\DeclareMathOperator{\coev}{coev}
\DeclareMathOperator{\End}{End}
\DeclareMathOperator{\ev}{ev}
\DeclareMathOperator{\Ev}{Ev}
\DeclareMathOperator{\flip}{flip}
\DeclareMathOperator{\Hom}{Hom}
\DeclareMathOperator{\id}{id}
\DeclareMathOperator{\Mat}{Mat}

\DeclareMathOperator{\sgn}{sgn}

\usepackage{tikz}
\usetikzlibrary{arrows,patterns}
\usepackage{tikz-cd}

\newcommand{\strandlabel}[1]{$\scriptstyle{#1}$}
\newcommand{\braidup}{to[out=up,in=down]}
\newcommand{\braiddown}{to[out=down,in=up]}
\newcommand{\coupon}[3][0.15]{
    \filldraw[draw=black,fill=white] (#2) circle (#1);
    \node at (#2) {$\scriptscriptstyle{#3}$}
}
\newcommand{\token}[1]{
    \filldraw[black] (#1) circle (1.5pt)
}
\newcommand{\opendot}[1]{
    \filldraw[fill=white,draw=black] (#1) circle (1.5pt)
}
\newcommand{\zebra}[1]{
    \filldraw[draw=Mahogany,fill=white] (#1)+(0.075,0) arc(0:-180:0.075);
    \filldraw[Mahogany] (#1)+(0.075,0) arc(0:180:0.075)
}
\newcommand{\zebras}[3]{
    \zebra{#1};
    \draw (#1) node[anchor=#2] {\color{Mahogany} $\scriptstyle{#3}$}
}
\newcommand{\bubright}[1]{
    \draw[->] (#1)+(0.2,0) arc(360:0:0.2)
}
\newcommand{\bubleft}[1]{
    \draw[->] (#1)+(-0.2,0) arc(-180:180:0.2)
}
\newcommand{\bubun}[1]{
    \draw (#1)+(0.2,0) arc(0:360:0.2)
}

\newcommand{\posupcross}{
    \begin{tikzpicture}[centerzero]
        \draw[->] (0.2,-0.2) -- (-0.2,0.2);
        \draw[wipe] (-0.2,-0.2) -- (0.2,0.2);
        \draw[->] (-0.2,-0.2) -- (0.2,0.2);
    \end{tikzpicture}
}
\newcommand{\negupcross}{
    \begin{tikzpicture}[centerzero]
        \draw[->] (-0.2,-0.2) -- (0.2,0.2);
        \draw[wipe] (0.2,-0.2) -- (-0.2,0.2);
        \draw[->] (0.2,-0.2) -- (-0.2,0.2);
    \end{tikzpicture}
}
\newcommand{\posrightcross}{
    \begin{tikzpicture}[centerzero]
        \draw[->] (-0.2,-0.2) -- (0.2,0.2);
        \draw[wipe] (0.2,-0.2) -- (-0.2,0.2);
        \draw[<-] (0.2,-0.2) -- (-0.2,0.2);
    \end{tikzpicture}
}
\newcommand{\negrightcross}{
    \begin{tikzpicture}[centerzero]
        \draw[<-] (0.2,-0.2) -- (-0.2,0.2);
        \draw[wipe] (-0.2,-0.2) -- (0.2,0.2);
        \draw[->] (-0.2,-0.2) -- (0.2,0.2);
    \end{tikzpicture}
}
\newcommand{\posdowncross}{
    \begin{tikzpicture}[centerzero]
        \draw[<-] (0.2,-0.2) -- (-0.2,0.2);
        \draw[wipe] (-0.2,-0.2) -- (0.2,0.2);
        \draw[<-] (-0.2,-0.2) -- (0.2,0.2);
    \end{tikzpicture}
}
\newcommand{\negdowncross}{
    \begin{tikzpicture}[centerzero]
        \draw[<-] (-0.2,-0.2) -- (0.2,0.2);
        \draw[wipe] (0.2,-0.2) -- (-0.2,0.2);
        \draw[<-] (0.2,-0.2) -- (-0.2,0.2);
    \end{tikzpicture}
}
\newcommand{\negleftcross}{
    \begin{tikzpicture}[centerzero]
        \draw[->] (0.2,-0.2) -- (-0.2,0.2);
        \draw[wipe] (-0.2,-0.2) -- (0.2,0.2);
        \draw[<-] (-0.2,-0.2) -- (0.2,0.2);
    \end{tikzpicture}
}
\newcommand{\posleftcross}{
    \begin{tikzpicture}[centerzero]
        \draw[<-] (-0.2,-0.2) -- (0.2,0.2);
        \draw[wipe] (0.2,-0.2) -- (-0.2,0.2);
        \draw[->] (0.2,-0.2) -- (-0.2,0.2);
    \end{tikzpicture}
}
\newcommand{\symcross}{
    \begin{tikzpicture}[centerzero]
        \draw[->] (-0.2,-0.2) -- (0.2,0.2);
        \draw[->] (0.2,-0.2) -- (-0.2,0.2);
    \end{tikzpicture}
}
\newcommand{\posupredcross}{
    \begin{tikzpicture}[centerzero]
        \draw[->] (0.2,-0.2) -- (-0.2,0.2);
        \draw[wipe]  (-0.2,-0.2) -- (0.2,0.2);
        \draw[thick,red]  (-0.2,-0.2) -- (0.2,0.2);
    \end{tikzpicture}
}
\newcommand{\negupredcross}{
    \begin{tikzpicture}[centerzero]
        \draw[->] (-0.2,-0.2) -- (0.2,0.2);
        \draw[wipe]  (0.2,-0.2) -- (-0.2,0.2);
        \draw[thick,red] (0.2,-0.2) -- (-0.2,0.2);
    \end{tikzpicture}
}
\newcommand{\posdownredcross}{
    \begin{tikzpicture}[centerzero]
        \draw[<-] (0.2,-0.2) -- (-0.2,0.2);
        \draw[wipe]  (-0.2,-0.2) -- (0.2,0.2);
        \draw[thick,red]  (-0.2,-0.2) -- (0.2,0.2);
    \end{tikzpicture}
}
\newcommand{\negdownredcross}{
    \begin{tikzpicture}[centerzero]
        \draw[<-] (-0.2,-0.2) -- (0.2,0.2);
        \draw[wipe]  (0.2,-0.2) -- (-0.2,0.2);
        \draw[thick,red] (0.2,-0.2) -- (-0.2,0.2);
    \end{tikzpicture}
}

\newcommand{\tokup}{
    \begin{tikzpicture}[centerzero]
        \draw[->] (0,-0.2) -- (0,0.2);
        \token{0,0};
    \end{tikzpicture}
}
\newcommand{\tokdown}{
    \begin{tikzpicture}[centerzero]
        \draw[<-] (0,-0.2) -- (0,0.2);
        \token{0,0};
    \end{tikzpicture}
}

\newcommand{\upstrand}{
    \begin{tikzpicture}[centerzero]
        \draw[->] (0,-0.2) -- (0,0.2);
    \end{tikzpicture}
}
\newcommand{\downstrand}{
    \begin{tikzpicture}[centerzero]
        \draw[<-] (0,-0.2) -- (0,0.2);
    \end{tikzpicture}
}

\newcommand{\redobject}{
    \begin{tikzpicture}[centerzero]
        \draw[thick,red] (0,-0.17) -- (0,0.17);
    \end{tikzpicture}
}

\newcommand{\dotup}{
    \begin{tikzpicture}[centerzero]
        \draw[->] (0,-0.2) -- (0,0.2);
        \opendot{0,0};
    \end{tikzpicture}
}
\newcommand{\dotdown}{
    \begin{tikzpicture}[centerzero]
        \draw[->] (0,0.2) -- (0,-0.2);
        \opendot{0,0};
    \end{tikzpicture}
}
\newcommand{\zebralone}{
    \begin{tikzpicture}[centerzero]
        \zebra{0,0};
    \end{tikzpicture}
}
\newcommand{\zebraup}{
    \begin{tikzpicture}[centerzero]
        \draw[->] (0,-0.2) -- (0,0.2);
        \zebra{0,-0.02};
    \end{tikzpicture}
}
\newcommand{\zebramultup}[1]{
    \begin{tikzpicture}[centerzero]
        \draw[->] (0,-0.2) -- (0,0.2);
        \zebras{0,-0.02}{west}{#1};
    \end{tikzpicture}
}

\newcommand{\rightcup}{
    \begin{tikzpicture}[anchorbase]
        \draw[->] (-0.15,0.15) -- (-0.15,0) arc(180:360:0.15) -- (0.15,0.15);
    \end{tikzpicture}
}
\newcommand{\leftcup}{
    \begin{tikzpicture}[anchorbase]
        \draw[<-] (-0.15,0.15) -- (-0.15,0) arc(180:360:0.15) -- (0.15,0.15);
    \end{tikzpicture}
}
\newcommand{\rightcap}{
    \begin{tikzpicture}[anchorbase]
        \draw[->] (-0.15,-0.15) -- (-0.15,0) arc(180:0:0.15) -- (0.15,-0.15);
    \end{tikzpicture}
}
\newcommand{\leftcap}{
    \begin{tikzpicture}[anchorbase]
        \draw[<-] (-0.15,-0.15) -- (-0.15,0) arc(180:0:0.15) -- (0.15,-0.15);
    \end{tikzpicture}
}

\newcommand{\rightbub}{
    \begin{tikzpicture}[centerzero]
        \bubright{0,0};
    \end{tikzpicture}
}

\newcommand{\unbub}{
    \begin{tikzpicture}[centerzero]
        \bubun{0,0};
    \end{tikzpicture}
}
\newcommand{\rightzebrabub}[1]{
    \begin{tikzpicture}[centerzero]
        \draw[->] (0.2,0) arc(360:0:0.2);
        \zebras{-0.2,0}{east}{#1};
    \end{tikzpicture}
}
\newcommand{\leftzebrabub}[1]{
    \begin{tikzpicture}[centerzero]
        \draw[->] (-0.2,0) arc(-180:180:0.2);
        \zebras{0.2,0}{west}{#1};
    \end{tikzpicture}
}

\tikzset{anchorbase/.style={>=To,baseline={([yshift=-0.5ex]current bounding box.center)}}}
\tikzset{ 
    centerzero/.style={>=To,baseline={([yshift=-0.5ex](#1))}},
    centerzero/.default={0,0}
}
\tikzset{wipe/.style={white,line width=4pt}}

\newtheorem{theo}{Theorem}[section]

\newtheorem{prop}[theo]{Proposition}
\newtheorem{lem}[theo]{Lemma}
\newtheorem{cor}[theo]{Corollary}
\newtheorem{conj}[theo]{Conjecture}

\theoremstyle{definition}
\newtheorem{defin}[theo]{Definition}
\newtheorem{rem}[theo]{Remark}

\numberwithin{equation}{section}
\allowdisplaybreaks

\setenumerate[1]{label=(\alph*)}          

\newtoggle{comments}
\newtoggle{details}
\newtoggle{detailsnote}

\toggletrue{detailsnote}   

\iftoggle{comments}{%
    \usepackage[notref,notcite]{showkeys}   
    \newcommand{\acomments}[1]{
        \ \\
        {\color{red}
            \textbf{AS:} #1
        }
        \ \\
    }
    \newcommand{\question}[1]{
        \ \\
        {\color{blue}
            \textbf{Question:} #1
        }
        \ \\
    }
    }{%
        \newcommand{\acomments}[1]{}
        \newcommand{\question}[1]{}
    }

\iftoggle{details}{%
    \newcommand{\details}[1]{
        \ \\
        {\color{OliveGreen}
            \textbf{Details:} #1
        }
        \\
    }
}{%
    \newcommand{\details}[1]{}
}

\begin{document}

\title{The quantum isomeric supercategory}

\author{Alistair Savage}
\address{
  Department of Mathematics and Statistics \\
  University of Ottawa \\
  Ottawa, ON K1N 6N5, Canada
}
\urladdr{\href{https://alistairsavage.ca}{alistairsavage.ca}, \textrm{\textit{ORCiD}:} \href{https://orcid.org/0000-0002-2859-0239}{orcid.org/0000-0002-2859-0239}}
\email{alistair.savage@uottawa.ca}

\begin{abstract}
    We introduce the quantum isomeric supercategory and the quantum affine isomeric supercategory.  These diagrammatically defined supercategories, which can be viewed as isomeric analogues of the HOMFLYPT skein category and its affinization, provide powerful categorical tools for studying the representation theory of the quantum isomeric superalgebras (commonly known as quantum queer superalgebras).
\end{abstract}



\ifboolexpr{togl{comments} or togl{details}}{%
  {\color{magenta}DETAILS OR COMMENTS ON}
}{%
}

\maketitle
\thispagestyle{empty}


\section{Introduction}

One of the most fundamental facts in representation theory is Schur's lemma, which implies that if $V$ is a finite-dimensional simple module over an associative $\kk$-algebra $A$, where $\kk$ is an algebraically closed field, then $\End_A(V) \cong \kk$.  On the other hand, if $A$ is an associative $\kk$-\emph{super}algebra, then there are two possibilities: we can have $\End_A(V) \cong \kk$ or we can have that $\End_A(V)$ is a two-dimensional Clifford superalgebra generated by the parity shift.  In the theory of Lie superalgebras, this phenomenon underlies the fact that the general linear Lie algebra $\gl_n$ has two natural analogues in the super setting: the general linear Lie superalgebra $\gl_{m|n}$ and the isomeric Lie superalgebra $\fq_n$.  (Following \cite{NSS21}, we use the term \emph{isomeric} instead of the more traditional term \emph{queer}.)  The purpose of the current paper is to develop diagrammatic tools for studying the representation theory of the quantum analogue of $\fq_n$.  Our hope is that this is the starting point of the development of isomeric analogues of much of the rich mathematics that has emerged from connections between low-dimensional topology,  representation theory, and categorification.

Before describing our results, we begin with an overview of the situation for the better-understood case of $\gl_{m|n}$.  The finite-dimensional complex representation theory of $\gl_{m|n}$ is controlled by the \emph{oriented Brauer category} $\OB(m-n)$.  More precisely, $\OB(t)$ is a diagrammatic symmetric monoidal category depending on a dimension parameter $t$.  When $t=m-n$, there is a full monoidal functor
\[
    \OB(m-n) \to \gl_{m|n}\smod
\]
to the category of $\gl_{m|n}$-supermodules, sending the generating object of $\OB(m-n)$ to the natural supermodule of $\gl_{m|n}$; see \cite[\S8.3]{CW12}.  (Throughout this introduction we work with finite-dimensional supermodules.)  The additive Karoubi envelope (i.e.\ the idempotent completion of the additive envelope) of $\OB(t)$ is Deligne's interpolating category $\Rep(\mathrm{GL}_t)$.  Similar statements hold in the orthosymplectic case, where $\OB(t)$ is replaced by the \emph{Brauer category} (no longer oriented, due to the fact that the natural supermodule is self-dual); see \cite[Th.~5.6]{LZ17}.

Any monoidal category acts on itself via the tensor product.  In particular, \emph{translation functors}, given by tensoring with a given supermodule, are key tools for studying the representation theory of Lie superalgebras.  In the case of $\gl_{m|n}$, this action by tensoring can be enlarged to a monoidal functor
\[
    \AOB(m-n) \to \cEnd(\gl_{m|n}\smod),
\]
where $\AOB(t)$ is the \emph{affine oriented Brauer category} of \cite{BCNR17} and $\cEnd(\cC)$ denotes the monoidal category of endofunctors of a category $\cC$.  The category $\AOB(t)$ allows one to study natural transformations between translation functors, provides tools to study cyclotomic Hecke algebras, and yields natural elements in the center of $U(\gl_{m|n})$.  Again, a similar picture exists for the orthosymplectic Lie superalgebras, where $\AOB(t)$ is replaced by the \emph{affine Brauer category} of \cite{RS19}.

Quantum analogues of the above pictures play a particularly important role in connections to link invariants and integrable models in statistical mechanics.  The quantum analogue of the oriented Brauer category is the \emph{HOMFLYPT skein category} $\OS(z,t)$, originally introduced in \cite[\S5.2]{Tur89}, where it was called the Hecke category.  The affine version $\AOS(z,t)$ was introduced in \cite{Bru17}, and there are monoidal functors
\[
    \OS(q-q^{-1},q^n) \to U_q(\gl_n)\md,\qquad
    \AOS(q-q^{-1},q^n) \to \cEnd(U_q(\gl_n)\md),
\]
with many of the properties mentioned above for the non-quantum case.  We expect that these functors can be generalized to the super setting of $U_q(\gl_{m|n})$.  The generalization of the first functor should follow from the results \cite{LZZ20}, and then the affine case follows from the general affinization procedure of \cite{MS21}.  Once again, analogues exist in the orthosymplectic case, where the relevant categories are the \emph{Kauffman skein category}, together with its affine analogue introduced in \cite{GRS20}.

The isomeric analogues of the oriented Brauer category and its affine version are the \emph{oriented Brauer--Clifford supercategory} $\OBC$ and the \emph{degenerate affine oriented Brauer--Clifford supercategory} $\AOBC$ introduced in \cite{BCK19}.  In analogy with the above, one has monoidal superfunctors
\[
    \OBC \to \fq_n\smod,\qquad
    \AOBC \to \SEnd(\fq_n\smod),
\]
where $\SEnd(\cC)$ denotes the monoidal supercategory of endosuperfunctors of a supercategory $\cC$.  The need to move to the setting of \emph{super}categories here arises from the super version of Schur's lemma mentioned earlier.  There is an \emph{odd} endomorphism of the natural representation of $\fq_n$ that corresponds to an odd morphism in $\OBC$ and $\AOBC$.  (The category of vector superspaces, with parity preserving linear maps, is a monoidal category, with no need to introduce the notion of a monoidal \emph{super}category.)  Note also the absence of the parameter $t$ that appears for the oriented Brauer category.  This is because the natural representation of $\fq_n$ always has superdimension zero.

In the current paper we develop analogues of the above results for the \emph{quantum isomeric superalgebra} $U_q(\fq_n)$.  As we will explain below, this case requires several new techniques.  We begin by defining the \emph{quantum isomeric supercategory} $\Qcat(z)$ depending on a parameter $z$ in the ground ring.  This is a strict monoidal supercategory generated by two objects, $\uparrow$ and $\downarrow$, and morphisms
\begin{gather*}
    \posupcross, \negupcross \colon \uparrow \otimes \uparrow\ \to\ \uparrow \otimes \uparrow,
    \qquad
    \negrightcross \colon \uparrow \otimes \downarrow\ \to\ \downarrow \otimes \uparrow,
    \\
    \leftcap \colon \downarrow \otimes \uparrow\ \to \one,
    \qquad
    \leftcup \colon \one \to\ \uparrow \otimes \downarrow,
    \qquad
    \tokup \colon \uparrow\ \to\ \uparrow,
\end{gather*}
subject to certain relations; see \cref{Qdef}.  The supercategory $\Qcat(z)$ should be viewed as a quantization of $\OBC$.  In particular, $\Qcat(0)$ is isomorphic to $\OBC$; see \cref{degenerate}.  From the definition of $\Qcat(z)$, we deduce further relations, showing, in particular, that this supercategory is pivotal.  We also prove a basis theorem (\cref{basisthm}) showing that the morphism spaces have bases given by tangle-like diagrams where strands can carry the odd \emph{Clifford token} $\tokup$ corresponding to the odd endomorphism appearing in the super version of Schur's lemma.  We define, in \cref{bread}, a monoidal superfunctor
\begin{equation} \label{spore}
    \Qcat(q-q^{-1}) \to U_q(\fq_n)\smod,
\end{equation}
which we call the \emph{incarnation superfunctor}.  This superfunctor is full and asymptotically faithful, in the sense that the induced map on any morphism space in $\Qcat(q-q^{-1})$ is an isomorphism for sufficiently large $n$ (\cref{full}).  This can be viewed as a categorical version of the \emph{first fundamental theorem} for $U_q(\fq_n)$-invariants.

The endomorphism superalgebras $\End_{\Qcat(z)}(\uparrow^{\otimes r})$ are \emph{Hecke--Clifford superalgebras}, which appear in quantum Sergeev duality (the quantum isomeric analogue of Schur--Weyl duality).  More generally, the endomorphism superalgebras in $\Qcat(z)$ are isomorphic to the \emph{quantum walled Brauer--Clifford superalgebras} introduced in \cite{BGJKW16}; see \cref{window}.  However, the category $\Qcat(z)$ contains more information, since it also involves morphism spaces between \emph{different} objects.  Consideration of the entire monoidal category $\Qcat(z)$, as opposed to the more traditional approach (e.g.\ taken in \cite{BGJKW16}) of treating the endomorphisms superalgebras individually, as associative superalgebras, offers significant advantages.  In particular, the added structure of cups and caps, arising from the duality between $V$ and $V^*$, allows us to translate between general morphism spaces and ones of the form $\End_{\Qcat(z)}(\uparrow^{\otimes r})$.  This allows us to recover some of the results of \cite{BGJKW16} with simplified arguments.

The additive Karoubi envelope of $\Qcat(z)$ should be viewed as an interpolating category $\Rep(U_q(\fq))$ for the quantum isomeric superalgebras.  However, since the supercategory $\Qcat(z)$ does not depend on $n$, we have an ``interpolating'' category without a dimension parameter.  The same is true for the additive Karoubi envelope $\Rep(Q)$ of $\OBC$, which is the isomeric ``interpolating'' category in the non-quantum setting.  Of course, the kernel of the incarnation superfunctor \cref{spore} \emph{does} depend on $n$; see, for example, \cref{full}.  The semisimplification of $\Qcat(z)$, which is the quotient by the tensor ideal of negligible morphisms, is the trivial supercategory with one object, since the identity morphisms of the generating objects $\uparrow$ and $\downarrow$ are negligible.  Similar phenomena occur for the periplectic Lie superalgebras.  For a discussion of the Deligne interpolating category in that case, we refer the reader to \cite[\S4.5]{Ser14}, \cite[\S5]{KT17}, \cite[\S3.1]{CE21}, and \cite{ES21}.

In the second half of the current paper, we define and study the \emph{quantum affine isomeric supercategory} $\AQcat(z)$.  One important difference between the quantum isomeric supercategory and the HOMFLYPT skein category is that the category $\Qcat(z)$ is \emph{not} braided.  This corresponds to the fact that $U_q(\fq_n)$ is \emph{not} a quasi-triangular Hopf superalgebra.  Diagrammatically, this is manifested in the fact (see \cref{nickelfull}) that
\[
    \begin{tikzpicture}[centerzero]
        \draw[->] (0.3,-0.3) -- (-0.3,0.3);
        \draw[wipe] (-0.3,-0.3) -- (0.3,0.3);
        \draw[->] (-0.3,-0.3) -- (0.3,0.3);
        \token{-0.15,-0.15};
    \end{tikzpicture}
    =
    \begin{tikzpicture}[centerzero]
        \draw[->] (0.3,-0.3) -- (-0.3,0.3);
        \draw[wipe] (-0.3,-0.3) -- (0.3,0.3);
        \draw[->] (-0.3,-0.3) -- (0.3,0.3);
        \token{0.15,0.15};
    \end{tikzpicture}
    \qquad \text{but} \qquad
    \begin{tikzpicture}[centerzero]
        \draw[->] (-0.3,-0.3) -- (0.3,0.3);
        \draw[wipe] (0.3,-0.3) -- (-0.3,0.3);
        \draw[->] (0.3,-0.3) -- (-0.3,0.3);
        \token{-0.18,-0.18};
    \end{tikzpicture}
    \ne
    \begin{tikzpicture}[centerzero]
        \draw[->] (-0.3,-0.3) -- (0.3,0.3);
        \draw[wipe] (0.3,-0.3) -- (-0.3,0.3);
        \draw[->] (0.3,-0.3) -- (-0.3,0.3);
        \token{0.15,0.15};
    \end{tikzpicture}
    \ .
\]
That is, Clifford tokens slide over crossings, but not under them.  Since $\Qcat(z)$ is not braided, the usual affinization procedure, which corresponds to considering string diagrams on a cylinder (see \cite{MS21}) is not appropriate.  Instead, we must develop a new approach.  To pass from $\Qcat(z)$ to the affine version $\AQcat(z)$, we adjoin an odd morphism $\dotup \colon \uparrow\ \to\ \uparrow$ satisfying
\[
    \begin{tikzpicture}[centerzero]
        \draw[->] (0.3,-0.3) -- (-0.3,0.3);
        \draw[wipe] (-0.3,-0.3) -- (0.3,0.3);
        \draw[->] (-0.3,-0.3) -- (0.3,0.3);
        \opendot{0.17,-0.17};
    \end{tikzpicture}
    =
    \begin{tikzpicture}[centerzero]
        \draw[->] (0.3,-0.3) -- (-0.3,0.3);
        \draw[wipe] (-0.3,-0.3) -- (0.3,0.3);
        \draw[->] (-0.3,-0.3) -- (0.3,0.3);
        \opendot{-0.15,0.15};
    \end{tikzpicture}
    \ ,
\]
among other relations; see \cref{AQdef}.  This procedure of \emph{odd affinization} (see \cref{oddaffinization}) makes apparent a symmetry of $\AQcat(z)$ that interchanges $\tokup$ and $\dotup$ and flips all crossings.  There does not seem to be any analogous symmetry of the affine HOMFLYPT skein category.  The supercategory $\Qcat(z)$ is naturally a sub-supercategory of $\AQcat(z)$ (\cref{feline}).

We define, in \cref{butter}, a monoidal superfunctor
\begin{equation} \label{affinespore}
    \AQcat(q-q^{-1}) \to \SEnd(U_q(\fq_n)\smod),
\end{equation}
which we call the \emph{affine action superfunctor}.  As for the case of the affine HOMFLYPT skein category, the superfunctor \cref{affinespore} contains information about supernatural transformations between translation superfunctors acting on $U_q(\fq_n\smod)$.  However, in the HOMFLYPT setting, the affine action comes from the braiding in the category.  Intuitively, it arises from an action of $\AOS(z,t)$ on $\OS(z,t)$ corresponding to placing string diagrams representing morphisms of $\OS(z,t)$ inside the cylinders representing morphisms of $\AOS(z,t)$.  We refer the reader to \cite[\S3]{MS21} for further details of this interpretation.  The fact that $\Qcat(z)$ is not braided means that we cannot simply apply this general framework and we must formulate new methods.  As a replacement, we develop in \cref{sec:chiral} the concept of a \emph{chiral braiding}, which is similar to a braiding, but is only natural in one argument.

The endomorphism superalgebras $\End_{\AQcat(z)}(\uparrow^{\otimes r})$ are related to the \emph{affine Hecke--Clifford superalgebras} introduced in \cite{JN99}, where they are called \emph{affine Sergeev algebras}; see \cref{sec:affineEnd}.  These have played an important role in representation theory and categorification; see, for example, \cite{BK01}.  However, our presentation of these superalgebras is different from the original one appearing in \cite{JN99}.  There, the affine Hecke--Clifford superalgebra is obtained from the Hecke--Clifford superalgebra by adding a set of pairwise-commuting even elements.  In our presentation, we add pairwise-supercommuting \emph{odd} elements, corresponding to the odd generator $\dotup$ appearing on various strands.   While the translation between the two presentations is straightforward, the new approach yields a simpler description of the affine Hecke--Clifford superalgebras with an obvious symmetry, corresponding to the symmetry of $\AQcat(z)$ that interchanges $\tokup$ and $\dotup$ and flips crossings. The more general endomorphism superalgebras $\End_{\AQcat(z)}(\uparrow^{\otimes r} \otimes \downarrow^{\otimes s})$ are affine versions of quantum walled Brauer--Clifford superalgebras which do not seem to have appeared in the literature.

As a final application of our approach to the representation theory of the quantum isomeric superalgebra, we use the affine action superfunctor \cref{affinespore} to compute an infinite sequence of elements \cref{dark} in the center of $U_q(\fq_n)$.  These elements arise from ``bubbles'' in $\AQcat(z)$, which are closed diagrams corresponding to endomorphisms of the unit object.  We expect these elements will be useful in a computation of the center of $U_q(\fq_n)$, which has yet to appear in the literature.  Typically one uses the Harish--Chandra homomorphism to compute centers.  This homomorphism has recently been studied for basic classical Lie superalgebras in \cite{LWY21}, but the quantum isomeric case remains open.  It is often not difficult to show that the Harish--Chandra homomorphism is injective.  The difficulty lies in showing that its image is as large as expected.  By analogy with the $U_q(\gl_n)$ case, we expect that the central elements \cref{dark} computed here, together with some obviously central elements, generate the center of $U_q(\fq_n)$.

\subsection*{Further directions and open problems}

The quantum affine isomeric supercategory $\AQcat(z)$ should be thought of as an isomeric analogue of the affine HOMFLYPT skein category from \cite[\S4]{Bru17}.  The latter is the central charge zero special case of the \emph{quantum Heisenberg category} of \cite{BSW-qheis}.  A suitable modification of the approach of \cite{BSW-qheis} should lead to the definition of a \emph{quantum isomeric Heisenberg supercategory} depending on a central charge $k \in \Z$.  Taking $k=0$ would recover $\AQcat(z)$.  On the other hand, for nonzero $k$, this supercategory should act on supercategories of supermodules over cyclotomic Hecke--Clifford superalgebras.  Furthermore, we expect that one can adapt the categorical comultiplication technique of \cite{BSW-qheis} to prove a basis theorem, yielding a proof of \cref{affinebasis} (giving a conjectural basis for each morphism space in $\AQcat(z)$) as a special case.

An even more general \emph{quantum Frobenius Heisenberg category} was defined in \cite{BSW-qFrobHeis}.  This is a monoidal supercategory depending on a central charge $k \in \Z$ and a Frobenius superalgebra $A$.  Taking $A = \kk$ recovers the usual quantum Heisenberg category.  It should be possible to define a \emph{quantum isomeric Frobenius Heisenberg supercategory} such that specializing $A=\kk$ yields the quantum isomeric Heisenberg category.

The quantum webs of type $Q$ introduced in \cite{BDK20} should be related to a partial idempotent completion of supercategory $\Qcat(z)$.  It would be interesting to work out this precise connection, and then use it to define affine versions of quantum webs of type $Q$, based on the supercategory $\AQcat(z)$.

Finally, in \cite{BCK19}, the authors studied cyclotomic quotients of the degenerate affine oriented Brauer--Clifford supercategory.  It would be natural to investigate the quantum analogue, namely cyclotomic quotients of $\AQcat(z)$.  These could also be thought of as isomeric analogues of the central charge zero case of the cyclotomic quotients considered in \cite[\S9]{BSW-qheis}.

\iftoggle{detailsnote}{
\subsection*{Hidden details} For the interested reader, the tex file of the \href{https://arxiv.org/abs/2207.03254}{arXiv version} of this paper includes hidden details of some straightforward computations and arguments that are omitted in the pdf file.  These details can be displayed by switching the \texttt{details} toggle to true in the tex file and recompiling.
}{}

\subsection*{Acknowledgements}

This research was supported by Discovery Grant RGPIN-2017-03854 from the Natural Sciences and Engineering Research Council of Canada.  The author would like to thank Jon Brundan and Dimitar Grantcharov for helpful conversations, and the referee for useful comments.

\section{The quantum isomeric supercategory\label{sec:Qdef}}

Throughout the paper we work over a commutative ring $\kk$ whose characteristic is not equal to two, and we fix an element $z \in \kk$.  Statements about abstract categories will typically be at this level of generality.  When making statements involving supermodules over the quantum isomeric superalgebra, we will specialize to $\kk = \C(q)$ and $z=q-q^{-1}$.  We let $\N$ denote the set of nonnegative integers.

All vector spaces, algebras, categories, and functors will be assumed to be linear over $\kk$ unless otherwise specified.  Almost everything in the paper will be enriched over the category $\SVec$ of vector superspaces with parity-preserving morphisms.  We write $\bar{v}$ for the parity of a homogeneous vector $v$ in a vector superspace.  When we write formulae involving parities, we assume the elements in question are homogeneous; we then extend by linearity.

For associative superalgebras $A$ and $B$, multiplication in the superalgebra $A \otimes B$ is defined by
\begin{equation}
    (a' \otimes b) (a \otimes b') = (-1)^{\bar a \bar b} a'a \otimes bb'
\end{equation}
for homogeneous $a,a' \in A$, $b,b' \in B$.  For $A$-supermodules $M$ and $N$, we let $\Hom_A(M,N)$ denote the $\kk$-supermodule of \emph{all} (i.e.\ not necessarily parity-preserving) $A$-linear maps from $M$ to $N$.  The \emph{opposite} superalgebra $A^\op$ is a copy $\{a^\op : a \in A\}$ of the vector superspace $A$ with multiplication defined from
\begin{equation}\label{mups}
    a^\op  b^\op := (-1)^{\bar a \bar b} (ba)^{\op}.
\end{equation}
A superalgebra homomorphism $A \to B^\op$ is equivalent to an antihomomorphism of superalgebras $A \to B$.  When viewing it in this way, we will often omit the superscript `$\op$' on elements of $B$.

Throughout this paper we will work with \emph{strict monoidal supercategories}, in the sense of \cite{BE17}.  We summarize here a few crucial properties that play an important role in the present paper.  A \emph{supercategory} means a category enriched in $\SVec$.  Thus, its morphism spaces are vector superspaces and composition is parity-preserving.  A \emph{superfunctor} between supercategories induces a parity-preserving linear map between morphism superspaces.  For superfunctors $F,G \colon \cA \to \cB$, a \emph{supernatural transformation} $\alpha \colon F \Rightarrow G$ of \emph{parity $r\in\Z/2$} is the data of morphisms $\alpha_X\in \Hom_{\cB}(FX, GX)$ of parity $r$, for each $X \in \cA$, such that $Gf \circ \alpha_X = (-1)^{r \bar f}\alpha_Y\circ Ff$ for each homogeneous $f \in \Hom_{\cA}(X, Y)$.  Note when $r$ is odd that $\alpha$ is \emph{not} a natural transformation in the usual sense due to the sign. A \emph{supernatural transformation} $\alpha \colon F \Rightarrow G$ is of the form $\alpha = \alpha_0 + \alpha_1$, with each $\alpha_r$ being a supernatural transformation of parity $r$.

In a \emph{strict monoidal supercategory}, morphisms satisfy the \emph{super interchange law}:
\begin{equation}\label{interchange}
    (f' \otimes g) \circ (f \otimes g')
    = (-1)^{\bar f \bar g} (f' \circ f) \otimes (g \circ g').
\end{equation}
We denote the unit object by $\one$ and the identity morphism of an object $X$ by $1_X$.  We will use the usual calculus of string diagrams, representing the horizontal composition $f \otimes g$ (resp.\ vertical composition $f \circ g$) of morphisms $f$ and $g$ diagrammatically by drawing $f$ to the left of $g$ (resp.\ drawing $f$ above $g$).  Care is needed with horizontal levels in such diagrams due to the signs arising from the super interchange law:
\begin{equation}\label{intlaw}
    \begin{tikzpicture}[anchorbase]
        \draw (-0.5,-0.5) -- (-0.5,0.5);
        \draw (0.5,-0.5) -- (0.5,0.5);
        \filldraw[fill=white,draw=black] (-0.5,0.15) circle (5pt);
        \filldraw[fill=white,draw=black] (0.5,-0.15) circle (5pt);
        \node at (-0.5,0.15) {$\scriptstyle{f}$};
        \node at (0.5,-0.15) {$\scriptstyle{g}$};
    \end{tikzpicture}
    \quad=\quad
    \begin{tikzpicture}[anchorbase]
        \draw (-0.5,-0.5) -- (-0.5,0.5);
        \draw (0.5,-0.5) -- (0.5,0.5);
        \filldraw[fill=white,draw=black] (-0.5,0) circle (5pt);
        \filldraw[fill=white,draw=black] (0.5,0) circle (5pt);
        \node at (-0.5,0) {$\scriptstyle{f}$};
        \node at (0.5,0) {$\scriptstyle{g}$};
    \end{tikzpicture}
    \quad=\quad
    (-1)^{\bar f\bar g}\
    \begin{tikzpicture}[anchorbase]
        \draw (-0.5,-0.5) -- (-0.5,0.5);
        \draw (0.5,-0.5) -- (0.5,0.5);
        \filldraw[fill=white,draw=black] (-0.5,-0.15) circle (5pt);
        \filldraw[fill=white,draw=black] (0.5,0.15) circle (5pt);
        \node at (-0.5,-0.15) {$\scriptstyle{f}$};
        \node at (0.5,0.15) {$\scriptstyle{g}$};
    \end{tikzpicture}
    \ .
\end{equation}
If $\cA$ is a supercategory, the category $\SEnd(\cA)$ of superfunctors $\cA \to \cA$ and supernatural transformations is a strict monoidal supercategory.  The notation $\cA^\op$ denotes the opposite supercategory and, if $\cA$ is also monoidal, $\cA^\rev$ denotes the reverse monoidal supercategory (changing the order of the tensor product); these are defined as for categories, but with appropriate signs.

\begin{defin} \label{Qdef}
    We define the \emph{quantum isomeric supercategory} $\Qcat(z)$ to be the strict monoidal supercategory generated by objects $\uparrow$ and $\downarrow$ and morphisms
    \begin{gather} \label{Qgen1}
        \posupcross, \negupcross \colon \uparrow \otimes \uparrow\ \to\ \uparrow \otimes \uparrow,
        \quad
        \negrightcross \colon \uparrow \otimes \downarrow\ \to\ \downarrow \otimes \uparrow,
        \\ \label{Qgen2}
        \leftcap \colon \downarrow \otimes \uparrow\ \to \one,
        \quad
        \leftcup \colon \one \to\ \uparrow \otimes \downarrow,
        \quad
        \tokup \colon \uparrow\ \to\ \uparrow,
    \end{gather}
    subject to the relations
    \begin{gather} \label{braid}
        \begin{tikzpicture}[centerzero]
            \draw[->] (0.2,-0.4) to[out=135,in=down] (-0.15,0) to[out=up,in=225] (0.2,0.4);
            \draw[wipe] (-0.2,-0.4) to[out=45,in=down] (0.15,0) to[out=up,in=-45] (-0.2,0.4);
            \draw[->] (-0.2,-0.4) to[out=45,in=down] (0.15,0) to[out=up,in=-45] (-0.2,0.4);
        \end{tikzpicture}
        \ =\
        \begin{tikzpicture}[centerzero]
            \draw[->] (-0.2,-0.4) -- (-0.2,0.4);
            \draw[->] (0.2,-0.4) -- (0.2,0.4);
        \end{tikzpicture}
        \ =\
        \begin{tikzpicture}[centerzero]
            \draw[->] (-0.2,-0.4) to[out=45,in=down] (0.15,0) to[out=up,in=-45] (-0.2,0.4);
            \draw[wipe] (0.2,-0.4) to[out=135,in=down] (-0.15,0) to[out=up,in=225] (0.2,0.4);
            \draw[->] (0.2,-0.4) to[out=135,in=down] (-0.15,0) to[out=up,in=225] (0.2,0.4);
        \end{tikzpicture}
        \ ,\quad
        \begin{tikzpicture}[centerzero]
            \draw[<-] (0.2,-0.4) to[out=135,in=down] (-0.15,0) to[out=up,in=225] (0.2,0.4);
            \draw[wipe] (-0.2,-0.4) to[out=45,in=down] (0.15,0) to[out=up,in=-45] (-0.2,0.4);
            \draw[->] (-0.2,-0.4) to[out=45,in=down] (0.15,0) to[out=up,in=-45] (-0.2,0.4);
        \end{tikzpicture}
        \ =\
        \begin{tikzpicture}[centerzero]
            \draw[->] (-0.2,-0.4) -- (-0.2,0.4);
            \draw[<-] (0.2,-0.4) -- (0.2,0.4);
        \end{tikzpicture}
        \ ,\quad
        \begin{tikzpicture}[centerzero]
            \draw[<-] (-0.2,-0.4) to[out=45,in=down] (0.15,0) to[out=up,in=-45] (-0.2,0.4);
            \draw[wipe] (0.2,-0.4) to[out=135,in=down] (-0.15,0) to[out=up,in=225] (0.2,0.4);
            \draw[->] (0.2,-0.4) to[out=135,in=down] (-0.15,0) to[out=up,in=225] (0.2,0.4);
        \end{tikzpicture}
        \ =\
        \begin{tikzpicture}[centerzero]
            \draw[<-] (-0.2,-0.4) -- (-0.2,0.4);
            \draw[->] (0.2,-0.4) -- (0.2,0.4);
        \end{tikzpicture}
        \ ,\quad
        \begin{tikzpicture}[centerzero]
            \draw[->] (0.4,-0.4) -- (-0.4,0.4);
            \draw[wipe] (0,-0.4) to[out=135,in=down] (-0.32,0) to[out=up,in=225] (0,0.4);
            \draw[->] (0,-0.4) to[out=135,in=down] (-0.32,0) to[out=up,in=225] (0,0.4);
            \draw[wipe] (-0.4,-0.4) -- (0.4,0.4);
            \draw[->] (-0.4,-0.4) -- (0.4,0.4);
        \end{tikzpicture}
        \ =\
        \begin{tikzpicture}[centerzero]
            \draw[->] (0.4,-0.4) -- (-0.4,0.4);
            \draw[wipe] (0,-0.4) to[out=45,in=down] (0.32,0) to[out=up,in=-45] (0,0.4);
            \draw[->] (0,-0.4) to[out=45,in=down] (0.32,0) to[out=up,in=-45] (0,0.4);
            \draw[wipe] (-0.4,-0.4) -- (0.4,0.4);
            \draw[->] (-0.4,-0.4) -- (0.4,0.4);
        \end{tikzpicture}
        \ ,
        \\ \label{skein}
        \posupcross - \negupcross
        = z\
        \begin{tikzpicture}[centerzero]
            \draw[->] (-0.15,-0.2) -- (-0.15,0.2);
            \draw[->] (0.15,-0.2) -- (0.15,0.2);
        \end{tikzpicture}
        \ ,\quad
        \\ \label{tokrel}
        \begin{tikzpicture}[centerzero]
            \draw[->] (0,-0.3) -- (0,0.3);
            \token{0,-0.1};
            \token{0,0.1};
        \end{tikzpicture}
        = -
        \begin{tikzpicture}[centerzero]
            \draw[->] (0,-0.3) -- (0,0.3);
        \end{tikzpicture}
        \ ,\quad
        \begin{tikzpicture}[centerzero]
            \draw[->] (0.3,-0.3) -- (-0.3,0.3);
            \draw[wipe] (-0.3,-0.3) -- (0.3,0.3);
            \draw[->] (-0.3,-0.3) -- (0.3,0.3);
            \token{-0.15,-0.15};
        \end{tikzpicture}
        =
        \begin{tikzpicture}[centerzero]
            \draw[->] (0.3,-0.3) -- (-0.3,0.3);
            \draw[wipe] (-0.3,-0.3) -- (0.3,0.3);
            \draw[->] (-0.3,-0.3) -- (0.3,0.3);
            \token{0.15,0.15};
        \end{tikzpicture}
        \ ,\quad
        \begin{tikzpicture}[centerzero]
            \bubright{0,0};
            \token{-0.2,0};
        \end{tikzpicture}
        = 0 =
        \begin{tikzpicture}[centerzero]
            \bubright{0,0};
        \end{tikzpicture}
        \ ,
        \\ \label{leftadj}
        \begin{tikzpicture}[centerzero]
            \draw[<-] (-0.3,0.4) -- (-0.3,0) arc(180:360:0.15) arc(180:0:0.15) -- (0.3,-0.4);
        \end{tikzpicture}
        =
        \begin{tikzpicture}[centerzero]
            \draw[->] (0,-0.4) -- (0,0.4);
        \end{tikzpicture}
        \ ,\qquad
        \begin{tikzpicture}[centerzero]
            \draw[<-] (-0.3,-0.4) -- (-0.3,0) arc(180:0:0.15) arc(180:360:0.15) -- (0.3,0.4);
        \end{tikzpicture}
        =
        \begin{tikzpicture}[centerzero]
            \draw[<-] (0,-0.4) -- (0,0.4);
        \end{tikzpicture}
        \ .
    \end{gather}
    In the above, we have used left crossings and a right cap defined by
    \begin{equation} \label{lego}
        \posleftcross :=
        \begin{tikzpicture}[centerzero]
            \draw[->] (0.4,0.3) -- (0.4,0.1) to[out=down,in=right] (0.2,-0.2) to[out=left,in=right] (-0.2,0.2) to[out=left,in=up] (-0.4,-0.1) -- (-0.4,-0.3);
            \draw[wipe] (-0.2,-0.3) \braidup (0.2,0.3);
            \draw[->] (-0.2,-0.3) \braidup (0.2,0.3);
        \end{tikzpicture}
        \ ,\qquad
        \rightcap
        :=
        \begin{tikzpicture}[anchorbase]
            \draw[->] (0.15,0) arc(0:180:0.15) to[out=-90,in=120] (0.15,-0.4);
            \draw[wipe] (-0.15,-0.4) to[out=60,in=-90] (0.15,0);
            \draw (-0.15,-0.4) to[out=60,in=-90] (0.15,0);
        \end{tikzpicture}
        \ .
    \end{equation}
    The parity of $\tokup$ is odd, and all the other generating morphisms are even.  We refer to $\tokup$ as a \emph{Clifford token}.  (Later we will refer to this as a \emph{closed Clifford token}; see \cref{AQdef}.)
\end{defin}

In addition to the left crossing and right cap defined in \cref{lego}, we define
\begin{equation} \label{wolverine}
    \negleftcross :=
    \begin{tikzpicture}[centerzero]
        \draw[->] (-0.2,-0.3) \braidup (0.2,0.3);
        \draw[wipe] (0.4,0.3) -- (0.4,0.1) to[out=down,in=right] (0.2,-0.2) to[out=left,in=right] (-0.2,0.2) to[out=left,in=up] (-0.4,-0.1) -- (-0.4,-0.3);
        \draw[->] (0.4,0.3) -- (0.4,0.1) to[out=down,in=right] (0.2,-0.2) to[out=left,in=right] (-0.2,0.2) to[out=left,in=up] (-0.4,-0.1) -- (-0.4,-0.3);
    \end{tikzpicture}
    \ ,\qquad
    \posdowncross :=
    \begin{tikzpicture}[centerzero]
        \draw[<-] (-0.2,-0.3) \braidup (0.2,0.3);
        \draw[wipe] (0.4,0.3) -- (0.4,0.1) to[out=down,in=right] (0.2,-0.2) to[out=left,in=right] (-0.2,0.2) to[out=left,in=up] (-0.4,-0.1) -- (-0.4,-0.3);
        \draw[->] (0.4,0.3) -- (0.4,0.1) to[out=down,in=right] (0.2,-0.2) to[out=left,in=right] (-0.2,0.2) to[out=left,in=up] (-0.4,-0.1) -- (-0.4,-0.3);
    \end{tikzpicture}
    \ ,\qquad
    \negdowncross :=
    \begin{tikzpicture}[centerzero]
        \draw[->] (0.4,0.3) -- (0.4,0.1) to[out=down,in=right] (0.2,-0.2) to[out=left,in=right] (-0.2,0.2) to[out=left,in=up] (-0.4,-0.1) -- (-0.4,-0.3);
        \draw[wipe] (-0.2,-0.3) \braidup (0.2,0.3);
        \draw[<-] (-0.2,-0.3) \braidup (0.2,0.3);
    \end{tikzpicture}
    \ ,\qquad
    \rightcup
    :=
    \begin{tikzpicture}[anchorbase]
        \draw (-0.15,0.4) to[out=-60,in=90] (0.15,0);
        \draw[wipe] (0.15,0) arc(360:180:0.15) to[out=90,in=240] (0.15,0.4);
        \draw[->] (0.15,0) arc(360:180:0.15) to[out=90,in=240] (0.15,0.4);
    \end{tikzpicture}
    \ .
\end{equation}
It follows that we have left and down analogues of the \emph{skein relation} \cref{skein}:
\begin{equation} \label{ldskein}
    \posleftcross - \negleftcross
    = z\
    \begin{tikzpicture}[centerzero]
        \draw[<-] (-0.15,0.25) -- (-0.15,0.2) arc(180:360:0.15) -- (0.15,0.25);
        \draw[<-] (-0.15,-0.25) -- (-0.15,-0.2) arc(180:0:0.15) -- (0.15,-0.25);
    \end{tikzpicture}
    \ ,\qquad
    \posdowncross - \negdowncross
    = z\
    \begin{tikzpicture}[centerzero]
        \draw[<-] (-0.15,-0.2) -- (-0.15,0.2);
        \draw[<-] (0.15,-0.2) -- (0.15,0.2);
    \end{tikzpicture}
    \ .
\end{equation}
We then define the other right crossing so that the right skein relation also holds:
\begin{equation} \label{rskein}
    \posrightcross
    := \negrightcross + z\
    \begin{tikzpicture}[centerzero]
        \draw[->] (-0.15,0.25) -- (-0.15,0.2) arc(180:360:0.15) -- (0.15,0.25);
        \draw[->] (-0.15,-0.25) -- (-0.15,-0.2) arc(180:0:0.15) -- (0.15,-0.25);
    \end{tikzpicture}
    \ .
\end{equation}
We call $\posupcross$, $\posrightcross$, $\posdowncross$, and $\posleftcross$ \emph{positive crossings} and we call $\negupcross$, $\negrightcross$, $\negdowncross$, and $\negleftcross$ \emph{negative crossings}.

\begin{rem} \label{HOMFLYPT}
    Given $z,t \in \kk^\times$, the \emph{HOMFLYPT skein category} $\OS(z,t)$ is the quotient of the category of framed oriented tangles by the Conway skein relation \cref{skein} and the relations
    \[
        \begin{tikzpicture}[centerzero]
            \draw[->] (0.4,0) to[out=down,in=0] (0.25,-0.15) to[out=180,in=down] (0,0.4);
            \draw[wipe] (0,-0.4) to[out=up,in=180] (0.25,0.15);
            \draw (0,-0.4) to[out=up,in=180] (0.25,0.15) to[out=0,in=up] (0.4,0);
        \end{tikzpicture}
        = t\
        \begin{tikzpicture}[centerzero]
            \draw[->] (0,-0.4) -- (0,0.4);
        \end{tikzpicture}
        \ ,\qquad
        \rightbub = \frac{t-t^{-1}}{z} 1_\one.
    \]
    This category was first introduced in \cite[\S5.2]{Tur89}, where it was called the \emph{Hecke category} (not to be confused with the more modern use of this term, which is related to the category of Soergel bimodules).  We borrow the notation $\OS(z,t)$, which comes \emph{oriented skein}, from \cite{Bru17}.  It follows from \cite[Th.~1.1]{Bru17}, which gives a presentation of $\OS(z,t)$, that all of the relations in $\OS(z,1)$ hold in $\Qcat(z)$.  More precisely, reflecting diagrams in the vertical axis and flipping crossings (i.e.\ interchanging positive and negative crossings), we see that \cref{braid}, \cref{skein}, \cref{leftadj}, and the last equality in \cref{tokrel} correspond to the relations given in \cite[Th.~1.1]{Bru17} with $t=1$.  Thus, by that result, all relations in $\OS(z,1)$ hold in $\Qcat(z)$ after reflecting in the vertical axis and flipping crossings.  But $\OS(z,1)$ is invariant under this transformation, and so all its relations hold in $\Qcat(z)$.  In fact, $\Qcat(z)$ is the strict monoidal supercategory obtained from $\OS(z,1)$ by adjoining the Clifford token, subject to the relations \cref{tokrel} involving the Clifford token.  Note that the condition $t=1$ is essentially forced by the skein relation and the last relation in \cref{tokrel}, since
    \[
        t\
        \begin{tikzpicture}[centerzero]
            \draw[->] (0,-0.4) -- (0,0.4);
        \end{tikzpicture}
        =
        \begin{tikzpicture}[centerzero]
            \draw[->] (0.4,0) to[out=down,in=0] (0.25,-0.15) to[out=180,in=down] (0,0.4);
            \draw[wipe] (0,-0.4) to[out=up,in=180] (0.25,0.15);
            \draw (0,-0.4) to[out=up,in=180] (0.25,0.15) to[out=0,in=up] (0.4,0);
        \end{tikzpicture}
        \overset{\cref{skein}}{=}
        \begin{tikzpicture}[centerzero]
            \draw (0,-0.4) to[out=up,in=180] (0.25,0.15) to[out=0,in=up] (0.4,0);
            \draw[wipe] (0.25,-0.15) to[out=180,in=down] (0,0.4);
            \draw[->] (0.4,0) to[out=down,in=0] (0.25,-0.15) to[out=180,in=down] (0,0.4);
        \end{tikzpicture}
        + z\
        \begin{tikzpicture}[centerzero]
            \draw[->] (0,-0.4) -- (0,0.4);
            \bubright{0.3,0};
        \end{tikzpicture}
        \overset{\cref{tokrel}}{=}
        \begin{tikzpicture}[centerzero]
            \draw (0,-0.4) to[out=up,in=180] (0.25,0.15) to[out=0,in=up] (0.4,0);
            \draw[wipe] (0.25,-0.15) to[out=180,in=down] (0,0.4);
            \draw[->] (0.4,0) to[out=down,in=0] (0.25,-0.15) to[out=180,in=down] (0,0.4);
        \end{tikzpicture}
        = t^{-1}\
        \begin{tikzpicture}[centerzero]
            \draw[->] (0,-0.4) -- (0,0.4);
        \end{tikzpicture}
        \ .
    \]
    Hence $t = \pm 1$.  If $t=-1$, we can rescale the crossings by $-1$ and replace $z$ by $-z$ to reduce to the case $t=1$.  This explains why the category $\Qcat(z)$ depends on only one parameter $z \in \kk$.
\end{rem}

\begin{lem} \label{venom}
    The following relations hold in $\Qcat(z)$ for all orientations of the strands:
    \begin{gather} \label{venom1}
        \begin{tikzpicture}[centerzero]
            \draw (-0.2,0.3) -- (-0.2,0.1) arc(180:360:0.2) -- (0.2,0.3);
            \draw[wipe] (-0.3,-0.3) \braidup (0,0.3);
            \draw (-0.3,-0.3) \braidup (0,0.3);
        \end{tikzpicture}
        =
        \begin{tikzpicture}[centerzero]
            \draw (-0.2,0.3) -- (-0.2,0.1) arc(180:360:0.2) -- (0.2,0.3);
            \draw[wipe] (0.3,-0.3) \braidup (0,0.3);
            \draw (0.3,-0.3) \braidup (0,0.3);
        \end{tikzpicture}
        \ ,\quad
        \begin{tikzpicture}[centerzero]
            \draw (-0.2,-0.3) -- (-0.2,-0.1) arc(180:0:0.2) -- (0.2,-0.3);
            \draw[wipe] (-0.3,0.3) \braiddown (0,-0.3);
            \draw (-0.3,0.3) \braiddown (0,-0.3);
        \end{tikzpicture}
        =
        \begin{tikzpicture}[centerzero]
            \draw (-0.2,-0.3) -- (-0.2,-0.1) arc(180:0:0.2) -- (0.2,-0.3);
            \draw[wipe] (0.3,0.3) \braiddown (0,-0.3);
            \draw (0.3,0.3) \braiddown (0,-0.3);
        \end{tikzpicture}
        \ ,\quad
        \begin{tikzpicture}[centerzero]
            \draw (-0.3,-0.3) \braidup (0,0.3);
            \draw[wipe] (-0.2,0.3) -- (-0.2,0.1) arc(180:360:0.2) -- (0.2,0.3);
            \draw (-0.2,0.3) -- (-0.2,0.1) arc(180:360:0.2) -- (0.2,0.3);
        \end{tikzpicture}
        =
        \begin{tikzpicture}[centerzero]
            \draw (0.3,-0.3) \braidup (0,0.3);
            \draw[wipe] (-0.2,0.3) -- (-0.2,0.1) arc(180:360:0.2) -- (0.2,0.3);
            \draw (-0.2,0.3) -- (-0.2,0.1) arc(180:360:0.2) -- (0.2,0.3);
        \end{tikzpicture}
        \ ,\quad
        \begin{tikzpicture}[centerzero]
            \draw (-0.3,0.3) \braiddown (0,-0.3);
            \draw[wipe] (-0.2,-0.3) -- (-0.2,-0.1) arc(180:0:0.2) -- (0.2,-0.3);
            \draw (-0.2,-0.3) -- (-0.2,-0.1) arc(180:0:0.2) -- (0.2,-0.3);
        \end{tikzpicture}
        =
        \begin{tikzpicture}[centerzero]
            \draw (0.3,0.3) \braiddown (0,-0.3);
            \draw[wipe] (-0.2,-0.3) -- (-0.2,-0.1) arc(180:0:0.2) -- (0.2,-0.3);
            \draw (-0.2,-0.3) -- (-0.2,-0.1) arc(180:0:0.2) -- (0.2,-0.3);
        \end{tikzpicture}
        \ ,\quad
        \begin{tikzpicture}[centerzero]
            \draw (-0.3,-0.4) -- (-0.3,0) arc(180:0:0.15) arc(180:360:0.15) -- (0.3,0.4);
        \end{tikzpicture}
        =
        \begin{tikzpicture}[centerzero]
            \draw (0,-0.4) -- (0,0.4);
        \end{tikzpicture}
        =
        \begin{tikzpicture}[centerzero]
            \draw (-0.3,0.4) -- (-0.3,0) arc(180:360:0.15) arc(180:0:0.15) -- (0.3,-0.4);
        \end{tikzpicture}
        \ ,
        \\ \label{venom2}
        \begin{tikzpicture}[centerzero]
            \draw (0.2,-0.4) to[out=135,in=down] (-0.15,0) to[out=up,in=225] (0.2,0.4);
            \draw[wipe] (-0.2,-0.4) to[out=45,in=down] (0.15,0) to[out=up,in=-45] (-0.2,0.4);
            \draw (-0.2,-0.4) to[out=45,in=down] (0.15,0) to[out=up,in=-45] (-0.2,0.4);
        \end{tikzpicture}
        \ =\
        \begin{tikzpicture}[centerzero]
            \draw (-0.2,-0.4) -- (-0.2,0.4);
            \draw (0.2,-0.4) -- (0.2,0.4);
        \end{tikzpicture}
        \ =\
        \begin{tikzpicture}[centerzero]
            \draw (-0.2,-0.4) to[out=45,in=down] (0.15,0) to[out=up,in=-45] (-0.2,0.4);
            \draw[wipe] (0.2,-0.4) to[out=135,in=down] (-0.15,0) to[out=up,in=225] (0.2,0.4);
            \draw (0.2,-0.4) to[out=135,in=down] (-0.15,0) to[out=up,in=225] (0.2,0.4);
        \end{tikzpicture}
        \ ,\quad
        \begin{tikzpicture}[centerzero]
            \draw (0.4,-0.4) -- (-0.4,0.4);
            \draw[wipe] (0,-0.4) to[out=135,in=down] (-0.32,0) to[out=up,in=225] (0,0.4);
            \draw (0,-0.4) to[out=135,in=down] (-0.32,0) to[out=up,in=225] (0,0.4);
            \draw[wipe] (-0.4,-0.4) -- (0.4,0.4);
            \draw (-0.4,-0.4) -- (0.4,0.4);
        \end{tikzpicture}
        \ =\
        \begin{tikzpicture}[centerzero]
            \draw (0.4,-0.4) -- (-0.4,0.4);
            \draw[wipe] (0,-0.4) to[out=45,in=down] (0.32,0) to[out=up,in=-45] (0,0.4);
            \draw (0,-0.4) to[out=45,in=down] (0.32,0) to[out=up,in=-45] (0,0.4);
            \draw[wipe] (-0.4,-0.4) -- (0.4,0.4);
            \draw (-0.4,-0.4) -- (0.4,0.4);
        \end{tikzpicture}
        \ ,\quad
        \\ \label{venom3}
        \begin{tikzpicture}[anchorbase]
            \draw (0.15,0) arc(360:180:0.15) to[out=90,in=240] (0.15,0.4);
            \draw[wipe] (-0.15,0.4) to[out=-60,in=90] (0.15,0);
            \draw (-0.15,0.4) to[out=-60,in=90] (0.15,0);
        \end{tikzpicture}
        =
        \begin{tikzpicture}[anchorbase]
            \draw (-0.15,0.15) -- (-0.15,0) arc(180:360:0.15) -- (0.15,0.15);
        \end{tikzpicture}
        \ ,\qquad
        \begin{tikzpicture}[anchorbase]
            \draw (-0.15,-0.4) to[out=60,in=-90] (0.15,0);
            \draw[wipe] (0.15,0) arc(0:180:0.15) to[out=-90,in=120] (0.15,-0.4);
            \draw (0.15,0) arc(0:180:0.15) to[out=-90,in=120] (0.15,-0.4);
        \end{tikzpicture}
        =
        \begin{tikzpicture}[anchorbase]
            \draw (-0.15,-0.15) -- (-0.15,0) arc(180:0:0.15) -- (0.15,-0.15);
        \end{tikzpicture}
        \ ,\qquad
        \unbub = 0.
    \end{gather}
\end{lem}

\begin{proof}
    This follows from \cref{HOMFLYPT}, since all these relations holds in $\OS(z,1)$.
\end{proof}

We define
\begin{equation} \label{tokdown}
    \begin{tikzpicture}[centerzero]
        \draw[<-] (0,-0.4) -- (0,0.4);
        \token{0,0};
    \end{tikzpicture}
    :=
    \begin{tikzpicture}[centerzero]
        \draw[->] (0.3,0.4) -- (0.3,0) arc(360:180:0.15) arc(0:180:0.15) -- (-0.3,-0.4);
        \token{0,0};
    \end{tikzpicture}
    \ .
\end{equation}
It follows that
\begin{equation} \label{torch}
    \begin{tikzpicture}[centerzero]
        \draw[<-] (0,-0.3) -- (0,0.3);
        \token{0,-0.1};
        \token{0,0.1};
    \end{tikzpicture}
    =
    \begin{tikzpicture}[anchorbase]
        \draw[->] (0.9,0.7) -- (0.9,0.3) arc (360:180:0.15) arc(0:180:0.15) -- (0.3,0) arc(360:180:0.15) arc(0:180:0.15) -- (-0.3,-0.4);
        \token{0,0};
        \token{0.6,0.3};
    \end{tikzpicture}
    \overset{\cref{intlaw}}{=} -
    \begin{tikzpicture}[anchorbase]
        \draw[->] (0.9,0.7) -- (0.9,-0.1) arc (360:180:0.15) -- (0.6,0.1) arc(0:180:0.15) arc(360:180:0.15) -- (0,0.4) arc(0:180:0.15) -- (-0.3,-0.4);
        \token{0.6,-0.1};
        \token{0,0.4};
    \end{tikzpicture}
    \overset{\cref{leftadj}}{=} -
    \begin{tikzpicture}[anchorbase]
        \draw[->] (0.6,0.55) -- (0.6,-0.2) arc (360:180:0.15) -- (0.3,0.2) arc(0:180:0.15) -- (0,-0.55);
        \token{0.3,-0.1};
        \token{0.3,0.1};
    \end{tikzpicture}
    \overset{\cref{tokrel}}{=}
    \begin{tikzpicture}[anchorbase]
        \draw[->] (0.6,0.55) -- (0.6,-0.2) arc (360:180:0.15) -- (0.3,0.2) arc(0:180:0.15) -- (0,-0.55);
    \end{tikzpicture}
    \overset{\cref{leftadj}}{=}
    \begin{tikzpicture}[centerzero]
        \draw[<-] (0,-0.3) -- (0,0.3);
    \end{tikzpicture}
    \ .
\end{equation}

\begin{lem} \label{ironfull}
    The following relations hold in $\Qcat(z)$ for all orientations of the strands:
    \begin{equation} \label{iron}
        \begin{tikzpicture}[centerzero]
            \draw (0.3,-0.3) -- (-0.3,0.3);
            \draw[wipe] (-0.3,-0.3) -- (0.3,0.3);
            \draw (-0.3,-0.3) -- (0.3,0.3);
            \token{-0.18,-0.18};
        \end{tikzpicture}
        =
        \begin{tikzpicture}[centerzero]
            \draw (0.3,-0.3) -- (-0.3,0.3);
            \draw[wipe] (-0.3,-0.3) -- (0.3,0.3);
            \draw (-0.3,-0.3) -- (0.3,0.3);
            \token{0.18,0.18};
        \end{tikzpicture}
        \ ,\quad
        \begin{tikzpicture}[centerzero]
            \draw (-0.3,-0.3) -- (0.3,0.3);
            \draw[wipe] (0.3,-0.3) -- (-0.3,0.3);
            \draw (0.3,-0.3) -- (-0.3,0.3);
            \token{-0.18,0.18};
        \end{tikzpicture}
        =
        \begin{tikzpicture}[centerzero]
            \draw (-0.3,-0.3) -- (0.3,0.3);
            \draw[wipe] (0.3,-0.3) -- (-0.3,0.3);
            \draw (0.3,-0.3) -- (-0.3,0.3);
            \token{0.18,-0.18};
        \end{tikzpicture}
        \ ,\quad
        \begin{tikzpicture}[anchorbase]
            \draw (-0.2,0.2) -- (-0.2,0) arc (180:360:0.2) -- (0.2,0.2);
            \token{-0.2,0};
        \end{tikzpicture}
        \ =\
        \begin{tikzpicture}[anchorbase]
            \draw (-0.2,0.2) -- (-0.2,0) arc (180:360:0.2) -- (0.2,0.2);
            \token{0.2,0};
        \end{tikzpicture}
        \ ,\quad
        \begin{tikzpicture}[anchorbase]
            \draw (-0.2,-0.2) -- (-0.2,0) arc (180:0:0.2) -- (0.2,-0.2);
            \token{-0.2,0};
        \end{tikzpicture}
        \ =\
        \begin{tikzpicture}[anchorbase]
            \draw (-0.2,-0.2) -- (-0.2,0) arc (180:0:0.2) -- (0.2,-0.2);
            \token{0.2,0};
        \end{tikzpicture}
        \ ,\quad
        \begin{tikzpicture}[centerzero]
            \bubun{0,0};
            \token{-0.2,0};
        \end{tikzpicture}
        = 0.
    \end{equation}
\end{lem}

\begin{proof}
    Composing the second relation in \cref{tokrel} on the top and bottom with $\negupcross$, we see that the first two relations in \cref{iron} hold when both strands are oriented up.  Attaching a left cup to the bottom of \cref{tokdown} and using \cref{leftadj}, we see that the third relation in \cref{iron} holds for the strand oriented to the left.  Similarly, attaching a left cap to the top of \cref{tokdown}, we see that the fourth relation in \cref{iron} also holds for the strand oriented to the left.  Then, using the definitions \cref{lego,wolverine} of the left and down crossings, we see that the first two relations in \cref{iron} hold for the strands oriented to the left or oriented down.  Next, taking the second relation in \cref{iron} for the strands oriented to the left, and composing on the top and bottom with $\negrightcross$, we see that the first relation in \cref{iron} holds for the strands oriented to the right.  Similarly, taking the first relation in \cref{iron} for the strands oriented to the left, and composing on the top and bottom with $\posrightcross$,  we see that the second relation in \cref{iron} holds for the strands oriented to the right.

    So we have now proved the first two relations in \cref{iron} for all orientations of the strands, and the third and fourth relations for the strands oriented to the left.  Next we compute
    \[
        \begin{tikzpicture}[centerzero]
            \bubleft{0,0};
            \token{0.2,0};
        \end{tikzpicture}
        \overset{\cref{wolverine}}{=}
        \begin{tikzpicture}[centerzero]
            \draw (-0.2,-0.2) arc(180:360:0.2) \braidup (-0.2,0.2) arc(180:90:0.2);
            \draw[wipe] (0,0.4) arc(90:0:0.2) \braiddown (-0.2,-0.2);
            \draw[<-] (0,0.4) arc(90:0:0.2) \braiddown (-0.2,-0.2);
            \token{0.2,0.2};
        \end{tikzpicture}
        =
        \begin{tikzpicture}[centerzero]
            \draw (-0.2,-0.2) arc(180:360:0.2) \braidup (-0.2,0.2) arc(180:90:0.2);
            \draw[wipe] (0,0.4) arc(90:0:0.2) \braiddown (-0.2,-0.2);
            \draw[<-] (0,0.4) arc(90:0:0.2) \braiddown (-0.2,-0.2);
            \token{-0.2,-0.2};
        \end{tikzpicture}
        \overset{\cref{lego}}{=}
        \begin{tikzpicture}[centerzero]
            \bubright{0,0};
            \token{-0.2,0};
        \end{tikzpicture}
        \overset{\cref{tokrel}}{=}
        0.
    \]
    So the last equality in \cref{iron} holds for both orientations of the strand.  We also have
    \[
        \begin{tikzpicture}[anchorbase]
            \draw[->] (-0.2,0.2) -- (-0.2,0) arc (180:360:0.2) -- (0.2,0.2);
            \token{0.2,0};
        \end{tikzpicture}
        \overset{\cref{wolverine}}{=}
        \begin{tikzpicture}[anchorbase]
            \draw (-0.15,0.6) -- (-0.15,0.4) to[out=-90,in=90] (0.15,0);
            \draw[wipe] (-0.15,0) to[out=90,in=-90] (0.15,0.4);
            \draw[->] (0.15,0) arc(360:180:0.15) to[out=90,in=-90] (0.15,0.4) -- (0.15,0.6);
            \token{0.15,0.4};
        \end{tikzpicture}
        \overset{\cref{tokrel}}{=}
        \begin{tikzpicture}[anchorbase]
            \draw (-0.15,0.6) -- (-0.15,0.4) to[out=-90,in=90] (0.15,0);
            \draw[wipe] (-0.15,0) to[out=90,in=-90] (0.15,0.4);
            \draw[->] (0.15,0) arc(360:180:0.15) to[out=90,in=-90] (0.15,0.4) -- (0.15,0.6);
            \token{-0.15,0};
        \end{tikzpicture}
        \overset{\cref{rskein}}{\underset{\cref{tokrel}}{=}}
        \begin{tikzpicture}[anchorbase]
            \draw[->] (0.15,0) arc(360:180:0.15) to[out=90,in=-90] (0.15,0.4) -- (0.15,0.6);
            \draw[wipe] (-0.15,0.6) -- (-0.15,0.4) to[out=-90,in=90] (0.15,0);
            \draw (-0.15,0.6) -- (-0.15,0.4) to[out=-90,in=90] (0.15,0);
            \token{-0.15,0};
        \end{tikzpicture}
        =
        \begin{tikzpicture}[anchorbase]
            \draw[->] (0.15,0) arc(360:180:0.15) to[out=90,in=-90] (0.15,0.4) -- (0.15,0.6);
            \draw[wipe] (-0.15,0.6) -- (-0.15,0.4) to[out=-90,in=90] (0.15,0);
            \draw (-0.15,0.6) -- (-0.15,0.4) to[out=-90,in=90] (0.15,0);
            \token{0.15,0};
        \end{tikzpicture}
        =
        \begin{tikzpicture}[anchorbase]
            \draw[->] (0.15,0) arc(360:180:0.15) to[out=90,in=-90] (0.15,0.4) -- (0.15,0.6);
            \draw[wipe] (-0.15,0.6) -- (-0.15,0.4) to[out=-90,in=90] (0.15,0);
            \draw (-0.15,0.6) -- (-0.15,0.4) to[out=-90,in=90] (0.15,0);
            \token{-0.15,0.4};
        \end{tikzpicture}
        \overset{\cref{rskein}}{\underset{\cref{tokrel}}{=}}
        \begin{tikzpicture}[anchorbase]
            \draw (-0.15,0.6) -- (-0.15,0.4) to[out=-90,in=90] (0.15,0);
            \draw[wipe] (-0.15,0) to[out=90,in=-90] (0.15,0.4);
            \draw[->] (0.15,0) arc(360:180:0.15) to[out=90,in=-90] (0.15,0.4) -- (0.15,0.6);
            \token{-0.15,0.4};
        \end{tikzpicture}
        \overset{\cref{wolverine}}{=}
        \begin{tikzpicture}[anchorbase]
            \draw[->] (-0.2,0.2) -- (-0.2,0) arc (180:360:0.2) -- (0.2,0.2);
            \token{-0.2,0};
        \end{tikzpicture}
        \ .
    \]
    An analogous argument shows that the fourth relation in \cref{iron} holds for the strands oriented to the right.
\end{proof}

It follows from \cref{iron} that Clifford tokens slide \emph{over} all crossings.  However, they do not slide \emph{under} crossings.  In fact, we have the following result.

\begin{lem} \label{nickelfull}
    The following relations hold in $\Qcat(z)$:
    \begin{equation} \label{nickel}
        \begin{tikzpicture}[centerzero]
            \draw[->] (-0.3,-0.3) -- (0.3,0.3);
            \draw[wipe] (0.3,-0.3) -- (-0.3,0.3);
            \draw[->] (0.3,-0.3) -- (-0.3,0.3);
            \token{-0.18,-0.18};
        \end{tikzpicture}
        =
        \begin{tikzpicture}[centerzero]
            \draw[->] (-0.3,-0.3) -- (0.3,0.3);
            \draw[wipe] (0.3,-0.3) -- (-0.3,0.3);
            \draw[->] (0.3,-0.3) -- (-0.3,0.3);
            \token{0.15,0.15};
        \end{tikzpicture}
        + z\
        \left(
            \begin{tikzpicture}[centerzero]
                \draw[->] (-0.2,-0.3) -- (-0.2,0.3);
                \draw[->] (0.2,-0.3) -- (0.2,0.3);
                \token{0.2,0};
            \end{tikzpicture}
            -
            \begin{tikzpicture}[centerzero]
                \draw[->] (-0.2,-0.3) -- (-0.2,0.3);
                \draw[->] (0.2,-0.3) -- (0.2,0.3);
                \token{-0.2,0};
            \end{tikzpicture}
        \right),
        \qquad
        \begin{tikzpicture}[centerzero]
            \draw[->] (0.3,-0.3) -- (-0.3,0.3);
            \draw[wipe] (-0.3,-0.3) -- (0.3,0.3);
            \draw[->] (-0.3,-0.3) -- (0.3,0.3);
            \token{-0.15,0.15};
        \end{tikzpicture}
        =
        \begin{tikzpicture}[centerzero]
            \draw[->] (0.3,-0.3) -- (-0.3,0.3);
            \draw[wipe] (-0.3,-0.3) -- (0.3,0.3);
            \draw[->] (-0.3,-0.3) -- (0.3,0.3);
            \token{0.17,-0.17};
        \end{tikzpicture}
        + z\
        \left(
            \begin{tikzpicture}[centerzero]
                \draw[->] (-0.2,-0.3) -- (-0.2,0.3);
                \draw[->] (0.2,-0.3) -- (0.2,0.3);
                \token{-0.2,0};
            \end{tikzpicture}
            -
            \begin{tikzpicture}[centerzero]
                \draw[->] (-0.2,-0.3) -- (-0.2,0.3);
                \draw[->] (0.2,-0.3) -- (0.2,0.3);
                \token{0.2,0};
            \end{tikzpicture}
        \right).
    \end{equation}
\end{lem}

\begin{proof}
    We have
    \begin{equation}
        \begin{tikzpicture}[centerzero]
            \draw[->] (-0.3,-0.3) -- (0.3,0.3);
            \draw[wipe] (0.3,-0.3) -- (-0.3,0.3);
            \draw[->] (0.3,-0.3) -- (-0.3,0.3);
            \token{-0.18,-0.18};
        \end{tikzpicture}
        \overset{\cref{skein}}{=}
        \begin{tikzpicture}[centerzero]
            \draw[->] (0.3,-0.3) -- (-0.3,0.3);
            \draw[wipe] (-0.3,-0.3) -- (0.3,0.3);
            \draw[->] (-0.3,-0.3) -- (0.3,0.3);
            \token{-0.15,-0.15};
        \end{tikzpicture}
        - z\
        \begin{tikzpicture}[centerzero]
            \draw[->] (-0.2,-0.3) -- (-0.2,0.3);
            \draw[->] (0.2,-0.3) -- (0.2,0.3);
            \token{-0.2,0};
        \end{tikzpicture}
        \overset{\cref{tokrel}}{=}
        \begin{tikzpicture}[centerzero]
            \draw[->] (0.3,-0.3) -- (-0.3,0.3);
            \draw[wipe] (-0.3,-0.3) -- (0.3,0.3);
            \draw[->] (-0.3,-0.3) -- (0.3,0.3);
            \token{0.15,0.15};
        \end{tikzpicture}
        - z\
        \begin{tikzpicture}[centerzero]
            \draw[->] (-0.2,-0.3) -- (-0.2,0.3);
            \draw[->] (0.2,-0.3) -- (0.2,0.3);
            \token{-0.2,0};
        \end{tikzpicture}
        \overset{\cref{skein}}{=}
        \begin{tikzpicture}[centerzero]
            \draw[->] (-0.3,-0.3) -- (0.3,0.3);
            \draw[wipe] (0.3,-0.3) -- (-0.3,0.3);
            \draw[->] (0.3,-0.3) -- (-0.3,0.3);
            \token{0.15,0.15};
        \end{tikzpicture}
        + z\
        \left(
            \begin{tikzpicture}[centerzero]
                \draw[->] (-0.2,-0.3) -- (-0.2,0.3);
                \draw[->] (0.2,-0.3) -- (0.2,0.3);
                \token{0.2,0};
            \end{tikzpicture}
            -
            \begin{tikzpicture}[centerzero]
                \draw[->] (-0.2,-0.3) -- (-0.2,0.3);
                \draw[->] (0.2,-0.3) -- (0.2,0.3);
                \token{-0.2,0};
            \end{tikzpicture}
        \right).
    \end{equation}
    The proof of the second relation is analogous.
\end{proof}

We now describe several symmetries of the category $\Qcat(z)$.  First note that we have an isomorphism of monoidal supercategories
\[
    \Omega_- \colon \Qcat(z) \xrightarrow{\cong} \Qcat(-z)
\]
that is the identity objects and, on morphisms, multiplies all crossings by $-1$.

\begin{prop} \label{storm}
    There is a unique isomorphism of monoidal supercategories
    \[
        \Omega_\updownarrow \colon \Qcat(z) \to \Qcat(z)^\op
    \]
    determined on objects by $\uparrow\ \mapsto\ \downarrow$, $\downarrow\ \mapsto\ \uparrow$ and sending
    \[
        \tokup \mapsto \tokdown,\quad
        \posupcross \mapsto - \negdowncross,\quad
        \leftcap \mapsto \leftcup,\quad
        \leftcup \mapsto \leftcap.
    \]
    The superfunctor $\Omega_{\updownarrow}$ acts on the other crossings, cups, caps, and Clifford tokens as follows:
    \begin{gather*}
        \tokdown \mapsto \tokup,\quad
        \negupcross \mapsto - \posdowncross,\quad
        \posrightcross \mapsto - \negrightcross,\quad
        \negrightcross \mapsto - \posrightcross,\quad
        \posdowncross \mapsto - \negupcross,\quad
        \negdowncross \mapsto - \posupcross,\quad
        \\
        \negleftcross \mapsto - \posleftcross,\quad
        \posleftcross \mapsto - \negleftcross,\quad
        \rightcap \mapsto -\, \rightcup\, ,\quad
        \rightcup \mapsto -\, \rightcap\, .
    \end{gather*}
\end{prop}

\begin{proof}
    This follows from \cref{ldskein,venom1,venom2,venom3,torch,iron}.
\end{proof}

\begin{prop} \label{thor}
    There is a unique isomorphism of monoidal supercategories
    \[
        \Omega_\leftrightarrow \colon \Qcat(z) \to \Qcat(z)^\rev
    \]
    determined on objects by $\uparrow\ \mapsto\ \uparrow$, $\downarrow\ \mapsto\ \downarrow$ and sending
    \[
        \tokup \mapsto \tokup,\quad
        \posupcross \mapsto - \negupcross,\quad
        \leftcup \mapsto \rightcup,\quad
        \leftcap \mapsto \rightcap.
    \]
    The superfunctor $\Omega_\leftrightarrow$ acts on the other crossings, cups, caps, and Clifford tokens as follows:
    \begin{gather*}
        \tokdown \mapsto \tokdown,\quad
        \negupcross \mapsto - \posupcross,\quad
        \posrightcross \mapsto - \negleftcross,\quad
        \negrightcross \mapsto - \posleftcross,\quad
        \posdowncross \mapsto - \negdowncross,\quad
        \negdowncross \mapsto - \posdowncross,\quad
        \\
        \negleftcross \mapsto - \posrightcross,\quad
        \posleftcross \mapsto - \negrightcross,\quad
        \rightcup \mapsto - \leftcup\, ,\quad
        \rightcap \mapsto - \leftcap\, .
    \end{gather*}
\end{prop}

\begin{proof}
    This follows from \cref{venom1,venom3,iron}.
\end{proof}

\begin{rem} \label{numbering}
    In many instances, when we wish to number strands in diagrams, it is most natural to number them from right to left.  For instance, we will do so when discussing Jucys--Murphy elements in \cref{sec:affineEnd}.  However, at other times, when we want to discuss relationships to superalgebras appearing in the literature, it is useful to number strands from left to right to better match conventions in other papers.  The isomorphism $\Omega_\leftrightarrow$ allows us to move back and forth between these two conventions.
\end{rem}

It follows from \cref{storm,thor} that $\Qcat(z)$ is strictly pivotal, with duality superfunctor
\begin{equation} \label{swivel}
    \Omega_\leftrightarrow \circ \Omega_\updownarrow \colon \Qcat(z) \xrightarrow{\cong} (\Qcat(z)^\op)^\rev
\end{equation}
defined by rotating diagrams through $180\degree$ and multiplying by $(-1)^{\binom{y}{2}}$, where $y$ is the number of Clifford tokens in the diagram.  Intuitively, this means that morphisms are invariant under isotopy fixing the endpoints, multiplying by the appropriate sign when odd elements change height.  Thus, for example, we have rightward, leftward, and downward versions of the relations \cref{nickel}.

\begin{lem} \label{degenerate}
    When $z = 0$, reversing orientation of strands gives an isomorphism of monoidal supercategories from $\Qcat(0)$ to the oriented Brauer--Clifford supercategory of \cite[Def.~3.2]{BCK19}.
\end{lem}

\begin{proof}
    When $z=0$, \cref{skein} implies that
    \[
        \symcross := \posupcross = \negupcross.
    \]
    It is then straightforward to verify that the relations of \cref{Qdef}, without the last relation in \cref{tokrel}, become the relations in \cite[Def.~3.2]{BCK19} with the orientations of strands reversed.  The last relation in \cref{tokrel} also holds in the oriented Brauer--Clifford supercategory by \cite[(3.16)]{BCK19}.
\end{proof}

\begin{rem} \label{imagine}
    The reason we need to reverse orientation in \cref{degenerate} is that \cite[Def.~3.2]{BCK19} includes the relation
    \[
        \begin{tikzpicture}[centerzero]
            \draw[->] (0,-0.3) -- (0,0.3);
            \token{0,-0.1};
            \token{0,0.1};
        \end{tikzpicture}
        =
        \begin{tikzpicture}[centerzero]
            \draw[->] (0,-0.3) -- (0,0.3);
        \end{tikzpicture}
        \ ,
    \]
    which matches the sign in \cref{torch}, but not in the first relation in \cref{tokrel}.  If $\sqrt{-1} \in \kk$, then we have an automorphism of $\Qcat(z)$ that reverses orientation of strands and multiplies Clifford tokens by $\sqrt{-1}$.  In this case, there is an isomorphism from $\Qcat(0)$ to the oriented Brauer--Clifford category that multiplies Clifford tokens by $\sqrt{-1}$, with no need to reverse orientation.
\end{rem}

Let $X = X_1 \otimes \dotsb \otimes X_r$ and $Y = Y_1 \otimes \dotsb \otimes Y_s$ be objects of $\Qcat(z)$ for $X_i,Y_j \in \{\uparrow, \downarrow\}$.  An \emph{$(X,Y)$-matching} is a bijection between the sets
\begin{equation} \label{chain}
    \{i : X_i =\, \uparrow\} \sqcup \{j : Y_j =\, \downarrow\}
    \quad \text{and} \quad
    \{i : X_i =\, \downarrow\} \sqcup \{j : Y_j =\, \uparrow\}.
\end{equation}
A \emph{positive reduced lift} of an $(X,Y)$-matching is a string diagram representing a morphism $X \to Y$ such that
\begin{itemize}
    \item the endpoints of each string are points that correspond under the given matching;
    \item there are no Clifford tokens on any string and no closed strings (i.e.\ strings with no endpoints);
    \item there are no self-intersections of strings and no two strings cross each other more than once;
    \item all crossings are positive.
\end{itemize}
It follows from \cref{venom2} that any two positive reduced lifts of a given $(X,Y)$-matching are equal as morphisms in $\Qcat(z)$.

For each $(X,Y)$, fix a set $B(X,Y)$ consisting of a choice of positive reduced lift for each $(X,Y)$-matching.  Then let $B_\bullet(X,Y)$ denote the set of all morphisms that can be obtained from elements of $B(X,Y)$ by adding at most one (and possibly zero) Clifford token near the terminus of each string.  We require that all Clifford tokens occurring on strands whose terminus is at the top of the diagram to be at the same height; similarly we require that all Clifford tokens occurring on strands whose terminus is at the bottom of the diagram to be at the same height, and below those Clifford tokens on strands whose terminus is at the top of the diagram.

\begin{prop} \label{nonaffinespan}
    For any objects $X,Y$ of $\Qcat(z)$, the set $B_\bullet(X,Y)$ spans the $\kk$-supermodule\linebreak $\Hom_{\Qcat(z)}(X,Y)$ over $\kk$.
\end{prop}

\begin{proof}
    Let $X$ and $Y$ be two objects of $\Qcat(z)$.  Using \cref{iron,nickel}, Clifford tokens can be moved near the termini of strings.  Next, using \cref{tokrel,torch}, we can reduce the number of Clifford tokens to at most one on each string.  Then, since all the relations in the HOMFLYPT skein category hold (see \cref{HOMFLYPT}), we have a straightening algorithm to rewrite any diagram representing a morphism $X \to Y$ as a $\kk$-linear combination of the ones in $B_\bullet(X,Y)$.  Here we also use \cref{venom3,iron} to see that any string diagram with a closed component is equal to zero.
\end{proof}

We will prove later, in \cref{basisthm}, that the sets $B_\bullet(X,Y)$ are actually \emph{bases} of the morphism spaces.

\begin{defin}[{\cite[Def.~3.4]{BGJKW16}}] \label{hawkeye}
    For $r,s \in \Z_{> 0}$ and $z \in \kk$, the \emph{quantum walled Brauer--Clifford superalgebra} $\BC_{r,s}(z)$ is the associative superalgebra generated by
    \[
        \text{even elements } t_1,\dotsc,t_{r-1}, t_1^*,\dotsc,t_{s-1}^*, e
        \text{ and odd elements } \cg_1,\dotsc,\cg_r,\cg_1^*,\dotsc \cg_s^*
    \]
    satisfying the following relations (for $i,j$ in the allowable range)
    \begin{align*}
        t_i^2 &= z t_i + 1, & (t_i^*)^2 &= z t_i^* + 1, \\
        t_i t_{i+1} t_i &= t_{i+1} t_i t_{i+1}, & t_i^* t_{i+1}^* t_i^* &= t_{i+1}^* t_i^* t_{i+1}^*, \\
        t_i t_j &= t_j t_i \text{ for } |i-j| > 1, & t_i^* t_j^* &= t_j^* t_i^* \text{ for } |i-j| > 1, \\
        \cg_i^2 = -1,&\ \cg_i \cg_j = - \cg_j \cg_i \text{ for } i \ne j, & (\cg_i^*)^2 = 1,&\ \cg_i^* \cg_j^* = - \cg_j^* \cg_i^* \text{ for } i \ne j, \\
        t_i \cg_i &= \cg_{i+1} t_i, & t_i^* \cg_i^* &= \cg_{i+1}^* t_i^*, \\
        t_i \cg_j &= \cg_j t_i \text{ for } j \ne i,i+1, & t_i^* \cg_j^* &= \cg_j^* t_i^* \text{ for } j \ne i,i+1, \\
        e^2 = 0,\ et_{r-1}e = e,&\ et_j = t_j e \text{ for } j \ne r-1, & e t_1^* e = e,&\ e t_j^* = t_j^* e \text{ for } j \ne 1, \\
        t_i \cg_j^* = \cg_j^* t_i,&\ \cg_r e = \cg_1^* e,& t_i^* \cg_j = \cg_j t_i^*,&\ e \cg_r = e \cg_1^*,\\
        \cg_j e &= e\cg_j \text{ for } j \ne r,& \cg_j^* e &= e \cg_j^* \text{ for } j \ne 1, \\
        e t_{r-1}^{-1} t_1^* e t_1^* t_{r-1}^{-1} &= t_{r-1}^{-1} t_1^* e t_1^* t_{r-1}^{-1} e,&
        e \cg_r e &= 0.
    \end{align*}
    We define $\BC_{r,0}(z)$ to be the associative superalgebra generated by even elements $t_1,\dotsc,t_{r-1}$ and odd elements $\cg_1,\dots,\cg_r$ subject to the above relations involving only these elements.  We define $\BC_{0,s}(z)$ similarly.  Finally, we define $\BC_{0,0}(z) = \kk$.
\end{defin}

The relations in the first line in \cref{hawkeye} imply that $t_i$ and $t_i^*$ are invertible, with $t_i^{-1} = t_i - z$ and $(t_i^*)^{-1} = t_i^* - z$.  Then, multiplying both sides of the relation $t_i \cg_i = \cg_{i+1} t_i$ on the left and right by $t_i^{-1}$ gives the relation
\begin{equation} \label{upside}
    \cg_i t_i = t_i \cg_{i+1} + z(\cg_i-\cg_{i+1}).
\end{equation}
A straightforward computation shows that we have an isomorphism of superalgebras
\begin{equation} \label{wheel}
    \BC_{r,s}(z) \xrightarrow{\cong} \BC_{s,r}(z)^\op,\qquad
    t_i \mapsto t^*_{r-i},\
    t_i^* \mapsto t_{s-i},\
    \cg_i \mapsto \cg_{r+1-i},\
    \cg_i^* \mapsto \cg^*_{s+1-i},\
    e \mapsto e.
\end{equation}
We will soon see a diagrammatic interpretation of this isomorphism.

The superalgebra
\[
    \HC_r(z) := \BC_{r,0}(z)
\]
is the \emph{Hecke--Clifford superalgebra}, which first appeared in \cite[Def.~5.1]{Ols92}.  It follows from \cref{wheel} that we have an isomorphism of superalgebras $\BC_{0,s}(z) \cong \HC_s(z)^\op$.

\begin{prop} \label{walled}
    For $r,s \in \N$, we have a surjective homomorphism of associative superalgebras
    \[
        \BC_{r,s}(z) \twoheadrightarrow \End_{\Qcat(z)}(\uparrow^{\otimes r} \otimes \downarrow^{\otimes s})
    \]
    given by
    \begin{align*}
        t_i &\mapsto\, \uparrow^{\otimes (i-1)} \otimes \posupcross \otimes \uparrow^{\otimes (r-i-1)} \otimes \downarrow^{\otimes s},& 1 \le i \le r-1,
        \\
        t_i^* &\mapsto\, \uparrow^{\otimes r} \otimes \downarrow^{\otimes (i-1)} \otimes \posdowncross \otimes \downarrow^{\otimes (s-i-1)},& 1 \le i \le s-1,
        \\
        e &\mapsto\, \uparrow^{\otimes (r-1)} \otimes
        \begin{tikzpicture}[centerzero]
            \draw[->] (-0.15,-0.25) -- (-0.15,-0.2) arc(180:0:0.15) -- (0.15,-0.25);
            \draw[<-] (-0.15,0.25) -- (-0.15,0.2) arc(180:360:0.15) -- (0.15,0.25);
        \end{tikzpicture}
        \otimes \downarrow^{\otimes (s-1)},& \text{ if } r,s>0,
        \\
        \cg_i &\mapsto\, \uparrow^{\otimes (i-1)} \otimes\, \tokup\, \otimes \uparrow^{\otimes (r-i)} \otimes \downarrow^{\otimes s},& 1 \le i \le r,
        \\
        \cg_i^* &\mapsto\, \uparrow^{\otimes r} \otimes \downarrow^{\otimes (i-1)} \otimes\, \tokdown\, \otimes \downarrow^{\otimes (s-i)},& 1 \le i \le s.
    \end{align*}
\end{prop}

\begin{proof}
    It is a straightforward computation to verify that the given map is well-defined, i.e.\ that it respects the relations in \cref{hawkeye}.  Since all elements of $B_\bullet(\uparrow^{\otimes r} \otimes \downarrow^{\otimes s}, \uparrow^{\otimes r} \otimes \downarrow^{\otimes s})$ can clearly be written as compositions of the given images of the generators of $\BC_{r,s}(z)$, it follows from \cref{nonaffinespan} that the map is also surjective.
\end{proof}

We will show in \cref{walleder} that the homomorphism of \cref{walled} is actually an isomorphism.

\section{The quantum isomeric superalgebra}

In this section we recall the definition of the quantum isomeric superalgebra and prove some results about it that will be used in the sequel.  (Recall, as mentioned in the introduction, that this superalgebra is traditionally called the quantum queer superalgebra.)  Throughout this section we work over the field $\kk = \C(q)$ and we set $z := q - q^{-1}$.  To simplify the expressions to follow, we first introduce some notation and conventions.  Fix an index set
\[
    \tI := \{1,2,\dotsc,n,-1,-2,\dotsc,-n\}.
\]
We will use $a,b,c,d$ to denote elements of $\{1,2,\dotsc,n\}$ and $i,j,k,l$ to denote elements of $\tI$.  For $i,j \in \tI$, we define
\begin{align} \label{signage1}
    p(i) &:=
    \begin{cases}
        0 & \text{if } i > 0, \\
        1 & \text{if } i < 0,
    \end{cases}
    & p(i,j) &:= p(i)+p(j),
    \\ \label{signage2}
    \quad \sgn(i) &:= (-1)^{p(i)} = 1-2p(i),&
    \quad \varphi(i,j) &:= \delta_{|i|,|j|} \sgn(j).
\end{align}
If $C$ is some condition, we define $\delta_C = 1$ if the condition is satisfied, and $\delta_C = 0$ otherwise.  Then, for $i,j \in \tI$, $\delta_{ij} := \delta_{i=j}$ is the usual Kronecker delta.

Let $V$ denote the $\kk$-supermodule with basis $v_i$, $i \in \tI$, where the parity of $v_i$ is given by
\[
    \overline{v_i} = p(i).
\]
Using this basis, we will identify $V$ with $\kk^{n|n}$ as $\kk$-supermodules, and $\End_\kk(V)$ with $\Mat_{n|n}(\kk)$ as associative superalgebras.  Let $E_{ij} \in \Mat_{n|n}(\kk)$ denote the matrix with a $1$ in the $(i,j)$-position and a $0$ in all other positions.  Then the parity of $E_{ij}$ is $p(i,j)$.  The general linear Lie superalgebra $\gl_{n|n}$ is equal to $\End_\kk(V)$ as a $\kk$-supermodule, with bracket given by the supercommutator
\[
    [X,Y] = XY - (-1)^{\bar{X} \bar{Y}} YX.
\]

Let
\[
    J := \sum_{i \in \tI} (-1)^{p(i)} E_{-i,i} =
    \begin{pmatrix}
        0 & -I_n \\
        I_n & 0
    \end{pmatrix}
    \in \Mat_{n|n}(\kk),
\]
where $I_n$ is the $n \times n$ identity matrix.  Multiplication by $J$ is an odd linear automorphism of $V$, and $J^2 = -1$.  The \emph{isomeric Lie superalgebra} $\fq_n$ is the Lie superalgebra equal to the centralizer of $J$ in $\gl_{n|n}$:
\[
    \fq_n := \left\{ X \in \gl_{n|n} : JX = (-1)^{\bar{X}} X J \right\}.
\]
The elements
\[
    e^0_{ab} := E_{ab} + E_{-a,-b},\qquad
    e^1_{ab} := E_{a,-b} + E_{-a,b},\qquad
    a,b \in \{1,2,\dotsc,n\},
\]
give a $\kk$-basis of $\fq_n$.  The parities of these elements are indicated by their superscripts.

Define
\begin{equation} \label{Thetadef}
    \Theta = \sum_{i,j \in \tI} \Theta_{ij} \otimes E_{ij} \in \End_\kk(V)^{\otimes 2} = \Mat_{n|n}(\kk)^{\otimes 2},
\end{equation}
by
\begin{equation} \label{krypton}
    \begin{aligned}
        \Theta &= \sum_{a,j} q^{\delta_{a,j} - \delta_{a,-j}} e^0_{aa} \otimes E_{jj}
        + z \sum_{a<b} e^0_{ba} \otimes E_{ab}
        - z \sum_{a>b} e^0_{ba} \otimes E_{-a,-b}
        - z \sum_{a,b} e^1_{ba} \otimes E_{-a,b}
        \\
        &= \sum_{i,j} q^{\varphi(i,j)} E_{ii} \otimes E_{jj}
        + z \sum_{i<j} (-1)^{p(i)} (E_{ji} + E_{-j,-i}) \otimes E_{ij}.
    \end{aligned}
\end{equation}
The definition of $\Theta$ first appeared in \cite[\S4]{Ols92}, where it is denoted $S$.  We use the notation $\Theta$ to reserve the notation $S$ for the antipode, which will play an important role in the current paper.  It follows immediately from the definition that
\begin{equation} \label{cardinal}
    \Theta(J \otimes 1) = (J \otimes 1)\Theta.
\end{equation}
One can also verify that $\Theta$ satisfies the Yang--Baxter equation:
\begin{equation} \label{YBE}
    \Theta^{12} \Theta^{13} \Theta^{23} = \Theta^{23} \Theta^{13} \Theta^{12},
\end{equation}
where
\[
    \Theta^{12} = \Theta \otimes 1,\qquad
    \Theta^{23} = 1 \otimes \Theta,\qquad
    \Theta^{13} = \sum_{i,j \in \tI} \Theta_{ij} \otimes 1 \otimes E_{ij}.
\]
It follows from \cref{krypton} that
\begin{equation} \label{tavern}
    \Theta_{ii} = \sum_a q^{\delta_{a,i} - \delta_{a,-i}} e^0_{aa}
    \qquad \text{for all } i \in \tI,
\end{equation}
and so
\begin{equation} \label{mango}
    \Theta_{ii} \Theta_{-i,-i} = 1 = \Theta_{-i,-i} \Theta_{i,i}
    \qquad \text{for all } i \in \tI.
\end{equation}
When $q=1$, we have $\Theta = 1 \otimes 1$, and so $\Theta_{ij} = \delta_{ij} 1_V$.

Note that all the second tensor factors appearing in \cref{krypton} are upper triangular elements of $\Mat_{n|n}(\kk)$.  In addition, $\Theta$ is invertible with
\begin{equation}
    \begin{aligned}
        \Theta^{-1} &= \sum_{a,j} q^{\delta_{a,-j} - \delta_{a,j}} e^0_{aa} \otimes E_{jj}
        - z \sum_{a<b} e^0_{ba} \otimes E_{ab}
        + z \sum_{a>b} e^0_{ba} \otimes E_{-a,-b}
        + z \sum_{a,b} e^1_{ba} \otimes E_{-a,b}
        \\
        &= \sum_{i,j} q^{-\varphi(i,j)} E_{ii} \otimes E_{jj}
        - z \sum_{i<j} (-1)^{p(i)} (E_{ji} + E_{-j,-i}) \otimes E_{ij}.
    \end{aligned}
\end{equation}
Note that $\Theta^{-1}$ is obtained from $\Theta$ by replacing $q$ by $q^{-1}$.
\details{
    Setting
    \[
        A(q) = \sum_{i,j} q^{\varphi(i,j)} E_{ii} \otimes E_{jj}
        \quad \text{and} \quad
        B = \sum_{i<j} (-1)^{p(i)} (E_{ji} + E_{-j,-i}) \otimes E_{ij},
    \]
    we have $\Theta = A(q) + zB$.  We want to verify that $\Theta^{-1} = A(q^{-1}) - zB$.  First note that $A(q^{-1}) = A(q)^{-1}$ and
    \[
        B^2 = - \sum_{-a < j < a} (-1)^{p(j)} e^1_{|j|,|j|} \otimes E_{-a,a} = 0.
    \]
    Next we compute
    \begin{align*}
        A(q) B
        &= \sum_{i<j} (-1)^{p(i)} q^{\delta_{|i|,|j|}(1-2p(i))} (E_{ji} + E_{-j,-i}) \otimes E_{ij},
        \\
        B A(q^{-1})
        &= \sum_{i<j} (-1)^{p(i)} q^{\delta_{|i|,|j|}(2p(j)-1)} (E_{ji} + E_{-j,-i}) \otimes E_{ij}.
    \end{align*}
    Thus
    \[
        A(q)B - BA(q^{-1})
        = 0.
    \]
    Hence
    \[
        \left( A(q) + zB \right) \left( A(q)^{-1} - zB \right)
        = 1 + z \left( B A(q)^{-1} - A(q) B \right) - z^2 B
        = 1.
    \]
}

\begin{defin}
    The \emph{quantum isomeric superalgebra} $U_q = U_q(\mathfrak{q}_n)$ is the unital associative superalgebra over $\kk$ generated by elements $u_{ij}$, $i,j \in \tI$, $i \le j$, subject to the relations
    \begin{equation} \label{QQdef}
        u_{ii} u_{-i,-i} = 1 = u_{-i,-i} u_{ii},\qquad
        L^{12} L^{13} \Theta^{23} = \Theta^{23} L^{13} L^{12},
    \end{equation}
    where
    \begin{equation} \label{Ldef}
        L := \sum_{\substack{i,j \in \tI \\ i \le j}} u_{ij} \otimes E_{ij},\qquad
        L^{12} = L \otimes 1,\qquad
        L^{13} = \sum_{\substack{i,j \in \tI \\ i \le j}} u_{ij} \otimes 1 \otimes E_{ij},
    \end{equation}
    and the last equality in \cref{QQdef} takes place in $U_q \otimes \End_\kk(V)^{\otimes 2}$.  The parity of $u_{ij}$ is $p(i,j)$.
\end{defin}

The quantum isomeric Lie superalgebra was first defined in \cite[Def.~4.2]{Ols92}.  It is a Hopf superalgebra with comultiplication determined by
\begin{align} \label{dive}
    \Delta(L) :=& \sum_{\substack{i,j \in \tI \\ i \le j}} \Delta(u_{ij}) \otimes E_{ij} = L^{13} L^{23},
    \quad \text{or, more explicitly,} \\ \label{snooze}
    \Delta(u_{ij}) =& \sum_{\substack{k \in \tI \\ i \le k \le j}} (-1)^{p(i,k)p(k,j)} u_{ik} \otimes u_{kj}
    = \sum_{\substack{k \in \tI \\ i \le k \le j}} u_{ik} \otimes u_{kj}
\end{align}
(where the final equality holds since, for $i \le k \le j$, we must have $p(k)=p(i)$ or $p(k)=p(j)$), counit determined by
\begin{equation} \label{counit}
    \varepsilon(L) := \sum_{\substack{i,j \in \tI \\ i \le j}} \varepsilon(u_{ij}) E_{ij} = 1
    \quad \text{or, more explicitly,} \quad
    \varepsilon(u_{ij}) = \delta_{ij},
\end{equation}
and antipode $S$ determined by
\begin{equation} \label{antipode}
    \sum_{\substack{i,j \in \tI \\ i \le j}} S(u_{ij}) \otimes E_{ij}
    = L^{-1}.
\end{equation}
Note that, viewing $L$ as an element of $\Mat_{n|n}(U_q)$, it follows from its definition and \cref{QQdef} that it is triangular with invertible diagonal entries.  Thus $L$ is indeed invertible.  Since $U_q$ is a Hopf superalgebra, the supercategory $U_q\smod$ of finite-dimensional $U_q$-supermodules is naturally a rigid monoidal supercategory.

For $\kk$-supermodules $U$ and $W$, define
\[
    \flip_{U,W} \colon U \otimes W \to W \otimes U,\quad \flip_{U,W}(u \otimes w) = (-1)^{\bar{u} \bar{w}} w \otimes u.
\]
When $U$ and $W$ are clear from the context, we will sometimes write $\flip$ instead of $\flip_{U,W}$.  Note that
\[
    \flip_{V,V} = \sum_{i,j} (-1)^{p(j)} E_{ij} \otimes E_{ji}.
\]
Consider the opposite comultiplication
\[
    \Delta^\op = \flip \circ \Delta.
\]

\begin{lem}
    We have
    \begin{equation} \label{gummy}
        \Delta^\op(L)
        := \sum_{\substack{i,j \in \tI \\ i \le j}} \Delta^\op(u_{ij}) \otimes E_{ij} = L^{23} L^{13}.
    \end{equation}
\end{lem}

\begin{proof}
    We have
    \[
        \Delta^\op(u_{ij}) = \flip \sum_{\substack{k \in \tI \\ i \le k \le j}} (-1)^{p(i,k)p(k,j)} u_{ik} \otimes u_{kj}
        = \sum_{\substack{k \in \tI \\ i \le k \le j}} u_{kj} \otimes u_{ik}.
    \]
    Since
    \[
        L^{23} L^{13}
        = \left( \sum_{i \le k} 1 \otimes u_{ik} \otimes E_{ik} \right) \left( \sum_{l \le j} u_{lj} \otimes 1 \otimes E_{lj} \right)
        = \sum_{\substack{k \in \tI \\ i \le k \le j}} u_{kj} \otimes u_{ik} \otimes E_{ij},
    \]
    the result follows.
\end{proof}

The following result is stated in \cite[Th.~6.1]{Ols92} without proof.

\begin{prop} \label{clown}
    The quantum isomeric superalgebra $U_q$ is isomorphic is the unital associative superalgebra over $\kk$ generated by the elements $u_{ij}$, $i,j \in \tI$, $i \le j$, subject to the relations
    \begin{equation} \label{smri}
        u_{ii} u_{-i,-i} = 1 = u_{-i,-i} u_{ii},
        \qquad i \in \tI,
    \end{equation}
    and
    \begin{equation} \label{silly}
        \begin{multlined}
            (-1)^{p(i,j)p(k,l)} q^{\varphi(j,l)} u_{ij} u_{kl}
            + z \delta_{i \le l} \delta_{k \le j < l} \theta(i,j,k) u_{il} u_{kj}
            + z \delta_{i \le -l < j \le -k} \theta(-i,-j,k) u_{i,-l} u_{k,-j}
            \\
            = q^{\varphi(i,k)} u_{kl} u_{ij}
            + z \delta_{k < i \le l} \delta_{k \le j} \theta(i,j,k) u_{il} u_{kj}
            + z \delta_{-l \le i < -k \le j} \theta(-i,-j,k) u_{-i,l} u_{-k,j},
        \end{multlined}
    \end{equation}
    for all $i,j,k,l \in \tI$, $i \le j$, $k \le l$, where $\theta(i,j,k) = (-1)^{p(i)p(j) + p(j)p(k) + p(i)p(k)}$.
\end{prop}

\begin{proof}
    It suffices to prove that the relations \cref{silly} are equivalent to the second relation in \cref{QQdef}.  Direct computation shows that
    \begin{multline*}
        L^{12} L^{13} \Theta^{23}
        = \sum_{i \le j,\, k \le l} (-1)^{p(i,j)p(k,l)} q^{\varphi(j,l)} u_{ij} u_{kl} \otimes E_{ij} \otimes E_{kl}
        \\
        + z \sum_{i \le l,\, k \le j < l} \theta(i,j,k) u_{il} u_{kj} \otimes E_{ij} \otimes E_{kl}
        + z \sum_{i \le -l < j \le -k} \theta(-i,-j,k) u_{i,-l} u_{k,-j} \otimes E_{ij} \otimes E_{kl}
    \end{multline*}
    and
    \begin{multline*}
        \Theta^{23} L^{13} L^{12}
        = \sum_{i \le j,\, k \le l} q^{\varphi(i,k)} u_{kl} u_{ij} \otimes E_{ij} \otimes E_{kl}
        \\
        + z \sum_{k < i \le l,\, k \le j} \theta(i,j,k) u_{il} u_{kj} \otimes E_{ij} \otimes E_{kl}
        + z \sum_{-l \le i < -k \le j} \theta(-i,-j,k) u_{-i,l} u_{-k,j} \otimes E_{ij} \otimes E_{kl}.
    \end{multline*}
    The result follows.
    \details{
        We have
        \[
            L^{12}L^{13}
            = \sum_{\substack{i,j,k,l \in \tI \\ i \le j,\, k \le l}} (-1)^{p(i,j)p(k,l)} u_{ij} u_{kl} \otimes E_{ij} \otimes E_{kl},
        \]
        and so
        \begin{align*}
            L^{12} L^{13} \Theta^{23}
            &= \sum_{i \le j,\, k \le l} (-1)^{p(i,j)p(k,l)} q^{\varphi(j,l)} u_{ij} u_{kl} \otimes E_{ij} \otimes E_{kl}
            \\ &\qquad \qquad
            + z \sum_{i \le j,\, k \le l < j} (-1)^{p(i,j)p(k,l)+p(j,l)p(k,l)+p(l)} u_{ij} u_{kl} \otimes E_{il} \otimes E_{kj}
            \\ &\qquad \qquad
            + z \sum_{i \le j,\, k \le l < -j} (-1)^{p(i,j)p(k,l)+p(j,-l)p(k,l)+p(l)} u_{ij} u_{kl} \otimes E_{i,-l} \otimes E_{k,-j}
            \\
            &= \sum_{i \le j,\, k \le l} (-1)^{p(i,j)p(k,l)} q^{\varphi(j,l)} u_{ij} u_{kl} \otimes E_{ij} \otimes E_{kl}
            \\ &\qquad \qquad
            + z \sum_{i \le l,\, k \le j < l} \theta(i,j,k) u_{il} u_{kj} \otimes E_{ij} \otimes E_{kl}
            \\ &\qquad \qquad
            + z \sum_{i \le -l < j \le -k} \theta(-i,-j,k) u_{i,-l} u_{k,-j} \otimes E_{ij} \otimes E_{kl}.
        \end{align*}
        Similarly,
        \[
            L^{13} L^{12}
            = \sum_{\substack{i,j,k,l \in \tI \\ i \le j,\, k \le l}} u_{ij} u_{kl} \otimes E_{kl} \otimes E_{ij},
        \]
        and so
        \begin{align*}
            \Theta^{23} L^{13} L^{12}
            &= \sum_{i \le j,\, k \le l} q^{\varphi(k,i)} u_{ij} u_{kl} \otimes E_{kl} \otimes E_{ij}
            \\ &\qquad \qquad
            + z \sum_{k < i \le j,\, k \le l} (-1)^{p(i,k)p(k,l)+p(k)} u_{ij} u_{kl} \otimes E_{il} \otimes E_{kj}
            \\ &\qquad \qquad
            + z \sum_{i \le j,\, k \le l,\ -k<i} (-1)^{p(i,-k)p(k,l)+p(-k)} u_{ij} u_{kl} \otimes E_{-i,l} \otimes E_{-k,j}
            \\
            &= \sum_{i \le j,\, k \le l} q^{\varphi(i,k)} u_{kl} u_{ij} \otimes E_{ij} \otimes E_{kl}
            \\ &\qquad \qquad
            + z \sum_{k < i \le l,\, k \le j} \theta(i,j,k) u_{il} u_{kj} \otimes E_{ij} \otimes E_{kl}
            \\ &\qquad \qquad
            + z \sum_{-l \le i < -k \le j} \theta(-i,-j,k) u_{-i,l} u_{-k,j} \otimes E_{ij} \otimes E_{kl}.
        \end{align*}
    }
\end{proof}

\begin{cor} \label{donut}
    \begin{enumerate}
        \item \label{donut1} We have $u_{aa} u_{kl} = q^{\delta_{a,|l|} - \delta_{a,|k|}} u_{kl} u_{aa}$ for all $a \in \{1,2,\dotsc,n\}$ and $k,l \in \tI$, $k \le l$.

        \item \label{donut2} The element $u_{11} u_{22} \dotsm u_{nn}$ lies in the center of $U_q$.
    \end{enumerate}
\end{cor}

\begin{proof}
    \begin{enumerate}
        \item Setting $i=j=a$ in \cref{silly} gives
            \[
                q^{\varphi(a,l)} u_{aa} u_{kl} + z \delta_{k \le a < l} u_{al} u_{ka}
                = q^{\varphi(a,k)} u_{kl} u_{aa} + z \delta_{k < a \le l} u_{al} u_{ka},
            \]
            which implies
            \[
                (q^{\varphi(a,l)} - z \delta_{k<a=l}) u_{aa} u_{kl}
                = (q^{\varphi(a,k)} - z \delta_{k=a<l}) u_{kl} u_{aa}.
            \]
            When $k=l$, this becomes $u_{aa} u_{kl} = u_{kl} u_{aa}$, as desired.  When $k=-l$, it becomes
            \[
                (q^{\delta_{a,l}} - z \delta_{a,l}) u_{aa} u_{kl} = q^{-\delta_{a,l}} u_{kl} u_{aa}
                \implies u_{aa} u_{kl} = u_{kl} u_{aa},
            \]
            as desired.

            Now suppose $|k| \ne |l|$.  If $a \notin \{k,l\}$, then
            \[
                u_{aa} u_{kl} = q^{\varphi(a,k)-\varphi(a,l)} u_{kl} u_{aa}
                = q^{\delta_{a,|l|}-\delta_{a,|k|}} u_{kl} u_{aa}.
            \]
            If $a=k$, then
            \[
                u_{aa} u_{kl} = (q-z) u_{kl} u_{aa} = q^{-1} u_{kl} u_{aa}.
            \]
            Finally, if $a=l$, then
            \[
                (q-z) u_{aa} u_{kl} = u_{kl} u_{aa}
                \implies u_{aa} u_{kl} = q u_{kl} u_{aa}.
            \]

        \item It follows from \cref{donut1} that $u_{11} u_{22} \dotsm u_{nn}$ commutes with all $u_{kl}$, $k \le l$. \qedhere
    \end{enumerate}
\end{proof}

\begin{lem} \label{wario}
    As a unital associative superalgebra, $U_q$ is generated by
    \begin{equation} \label{gem}
        u_{a,a+1},\quad u_{-a-1,-a},\quad u_{ii},\quad u_{-1,1},\qquad 1 \le a \le n-1,\quad i \in \tI.
    \end{equation}
\end{lem}

\begin{proof}
    Let $\tilde{U}_q$ be the unital associative sub-superalgebra of $U_q$ generated by the elements \cref{gem}.  It is shown in \cite[Th.~2.1]{GJKK10} that $U_q$ is generated by
    \[
        u_{a,a+1},\quad u_{-a-1,-a},\quad u_{ii},\quad u_{-a-1,a},\quad u_{-a,a+1},\quad u_{-b,b},\qquad 1 \le a \le n-1,\ 1 \le b \le n,\ i \in \tI.
    \]
    Thus it suffices to show that
    \begin{equation} \label{mario}
        u_{-a-1,a},\ u_{-a,a+1},\ u_{-b,b} \in \tilde{U}_q
    \end{equation}
    for $1 \le a \le n-1$ and $1 \le b \le n$.  We prove this by induction on $a$.

    First note that, for $1 \le a \le n-1$, taking $i=-a$, $j=a$, $k=-a-1$, $l=-a$ in \cref{silly} gives
    \[
        q^{-1} u_{-a,a} u_{-a-1,-a}
        = u_{-a-1,-a} u_{-a,a} - z u_{-a,-a} u_{-a-1,a}.
    \]
    Taking $i=-a$, $j=k=a$, and $l=a+1$ in \cref{silly} gives
    \[
        u_{-a,a} u_{a,a+1} + z u_{-a,a+1} u_{aa}
        = q u_{a,a+1} u_{-a,a}.
    \]
    Taking $i=-a-1$, $j=k=-a$, $l=a+1$ in \cref{silly} gives
    \[
        u_{-a-1,-a} u_{-a,a+1} - z u_{-a-1,a+1} u_{-a,-a} + z u_{-a-1,-a-1} u_{-a, a}
        = u_{-a,a+1} u_{-a-1,-a}.
    \]
    So we have
    \begin{align} \label{luigi1}
        u_{-a-1,a} &= z^{-1} u_{aa} u_{-a-1,-a} u_{-a,a} - q^{-1}z^{-1} u_{aa} u_{-a,a} u_{-a-1,-a},
        \\ \label{luigi2}
        u_{-a,a+1} &= qz^{-1} u_{a,a+1} u_{-a,a} u_{-a,-a} - z^{-1} u_{-a,a} u_{a,a+1} u_{-a,-a},
        \\ \label{luigi3}
        u_{-a-1,a+1} &= z^{-1} u_{-a-1,-a} u_{-a,a+1} u_{aa} + u_{-a-1,-a-1} u_{-a,a} u_{aa} - z^{-1} u_{-a,a+1} u_{-a-1,-a} u_{aa}.
    \end{align}

    Taking $a=1$ in \cref{luigi1,luigi2} shows that $u_{-2,1}, u_{-1,2} \in \tilde{U}_q$.  Thus \cref{mario} holds for $a=b=1$.  Now suppose that $1 \le c \le n-2$ and that \cref{mario} holds for $1 \le a, b \le c$.  Then, replacing $a$ by $c$ in \cref{luigi3} shows that $u_{-c-1,c+1} \in \tilde{U}_q$.  Replacing $a$ by $c+1$ in \cref{luigi1,luigi2} then shows that $u_{-c-2,c+1}, u_{-c-1,c+2} \in \tilde{U}_q$.  Hence \cref{mario} holds for $1 \le a,b \le c+1$.  Thus, by induction, \cref{mario} holds for $1 \le a,b \le n-1$.  Finally, taking $a=n-1$ in \cref{luigi3} shows that $u_{-n,n} \in U_q$.
\end{proof}

It will be useful for future arguments to compute the square of the antipode.

\begin{prop}
    The square of the antipode of $U_q$ is given by $S^2(u_{ij}) = q^{2|j|-2|i|} u_{ij}$, $i,j \in \tI$.
\end{prop}

\begin{proof}
    It follows from the defining relations that $U_q$ is a $\Z$-graded Hopf superalgebra, where we define the degree of $u_{ij}$ to be $2|j|-2|i|$. Thus the map $u_{ij} \mapsto q^{2|j|-2|i|} u_{ij}$ is a homomorphism of superalgebras.  Since the antipode is an antihomomorphism of superalgebras, its square is a homomorphism of superalgebras.  Thus, by \cref{wario}, it suffices to prove that
    \[
        S^2(u_{a,a+1}) = q^2 u_{a,a+1},\quad
        S^2(u_{-a-1,-a}) = q^{-2} u_{-a-1,-a},\quad
        S^2(u_{ii}) = u_{ii},\quad
        S^2(u_{-1,1}) = u_{-1,1},
    \]
    for $1 \le a \le n$, $i \in \tI$.

    Using the definition \cref{antipode} of the antipode, which involves inverting an upper triangular matrix, we see that
    \begin{gather*}
        S(u_{a,a+1}) = - u_{-a,-a} u_{a,a+1} u_{-a-1,-a-1},\qquad
        S(u_{-a-1,-a}) = - u_{a+1,a+1} u_{-a-1,-a} u_{aa},
        \\
        S(u_{ii}) = u_{ii}^{-1} = u_{-i,-i},\qquad
        S(u_{-1,1}) = - u_{11} u_{-1,1} u_{-1,-1}.
    \end{gather*}
    By \cref{donut}\ref{donut1} and \cref{smri}, we have
    \begin{equation} \label{ramen}
        u_{aa} u_{kl} = q^{\delta_{a,|l|} - \delta_{a,|k|}} u_{kl}u_{aa}
        \qquad \text{and} \qquad
        u_{-a,-a} u_{kl} = q^{\delta_{a,|k|} - \delta_{a,|l|}} u_{kl} u_{-a,-a},
    \end{equation}
    for all $a \in \{1,2,\dotsc,n\}$ and $k,l \in \tI$, $k \le l$.  In particular,
    \[
        u_{-a,-a} u_{a,a+1} = q u_{a,a+1} u_{-a,-a},\quad
        u_{a+1,a+1} u_{a,a+1} = q u_{a,a+1} u_{a+1,a+1},\quad
        u_{ii} u_{kk} = u_{kk} u_{ii},
    \]
    for all $a \in \{1,2,\dotsc,n\}$ and $i,k \in \tI$.  Thus,
    \[
        S^2(u_{a,a+1})
        = - S(u_{-a-1,-a-1}) S(u_{a,a+1}) S(u_{-a,-a})
        = u_{a+1,a+1} u_{-a,-a} u_{a,a+1} u_{-a-1,-a-1} u_{aa}
        = q^2 u_{a,a+1}.
    \]
    The proof that $S^2(u_{-a-1,-a}) = q^{-2} u_{-a-1,-a}$ is similar.
    \details{
        It follows from \cref{ramen} that
        \[
            u_{a+1,a+1} u_{-a-1,-a} = q^{-1} u_{-a-1,-a} u_{a+1,a+1}
            \quad \text{and} \quad
            u_{-a,-a} u_{-a-1,-a} = q^{-1} u_{-a-1,-a} u_{-a,-a}.
        \]
        Thus,
        \begin{multline*}
            S^2(u_{-a-1,-a})
            = - S(u_{aa}) S(u_{-a-1,-a}) S(u_{a+1,a+1})
            \\
            = u_{-a,-a} u_{a+1,a+1} u_{-a-1,-a} u_{aa} u_{-a-1,-a-1}
            = q^{-2} u_{-a-1,-a}.
        \end{multline*}
    }

    Next, we have
    \[
        S^2(u_{ii}) = S(u_{-i,-i}) = u_{ii}.
    \]
    Finally, \cref{ramen} implies that
    \[
        u_{11} u_{-1,1} = u_{-1,1} u_{11}
        \qquad \text{and} \qquad
        u_{-1,-1} u_{-1,1} = u_{-1,1} u_{-1,-1}.
    \]
    Thus
    \[
        S^2(u_{-1,1})
        = -S \left( u_{11} u_{-1,1} u_{-1,-1} \right)
        = u_{11}^2 u_{-1,1} u_{-1,-1}^2
        = u_{-1,1}.
        \qedhere
    \]
\end{proof}

\begin{cor}
    The antipode $S$ is invertible and
    \begin{equation} \label{Sinv}
        S^{-1}(u_{ij}) = q^{2|i|-2|j|} S(u_{ij})
        \qquad i,j \in \tI.
    \end{equation}
\end{cor}

It follows from \cref{YBE,mango} that
\begin{equation} \label{rho}
    \rho \colon U_q \to \End_\kk(V),\qquad
    u_{ij} \mapsto \Theta_{ij},\qquad i,j \in \tI,
\end{equation}
defines a representation of $U_q$ on $V$. The $U_q$-supermodule structure on the dual space $V^* := \Hom_\kk(V,\kk)$ is given by
\[
    (xf)(v) = (-1)^{\bar{x}\bar{f}} f(S(x)v),\qquad x \in U_q,\ f \in V^*,\ v \in V.
\]
We have the natural evaluation map
\begin{equation}
    \ev \colon V^* \otimes V \to \kk,\qquad f \otimes v \mapsto f(v).
\end{equation}
Let $v_i^*$, $i \in \tI$, be the basis of $V^*$ dual to the basis $v_i$, $i \in \tI$, of $V$, so that
\[
    v_i^*(v_j) = \delta_{ij},\qquad i,j \in \tI.
\]
Then we have the coevaluation  map
\begin{equation}
    \coev \colon \kk \to V \otimes V^*,\qquad 1 \mapsto \sum_{i \in \tI} v_i \otimes v_i^*.
\end{equation}
It is a straightforward exercise, using only the properties of Hopf superalgebras, to verify that $\ev$ and $\coev$ are both homomorphisms of $U_q$-supermodules, where $\kk$ is the trivial $U_q$-supermodule, with action given by the counit $\varepsilon$.
\details{
    For $x \in U_q$, $v \in V$, and $f \in V^*$, we have, using Sweedler notation,
    \begin{multline*}
        \ev(x (f \otimes v))
        = \ev \left( \sum_{(x)} (-1)^{\bar{f} \, \overline{x_{(2)}}} x_{(1)} f \otimes x_{(2)} v  \right)
        \\
        = (-1)^{\bar{x}\bar{f}} f \left( \sum_{(x)} S(x_{(1)}) x_{(2)} v \right)
        = \varepsilon(x) f(v)
        = \varepsilon(x) \ev(f \otimes v).
    \end{multline*}
    Now note that
    \[
        x \coev(1)
        = x \sum_{i \in \tI} v_i \otimes v_i^*
        = \sum_{(x)} \sum_{i \in \tI} x_{(1)} (-1)^{\overline{x_{(2)}}\, \overline{v_i}} v_i \otimes x_{(2)} v_i^*
        \\
        = \sum_{(x)}  x_{(1)} \circ \left( \sum_{i \in I} v_i \otimes v_i^* \right) \circ S(x_{(2)}),
    \]
    where we are using the natural identification of $V \otimes V^*$ with $\End_\kk(V)$.  Under this identification, $\sum_{i \in \tI} v_i \otimes v_i^*$ is the identity map, and so the above becomes
    \[
        \sum_{(x)} x_{(1)} S \left( x_{(2)} \right)
        = \varepsilon(x)
        = \varepsilon(x) \sum_{i \in \tI} v_i \otimes v_i^*
        = \coev(\varepsilon(x)),
    \]
    as desired.
}

\begin{lem}
    The map $J \in \End_\kk(V)$ is an odd isomorphism of $U_q$-supermodules.
\end{lem}

\begin{proof}
    It follows from \cref{Thetadef,cardinal} that
    \[
        \Theta_{ij} J = (-1)^{p(i,j)} J \Theta_{ij}
        \qquad \text{for all } i,j \in \tI.
    \]
    Since $u_{ij}$ acts on $V$ as $\Theta_{ij}$, it follows that $J$ is an odd endomorphism of $U_q$-supermodules.  Since $J^2 = -1$, it is an isomorphism.
\end{proof}

\section{The incarnation superfunctor}

In this section we prove some of our main results.  We describe a full monoidal superfunctor from $\Qcat(z)$ to the category of $U_q$-supermodules, give explicit bases for the morphism spaces in $\Qcat(z)$, and identify the endomorphism superalgebras of $\Qcat(z)$ with walled Brauer--Clifford superalgebras.

Until further notice later in this section, we assume that $\kk = \C(q)$ and $z = q-q^{-1}$.  Recalling the definition \cref{krypton} of $\Theta$, define
\begin{equation} \label{Tdef}
    T := \flip \circ \Theta,
    \qquad \text{so that} \qquad
    T^{-1} = \Theta^{-1} \circ \flip.
\end{equation}
Thus
\begin{align*}
    T
    &= \sum_{i,j} (-1)^{p(i)} q^{\varphi(i,j)} E_{ji} \otimes E_{ij}
    + z \sum_{i<j} E_{ii} \otimes E_{jj}
    - z \sum_{i<j} (-1)^{p(i,j)} E_{i,-i} \otimes E_{-j,j},
    \\
    T^{-1}
    &= \sum_{i,j} (-1)^{p(i)} q^{-\varphi(j,i)} E_{ji} \otimes E_{ij}
    - z \sum_{i>j} E_{ii} \otimes E_{jj}
    - z \sum_{i<j} (-1)^{p(i,j)} E_{i,-i} \otimes E_{-j,j}.
\end{align*}
\details{
    Since $\flip = \sum_{k,l} (-1)^{p(k)} E_{lk} \otimes E_{kl}$, we have
    \begin{align*}
        T &= \left( \sum_{k,l} (-1)^{p(k)} E_{lk} \otimes E_{kl} \right)
        \left( \sum_{i,j} q^{\varphi(i,j)} E_{ii} \otimes E_{jj}
        + z \sum_{i<j} (-1)^{p(i)} (E_{ji} + E_{-j,-i}) \otimes E_{ij} \right)
        \\
        &= \sum_{i,j} (-1)^{p(i)} q^{\varphi(i,j)} E_{ji} \otimes E_{ij}
        + z \sum_{i<j} E_{ii} \otimes E_{jj}
        - z \sum_{i<j} (-1)^{p(i,j)} E_{i,-i} \otimes E_{-j,j}
    \end{align*}
    and
    \begin{align*}
        T^{-1} &=
        \left(
            \sum_{i,j} q^{-\varphi(i,j)} E_{ii} \otimes E_{jj}
            - z \sum_{i<j} (-1)^{p(i)} (E_{ji} + E_{-j,-i}) \otimes E_{ij}
        \right)
        \left( \sum_{k,l} (-1)^{p(k)} E_{lk} \otimes E_{kl} \right)
        \\
        &= \sum_{i,j} (-1)^{p(j)} q^{-\varphi(i,j)} E_{ij} \otimes E_{ji}
        - z \sum_{i<j} E_{jj} \otimes E_{ii}
        - z \sum_{i<j} (-1)^{p(i,j)} E_{-j,j} \otimes E_{i,-i}
        \\
        &= \sum_{i,j} (-1)^{p(i)} q^{-\varphi(j,i)} E_{ji} \otimes E_{ij}
        - z \sum_{i>j} E_{ii} \otimes E_{jj}
        - z \sum_{i<j} (-1)^{p(i,j)} E_{i,-i} \otimes E_{-j,j}.
    \end{align*}
}
Therefore, we have
\begin{align*}
    T(v_i \otimes v_j)
    &= (-1)^{p(i)p(j)} q^{\varphi(i,j)} v_j \otimes v_i
    + z \delta_{i<j} v_i \otimes v_j
    + z \delta_{i+j>0} (-1)^{p(j)} v_{-i} \otimes v_{-j},
    \\
    T^{-1}(v_i \otimes v_j)
    &= (-1)^{p(i)p(j)} q^{-\varphi(j,i)} v_j \otimes v_i
    - z \delta_{i > j} v_i \otimes v_j
    + z \delta_{i+j>0} (-1)^{p(j)} v_{-i} \otimes v_{-j},
\end{align*}
and
\begin{equation} \label{dream}
    T - T^{-1} = z 1_{V \otimes V}.
\end{equation}
\details{
    We have
    \[
        T - T^{-1}
        = \sum_{i,j} (-1)^{p(i)} \left( q^{\varphi(i,j)} - q^{-\varphi(j,i)} \right) E_{ji} \otimes E_{ij} + z \sum_{i \ne j} E_{ii} \otimes E_{jj}
        = z 1_V \otimes 1_V.
    \]
}

\begin{lem} \label{coffee}
    The map $T$ is an isomorphism of $U_q$-supermodules.
\end{lem}

\begin{proof}
    Since it is invertible, it remains to show that it is a homomorphism of $U_q$-supermodules.  To do this, it suffices to show that, as operators on $V \otimes V$, we have an equality
    \[
        T \Delta(u_{ij}) = \Delta(u_{ij}) T
        \qquad \text{for all } i,j \in \tI.
    \]
    Composing on the left with $\flip$, it suffices to show that
    \[
        \Theta \Delta(u_{ij}) = \Delta^\op(u_{ij}) \Theta
        \qquad \text{for all } i,j \in \tI.
    \]
    This is equivalent to showing that
    \[
        \sum_{i,j \in \tI} \Theta \Delta(u_{ij}) \otimes E_{ij}
        = \sum_{i,j \in \tI} \Delta^\op(u_{ij}) \Theta \otimes E_{ij}.
    \]
    Since $u_{ij}$ acts on $V$ as $\Theta_{ij}$, this is equivalent, using \cref{dive,gummy}, to
    \[
        \Theta^{12} \Theta^{13} \Theta^{23} = \Theta^{23} \Theta^{13} \Theta^{12}.
    \]
    But this is precisely the Yang--Baxter equation \cref{YBE}.
\end{proof}

\begin{rem}
    The map $T$ is a special case of a map $T_{MV}$ to be introduced in \cref{burrito}, where $M=V$.  Then \cref{coffee} will be a special case of \cref{coin}.
\end{rem}

For the computations to follow, it is useful to note that, for $i,j \in \tI$, $i<j$, we have
\begin{align} \label{ladder+}
    z \sum_{\substack{k \in \tI \\ i < k < j}} (-1)^{p(k)} q^{2|k|}
    &= q^{2|j|-\sgn(j)} - q^{2|i|+\sgn(i)},
    \\ \label{ladder-}
    z \sum_{\substack{k \in \tI \\ i < k < j}} (-1)^{p(k)} q^{-2|k|}
    &= q^{-2|i|-\sgn(i)} - q^{-2|j|+\sgn(j)}.
\end{align}
\details{
    We have
    \begin{align*}
        z \sum_{\substack{k \in \tI \\ i < k < j}} (-1)^{p(k)} q^{2|k|}
        &=
        \begin{cases}
            q^{2|j|-1} - q^{2|i|+1} & \text{if } 1 \le i \le j, \\
            q^{2|j|-1} - q & \text{if } i=-1,\, j>1, \\
            0 & \text{if } i=-1,\, j=1, \\
            q^{2|j|+1} - q^{2|i|-1} & \text{if } i < j \le -1, \\
            q-q^{2|i|-1} & \text{if } i<-1,\ j=1, \\
            q^{2|j|-1} - q^{2|i|-1} & \text{if } i<-1, j>1,
        \end{cases}
        \\
        &= q^{2|j|-\sgn(j)} - q^{2|i|+\sgn(i)}.
    \end{align*}
    Replacing $q$ by $q^{-1}$ in \cref{ladder+}, and noting that this sends $z$ to $-z$, we obtain \cref{ladder-}.
}

\begin{theo} \label{bread}
    For each $n \in \N$, there exists a unique monoidal superfunctor $\bF_n \colon \Qcat(z) \to U_q\smod$ such that
    \begin{gather*}
        \bF_n(\uparrow) = V,\qquad \qquad
        \bF_n(\downarrow) = V^*,
        \\
        \bF_n(\posupcross) = T,\qquad
        \bF_n(\leftcap) = \ev,\qquad
        \bF_n(\tokup) = J.
    \end{gather*}
    Furthermore, $\bF_n(\negupcross) = T^{-1}$, $\bF_n(\leftcup) = \coev$, and
    \begin{equation} \label{Fnegleft}
        \begin{multlined}
            \bF_n(\negrightcross) \colon v_i \otimes v_j^* \mapsto
            (-1)^{p(i)p(j)} q^{-\varphi(i,j)} v_j^* \otimes v_i
            \\
            - z \delta_{ij} \sum_{k>i} (-1)^{p(i,k)} q^{2|i|-2|k|} v_k^* \otimes v_k
            - z \delta_{i,-j} \sum_{k>j} (-1)^{p(k)} q^{2|i|-2|k|} v_k^* \otimes v_{-k}.
        \end{multlined}
    \end{equation}
\end{theo}

We call $\bF_n$ the \emph{incarnation superfunctor}.  Before giving the proof of \cref{bread}, we compute, using the definitions \cref{lego,wolverine,ldskein}, the images under $\bF$ of the leftward and downward crossings:
\begin{align} \label{Fposleft}
    \begin{split}
        \bF_n(\posleftcross) \colon v_i^* \otimes v_j
        \mapsto{}& (-1)^{p(i)p(j)} q^{\varphi(j,i)} v_j \otimes v_i^* \\
        &\qquad \qquad + z \delta_{ij} \sum_{k > i} v_k \otimes v_k^* + z \delta_{i,-j} \sum_{k>i} (-1)^{p(k)} v_{-k} \otimes v_k^*,
    \end{split}
    \\
    \begin{split}
        \bF_n(\negleftcross) \colon v_i^* \otimes v_j
        \mapsto{}& (-1)^{p(i)p(j)} q^{-\varphi(i,j)} v_j \otimes v_i^* \\
        &\qquad \qquad - z \delta_{ij} \sum_{k<i} v_k \otimes v_k^* + z \delta_{i,-j} \sum_{k>i} (-1)^{p(k)} v_{-k} \otimes v_k^*,
    \end{split}
    \\
    \begin{split}
        \bF_n(\posdowncross) \colon v_i^* \otimes v_j^*
        \mapsto{}& (-1)^{p(i)p(j)} q^{\varphi(i,j)} v_j^* \otimes v_i^* \\
        &\qquad \qquad + z \delta_{i>j} v_i^* \otimes v_j^* - z \delta_{i+j<0} (-1)^{p(i)} v_{-i}^* \otimes v_{-j}^*,
    \end{split}
    \\
    \begin{split}
        \bF_n(\negdowncross) \colon v_i^* \otimes v_j^*
        \mapsto{}& (-1)^{p(i)p(j)} q^{-\varphi(j,i)} v_j^* \otimes v_i^* \\
        &\qquad \qquad - z \delta_{i<j} v_i^* \otimes v_j^* - z \delta_{i+j<0} (-1)^{p(i)} v_{-i}^* \otimes v_{-j}^*,
    \end{split}
\end{align}
the right cup and cap
\begin{align} \label{bowser}
    \bF_n(\rightcap) \colon v_i \otimes v_j^*
    &\mapsto \delta_{ij} (-1)^{p(i)} q^{2|i|-2n-1},&
    \bF_n(\rightcup) \colon 1
    &\mapsto \sum_{i \in \tI} (-1)^{p(i)} q^{2n-2|i|+1} v_i^* \otimes v_i,
\end{align}
and the positive right crossing
\begin{multline*}
    \bF_n(\posrightcross) \colon v_i \otimes v_j^*
    \mapsto (-1)^{p(i)p(j)} q^{\varphi(j,i)} v_j^* \otimes v_i
    \\
    + z \delta_{ij} \sum_{k < i} (-1)^{p(i,k)} q^{2|i|-2|k|} v_k^* \otimes v_k
    - z \delta_{i,-j} \sum_{k>j} (-1)^{p(k)} q^{2|i|-2|k|} v_k^* \otimes v_{-k}.
\end{multline*}
(See \cref{word} for another description of the images under $\bF_n$ of the various crossings.)
\details{
    In this details environment, we will denote the image of a diagram under $\bF_n$ by the diagram itself.  For $\bF_n \left( \posleftcross \right)$, we compute
    \begin{align*}
        v_i^* \otimes v_j
        &\xmapsto{\downstrand \otimes \upstrand \otimes \leftcup} \sum_k v_i^* \otimes v_j \otimes v_k \otimes v_k^*
        \\
        &\xmapsto{\downstrand \otimes \posupcross \otimes \downstrand} \sum_k (-1)^{p(j)p(k)} q^{\varphi(j,k)} v_i^* \otimes v_k \otimes v_j \otimes v_k^* \\
        &\qquad \qquad \qquad \qquad + z \sum_{k>j} v_i^* \otimes v_j \otimes v_k \otimes v_k^* + z \sum_{k>-j} (-1)^{p(k)} v_i^* \otimes v_{-j} \otimes v_{-k} \otimes v_k^*
        \\
        &\xmapsto{\leftcap \otimes \upstrand \otimes \downstrand} (-1)^{p(i)p(j)} q^{\varphi(j,i)} v_j \otimes v_i^* + z \delta_{ij} \sum_{k > i} v_k \otimes v_k^* + z \delta_{i,-j} \sum_{k>i} (-1)^{p(k)} v_{-k} \otimes v_k^*.
    \end{align*}
    For $\bF_n \left( \negleftcross \right)$, we compute
    \begin{align*}
        v_i^* \otimes v_j
        &\xmapsto{\downstrand \otimes \upstrand \otimes \leftcup} \sum_k v_i^* \otimes v_j \otimes v_k \otimes v_k^*
        \\
        &\xmapsto{\downstrand \otimes \negupcross \otimes \downstrand} \sum_k (-1)^{p(j)p(k)} q^{-\varphi(k,j)} v_i^* \otimes v_k \otimes v_j \otimes v_k^* \\
        &\qquad \qquad \qquad \qquad - z \sum_{k<j} v_i^* \otimes v_j \otimes v_k \otimes v_k^* + z \sum_{k>-j} (-1)^{p(k)} v_i^* \otimes v_{-j} \otimes v_{-k} \otimes v_k^*
        \\
        &\xmapsto{\leftcap \otimes \upstrand \otimes \downstrand} (-1)^{p(i)p(j)} q^{-\varphi(i,j)} v_j \otimes v_i^* - z \delta_{ij} \sum_{k<i} v_k \otimes v_k^* + z \delta_{i,-j} \sum_{k>i} (-1)^{p(k)} v_{-k} \otimes v_k^*.
    \end{align*}
    For $\bF_n \left( \posdowncross \right)$, we compute
    \begin{align*}
        v_i^* \otimes v_j^*
        &\xmapsto{\downstrand \otimes \downstrand \otimes \leftcup} \sum_k v_i^* \otimes v_j^* \otimes v_k \otimes v_k^*
        \\
        &\xmapsto{\downstrand \otimes \posleftcross \otimes \downstrand} \sum_k (-1)^{p(j)p(k)} q^{\varphi(k,j)} v_i^* \otimes v_k \otimes v_j^* \otimes v_k^* \\
        &\qquad \qquad \qquad \qquad + z \sum_k \sum_{l>j} \left( \delta_{jk} v_i^* \otimes v_l \otimes v_l^* \otimes v_k^* + \delta_{j,-k} (-1)^{p(l)} v_i^* \otimes v_{-l} \otimes v_l^* \otimes v_k^* \right)
        \\
        &\xmapsto{\leftcap \otimes \downstrand \otimes \downstrand} (-1)^{p(i)p(j)} q^{\varphi(i,j)} v_j^* \otimes v_i + z \delta_{i>j} v_i^* \otimes v_j^* - z \delta_{i+j<0} (-1)^{p(i)} v_{-i}^* \otimes v_{-j}^*.
    \end{align*}
    For $\bF_n \left( \negdowncross \right)$, we compute
    \begin{align*}
        v_i^* \otimes v_j^*
        &\xmapsto{\downstrand \otimes \downstrand \otimes \leftcup} \sum_k v_i^* \otimes v_j^* \otimes v_k \otimes v_k^*
        \\
        &\xmapsto{\downstrand \otimes \negleftcross \otimes \downstrand} \sum_k (-1)^{p(j)p(k)} q^{-\varphi(j,k)} v_i^* \otimes v_k \otimes v_j^* \otimes v_k^* \\
        &\qquad \qquad \qquad \qquad - z \sum_k \sum_{l<j} \delta_{jk} v_i^* \otimes v_l \otimes v_l^* \otimes v_k^* + z \sum_k \sum_{l>j} \delta_{j,-k} (-1)^{p(l)} v_i^* \otimes v_{-l} \otimes v_l^* \otimes v_k^*
        \\
        &\xmapsto{\leftcap \otimes \downstrand \otimes \downstrand} (-1)^{p(i)p(j)} q^{-\varphi(j,i)} v_j^* \otimes v_i^* - z \delta_{i<j} v_i^* \otimes v_j^* - z \delta_{i+j<0} (-1)^{p(i)} v_{-i}^* \otimes v_{-j}^*.
    \end{align*}
    For $\bF_n \left( \rightcap \right)$, we compute, using \cref{lego},
    \begin{align*}
        v_i \otimes v_j^*
        &\xmapsto{\negrightcross} (-1)^{p(i)p(j)} q^{-\varphi(i,j)} v_j^* \otimes v_i
        \\
        &\qquad \qquad - z \delta_{ij} \sum_{k>i} (-1)^{p(i,k)} q^{2|i|-2|k|} v_k^* \otimes v_k
        - z \delta_{i,-j} \sum_{k>j} (-1)^{p(k)} q^{2|i|-2|k|} v_k^* \otimes v_{-k}
        \\
        &\xmapsto{\leftcap} \delta_{ij} (-1)^{p(i)} q^{-\sgn(i)} - \delta_{ij} (-1)^{p(i)} q^{2|i|} z \sum_{k>i} (-1)^{p(k)} q^{-2|k|}
        \\
        &\overset{\mathclap{\cref{ladder-}}}{=}\ \delta_{ij} (-1)^{p(i)} q^{-\sgn(i)} - \delta_{ij} (-1)^{p(i)} q^{2|i|} \left( q^{-2|i|-\sgn(i)} - q^{-2n-1} \right)
        \\
        &= \delta_{ij} (-1)^{p(i)} q^{2|i|-2n-1}.
    \end{align*}
    For $\bF_n \left( \rightcup \right)$, we compute, using \cref{wolverine},
    \begin{align*}
        1
        &\xmapsto{\leftcup} \sum_{i \in \tI} v_i \otimes v_i^*
        \\
        &\xmapsto{\negrightcross} \sum_{i \in \tI} \left( (-1)^{p(i)} q^{-\sgn(i)} v_i^* \otimes v_i - z \sum_{k>i} (-1)^{p(i,k)} q^{2|i|-2|k|} v_k^* \otimes v_k \right) \\
        &= \sum_i \left( (-1)^{p(i)} q^{-\sgn(i)} - (-1)^{p(i)} q^{-2|i|} z \sum_{k=-n}^{i-1} (-1)^{p(k)} q^{2|k|} \right) v_i^* \otimes v_i
        \\
        &\overset{\mathclap{\cref{ladder+}}}{=}\ \sum_{i \in \tI} \left( (-1)^{p(i)} q^{-\sgn(i)} - (-1)^{p(i)} q^{-2|i|} \left( q^{2|i|-\sgn(i)} - q^{2n+1} \right) \right) v_i^* \otimes v_i
        \\
        &= \sum_{i \in \tI} (-1)^{p(i)} q^{2n-2|i|+1} v_i^* \otimes v_i.
    \end{align*}
    Next we compute
    \begin{multline*}
        \bF_n
        \left(
            \begin{tikzpicture}[centerzero]
                \draw[->] (-0.15,0.25) -- (-0.15,0.2) arc(180:360:0.15) -- (0.15,0.25);
                \draw[->] (-0.15,-0.25) -- (-0.15,-0.2) arc(180:0:0.15) -- (0.15,-0.25);
            \end{tikzpicture}
        \right)
        \colon
        v_i \otimes v_j^*
        \mapsto \delta_{ij} (-1)^{p(i)} q^{2|i|-2n-1} \sum_{k \in \tI} (-1)^{p(k)} q^{2n-2|k|+1} v_k^* \otimes v_k
        \\
        = \delta_{ij} \sum_{k \in \tI} (-1)^{p(i,k)} q^{2|i|-2|k|} v_k^* \otimes v_k.
    \end{multline*}
    Then, for $\bF_n \left( \posrightcross \right)$, we compute, using \cref{rskein},
    \begin{align*}
        v_i \otimes v_j^*
        &\xmapsto{\posrightcross}
        (-1)^{p(i)p(j)} q^{-\varphi(i,j)} v_j^* \otimes v_i
        \\
        &\qquad \qquad - z \delta_{ij} \sum_{k>i} (-1)^{p(i,k)} q^{2|i|-2|k|} v_k^* \otimes v_k
        - z \delta_{i,-j} \sum_{k>j} (-1)^{p(k)} q^{2|i|-2|k|} v_k^* \otimes v_{-k}
        \\
        &\qquad \qquad + z \delta_{ij} \sum_{k \in \tI} (-1)^{p(i,k)} q^{2|i|-2|k|} v_k^* \otimes v_k
        \\
        &= (-1)^{p(i)p(j)} q^{-\varphi(i,j)} v_j^* \otimes v_i
        \\
        &\qquad \qquad + z \delta_{ij} \sum_{k \le i} (-1)^{p(i,k)} q^{2|i|-2|k|} v_k^* \otimes v_k
        - z \delta_{i,-j} \sum_{k>j} (-1)^{p(k)} q^{2|i|-2|k|} v_k^* \otimes v_{-k}
        \\
        &= (-1)^{p(i)p(j)} q^{\varphi(j,i)} v_j^* \otimes v_i
        \\
        &\qquad \qquad + z \delta_{ij} \sum_{k < i} (-1)^{p(i,k)} q^{2|i|-2|k|} v_k^* \otimes v_k
        - z \delta_{i,-j} \sum_{k>j} (-1)^{p(k)} q^{2|i|-2|k|} v_k^* \otimes v_{-k}.
    \end{align*}
}

\details{
    As a consistency check, note that $\bF_n \left( \posleftcross \right) - \bF_n \left( \negleftcross \right)$ is the map
    \[
        v_i^* \otimes v_j
        \mapsto (-1)^{p(i)p(j)} \left( q^{\varphi(j,i)} - q^{-\varphi(i,j)} \right) v_j \otimes v_i^* + z \delta_{ij} \sum_{k \ne i} v_k \otimes v_k^*
        = z \delta_{ij} \sum_{k \in \tI} v_k \otimes v_k^*,
    \]
    which is equal to the map $z \bF_n
    \left(
        \begin{tikzpicture}[centerzero]
            \draw[<-] (-0.15,0.25) -- (-0.15,0.2) arc(180:360:0.15) -- (0.15,0.25);
            \draw[<-] (-0.15,-0.25) -- (-0.15,-0.2) arc(180:0:0.15) -- (0.15,-0.25);
        \end{tikzpicture}
    \right)$.  Similarly, $\bF_n \left( \posdowncross \right) - \bF_n \left( \negdowncross \right)$ is the map
    \begin{align*}
        v_i^* \otimes v_j^*
        &\mapsto (-1)^{p(i)p(j)} \left( q^{\varphi(i,j)} - q^{-\varphi(j,i)} \right) v_j^* \otimes v_i^* + z \delta_{i \ne j} v_i^* \otimes v_j^*
        = z v_i^* \otimes v_j^*,
    \end{align*}
    which is equal to the map $z \bF_n (\downstrand\, \downstrand)$.
}

\details{
    As another consistency check, we show that $\negleftcross$ and $\posrightcross$ are inverse.  We compute
    \begin{align*}
        v_i &\otimes v_i^*
        \xmapsto{\bF_n \left( \posrightcross \right)}
        (-1)^{p(i)} q^{\sgn(i)} v_i^* \otimes v_i + z \sum_{k<i} (-1)^{p(i,k)} q^{2|i|-2|k|} v_k^* \otimes v_k
        \\
        &\xmapsto{\bF_n \left( \negleftcross \right)}
        v_i \otimes v_i^* - (-1)^{p(i)} q^{\sgn(i)} z \sum_{l<i} v_l \otimes v_l^*
        \\ &\qquad \qquad
        + z \sum_{l<i} (-1)^{p(i)} q^{2|i|-2|l|-\sgn(l)} v_l \otimes v_l^*
        - z^2 \sum_{l<k<i} (-1)^{p(i,k)} q^{2|i|-2|k|} v_l \otimes v_l^*
        \\
        &= v_i \otimes v_i^* - (-1)^{p(i)} q^{2|i|} z \sum_{l<i} \left( q^{-2|i|+\sgn(i)} - q^{-2|l|-\sgn(l)} + z \sum_{k=l+1}^{i-1} (-1)^{p(k)} q^{-2|k|} \right) v_l \otimes v_l*
        \\
        &\overset{\mathclap{\cref{ladder-}}}{=}\ v_i \otimes v_i^*
    \end{align*}
    and
    \begin{align*}
        v_i &\otimes v_{-i}^*
        \xmapsto{\bF_n \left( \posrightcross \right)}
        q^{\sgn(i)} v_{-i}^* \otimes v_i - z \sum_{k>-i} (-1)^{p(k)} q^{2|i|-2|k|} v_k^* \otimes v_{-k}
        \\
        &\xmapsto{\bF_n \left( \negleftcross \right)}
        v_i \otimes v_{-i}^* + q^{\sgn(i)} z \sum_{l>-i} (-1)^{p(l)} v_{-l} \otimes v_l^*
        \\ &\qquad \qquad
        - z \sum_{l>-i} (-1)^{p(l)} q^{2|i|-2|l|+\sgn(l)} v_{-l} \otimes v_l^*
        - z^2 \sum_{l>k>-i} (-1)^{p(k,l)} q^{2|i|-2|k|} v_{-l} \otimes v_l^*
        \\
        &= v_i \otimes v_{-i}^* + q^{2|i|} z \sum_{l>-i} (-1)^{p(l)} \left( q^{-2|i|+\sgn(i)} - q^{-2|l|+\sgn(l)} - z \sum_{k=-i+1}^{l-1} (-1)^{p(k)} q^{-2|k|} \right) v_{-l} \otimes v_l^*
        \\
        &\overset{\mathclap{\cref{ladder-}}}{=}\ v_i \otimes v_{-i}^*.
    \end{align*}
    The case of $v_i^* \otimes v_j$ for $i \ne \pm j$ is straightforward.
}

\begin{proof}[Proof of \cref{bread}]
    We first show existence, taking $\bF_n (\negupcross) = T^{-1}$, $\bF_n (\leftcup) = \coev$, and $\bF_n(\negleftcross)$ as in \cref{Fnegleft}.  We must show that $\bF_n$ respects the relations in \cref{Qdef}.

    The first two relations in \cref{braid} are clear.  To verify the third relation in \cref{braid}, we compute
    \begin{align*}
        v_i &\otimes v_i^*
        \xmapsto{\bF_n\left(\negrightcross\right)} (-1)^{p(i)} q^{-\sgn(i)} v_i^* \otimes v_i
        - z \sum_{k>i} (-1)^{p(i,k)} q^{2|i|-2|k|} v_k^* \otimes v_k
        \\
        &\xmapsto{\bF_n\left(\posleftcross\right)} v_i \otimes v_i^* + z \sum_{l>i} q^{-\sgn(i)} (-1)^{p(i)} v_l \otimes v_l^* \\
        &\qquad \qquad
        - z \sum_{l>i} (-1)^{p(i)} q^{2|i|-2|l|+\sgn(l)} v_l \otimes v_l^*
        - z^2 \sum_{l>k>i} (-1)^{p(i,k)} q^{2|i|-2|k|} v_l \otimes v_l^*
        \\
        &= v_i \otimes v_i^* + (-1)^{p(i)} z q^{2|i|} \sum_{l>i} \left( q^{-2|i|-\sgn(i)} - q^{-2|l|+\sgn(l)} - z \sum_{k : i<k<l} (-1)^{p(k)} q^{-2|k|} \right) v_l \otimes v_l^*
        \\
        &\overset{\mathclap{\cref{ladder-}}}{=}\ v_i \otimes v_i^*,
    \end{align*}
    and
    \begin{align*}
        v_i &\otimes v_{-i}^*
        \xmapsto{\bF_n\left(\negrightcross\right)} q^{\sgn(i)} v_{-i}^* \otimes v_i
        - z \sum_{k>-i} (-1)^{p(k)} q^{2|i|-2|k|} v_k^* \otimes v_{-k}
        \\
        &\xmapsto{\bF_n\left(\posleftcross\right)} v_i \otimes v_{-i}^* + z \sum_{l>-i} q^{\sgn(i)} (-1)^{p(l)} v_{-l} \otimes v_l^* \\
        &\qquad \qquad
        - z \sum_{l>-i} (-1)^{p(l)} q^{2|i|-2|l|+\sgn(l)} v_{-l} \otimes v_l^* - z^2 \sum_{l>k>-i} q^{2|i|-2|k|} (-1)^{p(k,l)} v_{-l} \otimes v_l^*
        \\
        &= v_i \otimes v_{-i}^* + z q^{2|i|} \sum_{l>-i} (-1)^{p(l)} \left( q^{-2|i|+\sgn(i)} - q^{-2|l|+\sgn(l)} - z \sum_{k : -i<k<l} (-1)^{p(k)} q^{-2|k|} \right) v_{-l} \otimes v_l^*
        \\
        &\overset{\mathclap{\cref{ladder-}}}{=}\ v_i \otimes v_{-i}^*,
    \end{align*}
    and, for $i \ne \pm j$,
    \[
        v_i \otimes v_j^*
        \xmapsto{\bF_n\left(\negrightcross\right)} (-1)^{p(i)p(j)} v_j \otimes v_i^*
        \xmapsto{\bF_n\left(\posleftcross\right)} v_i \otimes v_j^*.
    \]
    Thus
    \[
        \bF_n
        \left(
            \begin{tikzpicture}[centerzero]
                \draw[<-] (0.2,-0.4) to[out=135,in=down] (-0.15,0) to[out=up,in=225] (0.2,0.4);
                \draw[wipe] (-0.2,-0.4) to[out=45,in=down] (0.15,0) to[out=up,in=-45] (-0.2,0.4);
                \draw[->] (-0.2,-0.4) to[out=45,in=down] (0.15,0) to[out=up,in=-45] (-0.2,0.4);
            \end{tikzpicture}
        \right)
        =
        \bF_n \left( \posleftcross \right) \circ \left( \negrightcross \right)
        = \bF_n
        \left(
            \begin{tikzpicture}[centerzero]
                \draw[->] (-0.2,-0.4) -- (-0.2,0.4);
                \draw[<-] (0.2,-0.4) -- (0.2,0.4);
            \end{tikzpicture}
        \right).
    \]
    So $\bF_n$ respects the third relation in \cref{braid}.  Since $V \otimes V^*$ is finite dimensional, it follows that we also have
    \[
        \bF_n
        \left(
            \begin{tikzpicture}[centerzero]
                \draw[<-] (-0.2,-0.4) to[out=45,in=down] (0.15,0) to[out=up,in=-45] (-0.2,0.4);
                \draw[wipe] (0.2,-0.4) to[out=135,in=down] (-0.15,0) to[out=up,in=225] (0.2,0.4);
                \draw[->] (0.2,-0.4) to[out=135,in=down] (-0.15,0) to[out=up,in=225] (0.2,0.4);
            \end{tikzpicture}
        \right)
        =
        \bF_n \left( \negrightcross \right) \circ \left( \posleftcross \right)
        = \bF_n
        \left(
            \begin{tikzpicture}[centerzero]
                \draw[<-] (-0.2,-0.4) -- (-0.2,0.4);
                \draw[->] (0.2,-0.4) -- (0.2,0.4);
            \end{tikzpicture}
        \right).
    \]
    Hence $\bF_n$ also respects the fourth relation in \cref{braid}.
    \details{
        We can also verify the fourth relation in \cref{braid} directly.  We have
        \begin{align*}
            v_i^* &\otimes v_i
            \xmapsto{\bF_n\left(\posleftcross\right)} (-1)^{p(i)} q^{\sgn(i)} v_i \otimes v_i^* + z \sum_{k > i} v_k \otimes v_k^*
            \\
            &\xmapsto{\bF_n\left(\negrightcross\right)} v_i^* \otimes v_i - q^{\sgn(i)} z \sum_{l>i} (-1)^{p(l)} q^{2|i|-2|l|} v_l^* \otimes v_l \\
            &\qquad \qquad + z \sum_{l>i} q^{-\sgn(l)} (-1)^{p(l)} v_l^* \otimes v_l - z^2 \sum_{l > k > i} (-1)^{p(k,l)} q^{2|k|-2|l|} v_l^* \otimes v_l
            \\
            &= v_i^* \otimes v_i - z \sum_{l > i} (-1)^{p(l)} q^{-2|l|} \left( q^{2|i|+\sgn(i)} - q^{2|l|-\sgn(l)} + z \sum_{k: i<k<l} (-1)^{p(k)} q^{2|k|} \right) v_l^* \otimes v_l
            \\
            &\overset{\mathclap{\cref{ladder+}}}{=}\ v_i^* \otimes v_i,
        \end{align*}
        and
        \begin{align*}
            v_i^* &\otimes v_{-i}
            \xmapsto{\bF_n\left(\posleftcross\right)} q^{\sgn(i)} v_{-i} \otimes v_i^* + z \sum_{k > i} (-1)^{p(k)} v_{-k} \otimes v_k^*
            \\
            &\xmapsto{\bF_n\left(\negrightcross\right)} v_i^* \otimes v_{-i} - z q^{\sgn(i)} \sum_{l>i} (-1)^{p(l)} q^{2|i|-2|l|} v_l^* \otimes v_{-l} \\
            &\qquad \qquad + z \sum_{l > i} (-1)^{p(l)} q^{-\sgn(l)} v_l^* \otimes v_{-l}
            - z^2 \sum_{l>k>i} (-1)^{p(k,l)} q^{2|k|-2|l|} v_l^* \otimes v_{-l}
            \\
            &= v_i^* \otimes v_{-i} - z \sum_{l>i} (-1)^{p(l)} q^{-2|l|} \left( q^{2|i|+\sgn(i)} - q^{2|l|-\sgn(l)} + z \sum_{k: i<k<l} (-1)^{p(k)} q^{2|k|} \right) v_l^* \otimes v_{-l}
            \\
            &\overset{\mathclap{\cref{ladder-}}}{=}\ v_i^* \otimes v_{-i}.
        \end{align*}
        The case of $v_i^* \otimes v_j$ for $i \ne \pm j$ is straightforward.
    }

    Next we verify the braid relation (the last relation in \cref{braid}).  The left-hand side is mapped by $\bF_n$ to the composite
    \[
        (T \otimes 1_V) (1_V \circ T) (T \otimes 1_V)
        = \flip^{12} \Theta^{12} \flip^{23} \Theta^{23} \flip^{12} \Theta^{12}
        = \flip^{12} \flip^{23} \flip^{12} \Theta^{23} \Theta^{13} \Theta^{12}.
    \]
    Similarly, the right-hand side is mapped by $\bF_n$ to the composite
    \[
        (1_V \otimes T) (T \otimes 1_V) (1_V \otimes T)
        = \flip^{23} \Theta^{23} \flip^{12} \Theta^{12} \flip^{23} \Theta^{23}
        = \flip^{23} \flip^{12} \flip^{23} \Theta^{12} \Theta^{13} \Theta^{23}.
    \]
    Since
    \[
        \flip^{12} \flip^{23} \flip^{12}
        = \flip^{23} \flip^{12} \flip^{23}
        \colon u \otimes v \otimes w \mapsto (-1)^{\bar{u}\bar{v} + \bar{u}\bar{w} + \bar{v}\bar{w}} w \otimes v \otimes u,
    \]
    and $\Theta^{23} \Theta^{13} \Theta^{12} = \Theta^{12} \Theta^{13} \Theta^{23}$ by \cref{YBE}, we see that $\bF_n$ respects the braid relation.

    Since
    \[
        \bF_n \left( \posupcross \right) - \bF_n \left( \negupcross \right)
        = T - T^{-1}
        \overset{\cref{dream}}{=} z 1_{V \otimes V},
    \]
    the superfunctor $\bF_n$ respects the skein relation \cref{skein}.  We also have
    \[
        \bF_n
        \left(
            \begin{tikzpicture}[centerzero]
                \draw[->] (0,-0.3) -- (0,0.3);
                \token{0,-0.1};
                \token{0,0.1};
            \end{tikzpicture}
        \right)
        = J^2
        = - 1_V,
    \]
    and so $\bF_n$ respects the first relation in \cref{tokrel}.  Next, we compute
    \[
        \bF_n
        \left(
            \begin{tikzpicture}[centerzero]
                \draw[->] (0.3,-0.3) -- (-0.3,0.3);
                \draw[wipe] (-0.3,-0.3) -- (0.3,0.3);
                \draw[->] (-0.3,-0.3) -- (0.3,0.3);
                \token{-0.15,-0.15};
            \end{tikzpicture}
        \right)
        =
        \flip \circ \Theta \circ (J \otimes 1)
        \overset{\cref{cardinal}}{=}
        \flip \circ (J \otimes 1) \circ \Theta
        = (1 \otimes J) \circ \flip \circ \Theta
        = \bF_n
        \left(
            \begin{tikzpicture}[centerzero]
                \draw[->] (0.3,-0.3) -- (-0.3,0.3);
                \draw[wipe] (-0.3,-0.3) -- (0.3,0.3);
                \draw[->] (-0.3,-0.3) -- (0.3,0.3);
                \token{0.15,0.15};
            \end{tikzpicture}
        \right).
    \]
    Thus $\bF_n$ preserves the second relation in \cref{tokrel}.  For the third equality in \cref{tokrel}, we compute
    \[
        \bF_n
        \left(
            \begin{tikzpicture}[centerzero]
                \bubright{0,0};
                \token{-0.2,0};
            \end{tikzpicture}
        \right)
        \colon 1
        \xmapsto{\bF_n \left( \leftcup \right)}
        \sum_{k \in \tI} v_k \otimes v_k^*
        \xmapsto{J \otimes 1}
        \sum_{k \in \tI} (-1)^{p(k)} v_{-k} \otimes v_k^*
        \xmapsto{\bF_n \left( \rightcap \right)}
        0.
    \]
    For the last equality in \cref{tokrel}, we compute
    \[
        \bF_n
        \left(
            \begin{tikzpicture}[centerzero]
                \bubright{0,0};
            \end{tikzpicture}
        \right)
        \colon 1
        \xmapsto{\bF_n \left( \leftcup \right)}
        \sum_{k \in \tI} v_k \otimes v_k^*
        \xmapsto{\bF_n \left( \rightcap \right)}
        \sum_{k \in \tI} (-1)^{p(k)} q^{2|k|-2n-1}
        = 0.
    \]
    Finally, the relations \cref{leftadj} are straightforward to verify.

    It remains to prove uniqueness.  Suppose $\bF_n$ is a monoidal superfunctor as described in the first sentence of the statement of the theorem.  Then $\bF_n(\negupcross)$ and $\bF_n(\negrightcross)$ are uniquely determined by the fact that they must be inverse to $\bF_n(\posupcross)$ and $\bF_n(\posleftcross)$, respectively.  Next, suppose that
    \[
        \bF_n(\leftcup) \colon 1 \mapsto \sum_{i,j} a_{ij} v_i \otimes v_j^*,\qquad
        a_{ij} \in \C(q).
    \]
    Then, for all $k \in \tI$,
    \[
        v_k =
        \bF_n
        \left(
            \begin{tikzpicture}[centerzero]
                \draw[->] (0,-0.4) -- (0,0.4);
            \end{tikzpicture}
        \right)
        (v_k)
        =
        \bF_n
        \left(
            \begin{tikzpicture}[centerzero]
                \draw[<-] (-0.3,0.4) -- (-0.3,0) arc(180:360:0.15) arc(180:0:0.15) -- (0.3,-0.4);
            \end{tikzpicture}
        \right)
        (v_k)
        \xmapsto{\bF_n \left( \leftcup\, \otimes \upstrand \right)}
        \sum_{i,j \in \tI} a_{ij} v_i \otimes v_j^* \otimes v_k
        \xmapsto{\bF_n \left( \upstrand \otimes \leftcap \right)}
        \sum_{i \in \tI} a_{ik} v_i.
    \]
    It follows that $a_{ij} = \delta_{ij}$ for all $i,j \in \tI$, and so $\bF_n(\leftcup) = \coev$.
\end{proof}

To simplify notation, we will start writing objects of $\Qcat(z)$ as sequences of $\uparrow$'s and $\downarrow$'s, omitting the $\otimes$ symbol.  For such an object $X$, we define $V^X := \bF_n(X)$, and we let $\# X$ denote the length of the sequence.

\begin{theo} \label{full}
    The superfunctor $\bF_n$ is full for all $n \in \N$.  Furthermore, the induced map
    \begin{equation} \label{revolt}
        \bF_n \colon \Hom_{\Qcat(z)}(X,Y) \to \Hom_{U_q}(V^X,V^Y)
    \end{equation}
    is an isomorphism when $\# X + \# Y \le 2n$.
\end{theo}

\begin{proof}
    Our proof is similar to that of \cite[Th~4.1]{BCK19}, which treats the case $z=0$.  We need to show that, for all objects $X$ and $Y$ in $\Qcat(z)$, the map \cref{revolt} is surjective, and that it is also injective when $\# X + \# Y \le 2n$.  Suppose that $X$ (respectively, $Y$) is a tensor product of $r_X$ (respectively, $r_Y$) copies of $\uparrow$ and $s_X$ (respectively, $s_Y$) copies of $\downarrow$.  Consider the following commutative diagram:
    \[
        \begin{tikzcd}
            \Hom_{\Qcat(z)}(X,Y) \arrow[r, "\cong"] \arrow[d, "\bF_n"]
            & \Hom_{\Qcat(z)}(\downarrow^{s_X} \uparrow^{r_X}, \uparrow^{r_Y} \downarrow^{s_Y}) \arrow[d, "\bF_n"] \arrow[r, "\cong"]
            & \Hom_{\Qcat(z)}(\uparrow^{r_X+s_Y}, \uparrow^{r_Y+s_X}) \arrow[d, "\bF_n"]
            \\
            \Hom_{U_q} \big( V^X,V^Y \big) \arrow[r, "\cong"]
            & \Hom_{U_q} \big( V^{\downarrow^{s_X} \uparrow^{r_X}}, V^{\uparrow^{r_Y} \downarrow^{s_Y}} \big) \arrow[r, "\cong"]
            & \Hom_{U_q} \big( V^{\otimes (r_X+s_Y)}, V^{\otimes (r_Y+s_X)} \big)
        \end{tikzcd}
    \]
    The top-left horizontal map is given by composing on the top and bottom of diagrams with $\negleftcross$ to move $\uparrow$'s on the top to the left and $\uparrow$'s on the bottom to the right.  The bottom-left horizontal map is given analogously, using $\bF_n(\negleftcross)$.  The right horizontal maps are the usual isomorphisms that hold in any rigid monoidal supercategory.  In particular, the top-right horizontal map is the $\C(q)$-linear isomorphism given on diagrams by
    
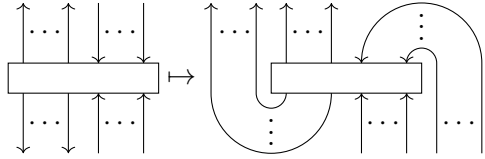
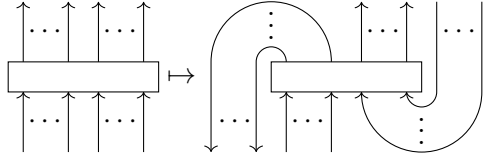
\[
        \begin{tikzpicture}[centerzero]
            \draw (-1,-0.2) rectangle (1,0.2);
            \draw[->] (-0.8,0.2) -- (-0.8,1);
            \node at (-0.48,0.6) {$\cdots$};
            \draw[->] (-0.2,0.2) -- (-0.2,1);
            \draw[<-] (0.2,0.2) -- (0.2,1);
            \draw[<-] (0.8,0.2) -- (0.8,1);
            \node at (0.52,0.6) {$\cdots$};
            \draw[->] (-0.8,-0.2) -- (-0.8,-1);
            \draw[->] (-0.2,-0.2) -- (-0.2,-1);
            \node at (-0.48,-0.6) {$\cdots$};
            \draw[<-] (0.2,-0.2) -- (0.2,-1);
            \draw[<-] (0.8,-0.2) -- (0.8,-1);
            \node at (0.52,-0.6) {$\cdots$};
        \end{tikzpicture}
        \mapsto
        \begin{tikzpicture}[centerzero]
            \draw (-1,-0.2) rectangle (1,0.2);
            \draw[->] (-0.8,0.2) -- (-0.8,1);
            \node at (-0.48,0.6) {$\cdots$};
            \draw[->] (-0.2,0.2) -- (-0.2,1);
            \draw[<-] (0.2,0.2) arc(180:0:0.8) -- (1.8,-1);
            \draw[<-] (0.8,0.2) arc(180:0:0.2) -- (1.2,-1);
            \node at (1,0.8) {$\vdots$};
            \draw[->] (-0.8,-0.2) arc(360:180:0.2) -- (-1.2,1);
            \draw[->] (-0.2,-0.2) arc(360:180:0.8) -- (-1.8,1);
            \node at (-1.45,0.6) {$\cdots$};
            \node at (-1,-0.6) {$\vdots$};
            \draw[<-] (0.2,-0.2) -- (0.2,-1);
            \draw[<-] (0.8,-0.2) -- (0.8,-1);
            \node at (0.52,-0.6) {$\cdots$};
            \node at (1.53,-0.6) {$\cdots$};
        \end{tikzpicture}
    \]
    with inverse
    \[
        \begin{tikzpicture}[centerzero]
            \draw (-1,-0.2) rectangle (1,0.2);
            \draw[->] (-0.8,0.2) -- (-0.8,1);
            \node at (-0.48,0.6) {$\cdots$};
            \draw[->] (-0.2,0.2) -- (-0.2,1);
            \draw[->] (0.2,0.2) -- (0.2,1);
            \draw[->] (0.8,0.2) -- (0.8,1);
            \node at (0.52,0.6) {$\cdots$};
            \draw[<-] (-0.8,-0.2) -- (-0.8,-1);
            \draw[<-] (-0.2,-0.2) -- (-0.2,-1);
            \node at (-0.48,-0.6) {$\cdots$};
            \draw[<-] (0.2,-0.2) -- (0.2,-1);
            \draw[<-] (0.8,-0.2) -- (0.8,-1);
            \node at (0.52,-0.6) {$\cdots$};
        \end{tikzpicture}
        \mapsto
        \begin{tikzpicture}[centerzero]
            \draw (-1,-0.2) rectangle (1,0.2);
            \draw[->] (-0.8,0.2) arc(0:180:0.2) -- (-1.2,-1);
            \node at (-1,0.8) {$\vdots$};
            \draw[->] (-0.2,0.2) arc(0:180:0.8) -- (-1.8,-1);
            \draw[->] (0.2,0.2) -- (0.2,1);
            \draw[->] (0.8,0.2) -- (0.8,1);
            \node at (0.52,0.6) {$\cdots$};
            \draw[<-] (-0.8,-0.2) -- (-0.8,-1);
            \draw[<-] (-0.2,-0.2) -- (-0.2,-1);
            \node at (-0.48,-0.6) {$\cdots$};
            \node at (-1.45,-0.6) {$\cdots$};
            \draw[<-] (0.2,-0.2) arc(180:360:0.8) -- (1.8,1);
            \draw[<-] (0.8,-0.2) arc(180:360:0.2) -- (1.2,1);
            \node at (1,-0.6) {$\vdots$};
            \node at (1.53,0.6) {$\cdots$};
        \end{tikzpicture}
    \]
    where the rectangle denotes some diagram.

    Since all the horizontal maps are isomorphisms, it suffices to show that the rightmost vertical map has the desired properties.  Thus, we must show that the map
    \begin{equation} \label{boing}
        \bF_n \colon \Hom_{\Qcat(z)}(\uparrow^r, \uparrow^s)
        \to \Hom_{U_q}(V^{\otimes r}, V^{\otimes s})
    \end{equation}
    is surjective for all $r,s \in \N$, and that it is injective when $r+s \le 2n$.  We first consider the case where $r \ne s$.  Let $x = u_{11} u_{22} \dotsm u_{nn}$ be the central element of \cref{donut}\cref{donut2}.  Since $u_{ii}$ acts on $V$ by $\Theta_{ii}$, it follows from \cref{tavern} that $x$ acts on $V$ as multiplication by $q$.  By \cref{snooze}, we have $\Delta(x) = x \otimes x$.  Thus $x$ acts on $V^{\otimes r}$ as multiplication by $q^r$.  Since $x$ is central, this implies that $\Hom_{U_q}(V^{\otimes r}, V^{\otimes s}) = 0$.  Since we also have $\Hom_{\Qcat(z)}(\uparrow^r, \uparrow^s) = 0$ in this case, by \cref{nonaffinespan}, the map \cref{boing} is an isomorphism when $r \ne s$.

    Now suppose $r=s$, and consider the composite
    \begin{equation} \label{hungry}
        \HC_r(z) \xrightarrowdbl{\varphi} \End_{\Qcat(z)}(\uparrow^r) \xrightarrow{\bF_n} \End_{U_q}(V^{\otimes r}),
    \end{equation}
    where $\varphi$ is the surjective homomorphism of \cref{walled} with $s=0$.  This composite is precisely the map of \cite[Th.~5.2]{Ols92}.  Surjectivity is asserted, without proof, in \cite[Th.~5.3]{Ols92}.  For the more precise statement, with proof, that this map is also an isomorphism when $\# X + \# Y = 2r \le 2n$, see \cite[Th.~3.28]{BGJKW16}.  It follows that $\bF_n \colon \End_{\Qcat(z)}(\uparrow^r) \to \End_{U_q}(V^{\otimes r})$ is always surjective, and that it is an isomorphism when $r \le n$, as desired.
\end{proof}

Note that $\# X + \# Y$ is twice the number of strands in any string diagram representing a morphism in $\Qcat(z)$ from $X$ to $Y$.  Thus, \cref{full} asserts that $\bF_n$ induces an isomorphism on morphism spaces whenever the number of strands is less than or equal to $n$.

We now loosen our assumption on the ground field.  For the remainder of this section
\begin{center}
    \emph{$\kk$ is an arbitrary commutative ring of characteristic not equal to two, and $z \in \kk$.}
\end{center}
We can now improve \cref{nonaffinespan}.

\begin{theo} \label{basisthm}
    For any objects $X,Y$ of $\Qcat(z)$, the $\kk$-supermodule $\Hom_{\Qcat(z)}(X,Y)$ is free with basis $B_\bullet(X,Y)$.
\end{theo}

\begin{proof}
    In light of \cref{nonaffinespan}, it remains to prove that the elements of $B_\bullet(X,Y)$ are linearly independent.  We first prove this when $\kk = \C(z)$.  Consider the superalgebra homomorphisms \cref{hungry}.  By \cite[Th.~3.28]{BGJKW16}, the composite $\bF_n \circ \varphi$, which is the map denoted $\rho_{n,q}^{r,0}$ there, is an isomorphism for $n \ge r$.  Since the map $\varphi$ is independent of $n$, it follows that $\varphi$ is injective, and hence an isomorphism.  Thus,
    \[
        \dim_\kk \End_{\Qcat(z)}(\uparrow^r) = \dim_\kk \HC_r(z) = r! 2^r,
    \]
    where the last equality is \cite[Prop.~2.1]{JN99}.  (The statement in \cite[Prop.~2.1]{JN99} is over the field $\C(q)$, with $z=q-q^{-1}$, but the proof is the same over $\C(z)$.)  Now suppose that $X$ (respectively, $Y$) is a tensor product of $r_X$ (respectively, $r_Y$) copies of $\uparrow$ and $s_X$ (respectively, $s_Y$) copies of $\downarrow$.  As in the proof of \cref{full}, we have a linear isomorphism
    \[
        \Hom_{\Qcat(z)}(X,Y) \cong \Hom_{\Qcat(z)}(\uparrow^{r_X+s_Y}, \uparrow^{r_Y + s_X}).
    \]
    Thus
    \[
        \dim_\kk \End_{\Qcat(z)}(X,Y) =
        \begin{cases}
            k!2^k & \text{if } k := r_X + s_Y = r_Y + s_X, \\
            0 & \text{otherwise}.
        \end{cases}
    \]
    This dimension is equal to the number of elements of $B_\bullet(X,Y)$.  Indeed, there are $k!$ $(X,Y)$-matchings and $2^k$ ways of adding Clifford tokens to the strings in a positive reduced lift.  It follows that $B_\bullet(X,Y)$ is a basis for $\End_{\Qcat(z)}(\uparrow^r)$.  This completes the proof of \cref{basisthm} for $\kk = \C(z)$.

    To complete the proof over more general base rings, note that $\C(q)$ is a free $\Z[z]$-module, with $z$ acting as $q-q^{-1}$.  Thus, any linear dependence relation over $\Z[z]$ yields a linear dependence relation over $\C(q)$ after extending scalars.  Therefore, it follows from the above that the elements of $B_\bullet(\uparrow^r, \uparrow^r)$ are a basis over $\Z[z]$ and hence, by extension of scalars, over any commutative ring $\kk$ of characteristic not equal to two and $z \in \kk$.
\end{proof}

\begin{rem}
    Taking $z=0$ in \cref{basisthm} recovers the basis theorem \cite[Th~3.4]{BCK19} for the oriented Brauer--Clifford supercategory; see \cref{degenerate,imagine}.
\end{rem}

\begin{cor} \label{dimension}
    Let $X = X_1 \otimes \dotsb \otimes X_r$ and $Y = Y_1 \otimes \dotsb \otimes Y_s$ be objects of $\Qcat(z)$ for $X_i,Y_j \in \{\uparrow, \downarrow\}$.  Then $\Hom_{\Qcat(z)}(X,Y) = 0$ if the cardinalities of the sets \cref{chain} are not equal.  If they are equal (which implies that $r+s$ is even), then $\Hom_{\Qcat(z)}(X,Y)$ is a free $\kk$-supermodule with even and odd parts each of rank $k!2^{k-1}$, where $k=\frac{r+s}{2}$ is the number of strings in the elements of $B(X,Y)$.
\end{cor}

\begin{proof}
    This follows immediately from \cref{basisthm}.  If the sets \cref{chain} have the same cardinality, then the number of $(X,Y)$-matchings is $k!$, and there are $2^k$ ways of adding Clifford tokens to the strings, half of which yield even string diagrams.
\end{proof}

\begin{cor} \label{walleder}
    The homomorphism of \cref{walled} is an isomorphism of associative superalgebras
    \[
        \BC_{r,s}(z) \xrightarrow{\cong} \End_{\Qcat(z)}(\uparrow^r \downarrow^s).
    \]
\end{cor}

\begin{proof}
    By \cref{walled}, the map is surjective.  When $\kk = \C(q)$, one can then conclude that it is an isomorphism by comparing dimensions.  Indeed, by \cref{dimension} and \cite[Cor.~3.25]{BGJKW16}, we have
    \[
        \dim_{\C(q)} \BC_{r,s}(z)
        = (r+s)! 2^{r+s}
        = \dim_{\C(q)} \End_{\Qcat(z)}(\uparrow^r \downarrow^s).
    \]
    More generally, one can argue as in Step 1 of the proof of \cite[Th.~5.1]{JK14} to show that $\BC_{r,s}(z)$ has a spanning set that maps to the basis $B_\bullet(\uparrow^r \downarrow^s,\uparrow^r \downarrow^s)$ of $\End_{\Qcat(z)}(\uparrow^r \downarrow^s)$.  It follows that this spanning set is linearly independent, hence a basis of $\BC_{r,s}(z)$.
\end{proof}

\begin{cor} \label{window}
    Let $X = X_1 \otimes \dotsb \otimes X_m$ be an object of $\Qcat(z)$ for $X_i \in \{\uparrow,\downarrow\}$.  If $r$ is the number of $i \in \{1,\dotsc,m\}$ such that $X_i =\, \uparrow$, then $\End_{\Qcat(z)}(X) \cong \BC_{r,m-r}(z)$ as associative superalgebras.
\end{cor}

\begin{proof}
    It follows from the third and fourth equalities in \cref{braid} that $\negrightcross \colon \uparrow \otimes \downarrow\ \to\ \downarrow \otimes \uparrow$ is an isomorphism with inverse $\posleftcross$.  Hence $X \cong\, \uparrow^{\otimes r} \otimes \downarrow^{\otimes (m-r)}$, and so the result follows from \cref{walleder}.
\end{proof}

As a special case of \Cref{window}, we have an isomorphism of associative superalgebras $\BC_{r,s}(z) \cong \End_{\Qcat(z)}(\downarrow^s \uparrow^r)$.  This recovers \cite[Th.~4.19]{BGJKW16}, which describes the walled Brauer--Clifford superalgebras in terms of \emph{bead tangle superalgebras}.  When converting string diagrams representing endomorphisms in $\Qcat(z)$ to the bead tangle diagrams of \cite[\S4]{BGJKW16}, one should forget the orientations of strings, and then rotate diagrams by $180\degree$.  This transformation is needed since the convention in \cite[\S4]{BGJKW16} for composing diagrams is the opposite of ours.

\begin{rem} \label{dungeon}
    The full monoidal sub-supercategory of $U_q\smod$ generated by $V$ and $V^*$ is not semisimple.  Indeed, it follows from \cref{full} that, for $n \ge 2$, $\End_{U_q}(V \otimes V^*)$ is isomorphic to $\End_{\Qcat(z)}(\uparrow \downarrow)$, which, by \cref{basisthm}, has basis
    \[
        \begin{tikzpicture}[centerzero]
            \draw[->] (0,-0.4) -- (0,0.4);
            \draw[<-] (0.3,-0.4) -- (0.3,0.4);
        \end{tikzpicture}
        \ ,\quad
        \begin{tikzpicture}[centerzero]
            \draw[->] (0,-0.4) -- (0,0.4);
            \draw[<-] (0.3,-0.4) -- (0.3,0.4);
            \token{0,0};
            \token{0.3,0};
        \end{tikzpicture}
        \ ,\quad
        \begin{tikzpicture}[centerzero]
            \draw[->] (0,-0.4) -- (0,0.4);
            \draw[<-] (0.3,-0.4) -- (0.3,0.4);
            \token{0,0};
        \end{tikzpicture}
        \ ,\quad
        \begin{tikzpicture}[centerzero]
            \draw[->] (0,-0.4) -- (0,0.4);
            \draw[<-] (0.3,-0.4) -- (0.3,0.4);
            \token{0.3,0};
        \end{tikzpicture}
        \ ,\quad
        \begin{tikzpicture}[centerzero]
            \draw[->] (0,-0.4) -- (0,-0.2) arc(180:0:0.15) -- (0.3,-0.4);
            \draw[<-] (0,0.4) -- (0,0.2) arc(180:360:0.15) -- (0.3,0.4);
        \end{tikzpicture}
        \ ,\quad
        \begin{tikzpicture}[centerzero]
            \draw[->] (0,-0.4) -- (0,-0.2) arc(180:0:0.15) -- (0.3,-0.4);
            \draw[<-] (0,0.4) -- (0,0.2) arc(180:360:0.15) -- (0.3,0.4);
            \token{0.3,-0.23};
            \token{0,0.23};
        \end{tikzpicture}
        \ ,\quad
        \begin{tikzpicture}[centerzero]
            \draw[->] (0,-0.4) -- (0,-0.2) arc(180:0:0.15) -- (0.3,-0.4);
            \draw[<-] (0,0.4) -- (0,0.2) arc(180:360:0.15) -- (0.3,0.4);
            \token{0.3,-0.23};
        \end{tikzpicture}
        \ ,\quad
        \begin{tikzpicture}[centerzero]
            \draw[->] (0,-0.4) -- (0,-0.2) arc(180:0:0.15) -- (0.3,-0.4);
            \draw[<-] (0,0.4) -- (0,0.2) arc(180:360:0.15) -- (0.3,0.4);
            \token{0,0.23};
        \end{tikzpicture}
        \ .
    \]
    By the last equalities in \cref{venom3,iron}, the span of the last four diagrams above is a nilpotent ideal.  Thus, $\End_{U_q}(V \otimes V^*)$ is not semisimple.  Note, however, that the full monoidal sub-supercategory of $U_q\smod$ generated by $V$ \emph{is} semisimple; see \cite[Th.~6.5]{GJKK10}.
\end{rem}

\section{The chiral braiding\label{sec:chiral}}

This section is the start of the second part of the current paper.  Our goal is to define and study an affine version of the quantum isomeric supercategory.  For braided monoidal supercategories, there is a general affinization procedure; see \cite{MS21}.  However, the supercategory $\Qcat(z)$ is \emph{not} braided since the Clifford dots do not slide through crossings both ways.  This corresponds, under the incarnation superfunctor, to the fact that $U_q$ is not a quasitriangular Hopf superalgebra.  In this section we discuss a \emph{chiral braiding}, which is like a braiding but only natural in one argument.  We begin this section with the assumption that $\kk$ is an arbitrary commutative ring of characteristic not equal to two, and $z \in \kk$.

\begin{defin}
    Let $\Wcat(z)$ be the strict monoidal supercategory obtained from $\Qcat(z)$ by adjoining an additional generating object $\redobject$ and two additional even morphisms
    \[
        \posupredcross \colon \redobject\ \otimes \uparrow\ \to\ \uparrow \otimes\ \redobject\, ,\qquad
        \negupredcross \colon \uparrow \otimes\ \redobject\ \to\ \redobject\ \otimes \uparrow,
    \]
    subject to the relations
    \begin{equation} \label{rouge}
        \begin{tikzpicture}[centerzero]
            \draw[->] (0.2,-0.4) to[out=135,in=down] (-0.15,0) to[out=up,in=225] (0.2,0.4);
            \draw[wipe] (-0.2,-0.4) to[out=45,in=down] (0.15,0) to[out=up,in=-45] (-0.2,0.4);
            \draw[thick,red] (-0.2,-0.4) to[out=45,in=down] (0.15,0) to[out=up,in=-45] (-0.2,0.4);
        \end{tikzpicture}
        \ =\
        \begin{tikzpicture}[centerzero]
            \draw[thick,red] (-0.2,-0.4) -- (-0.2,0.4);
            \draw[->] (0.2,-0.4) -- (0.2,0.4);
        \end{tikzpicture}
        \ ,\qquad
        \begin{tikzpicture}[centerzero]
            \draw[->] (-0.2,-0.4) to[out=45,in=down] (0.15,0) to[out=up,in=-45] (-0.2,0.4);
            \draw[wipe] (0.2,-0.4) to[out=135,in=down] (-0.15,0) to[out=up,in=225] (0.2,0.4);
            \draw[thick,red] (0.2,-0.4) to[out=135,in=down] (-0.15,0) to[out=up,in=225] (0.2,0.4);
        \end{tikzpicture}
        \ =\
        \begin{tikzpicture}[centerzero]
            \draw[->] (-0.2,-0.4) -- (-0.2,0.4);
            \draw[thick,red] (0.2,-0.4) -- (0.2,0.4);
        \end{tikzpicture}
        \ ,\qquad
        \begin{tikzpicture}[centerzero]
            \draw[<-] (0.2,-0.4) to[out=135,in=down] (-0.15,0) to[out=up,in=225] (0.2,0.4);
            \draw[wipe] (-0.2,-0.4) to[out=45,in=down] (0.15,0) to[out=up,in=-45] (-0.2,0.4);
            \draw[thick,red] (-0.2,-0.4) to[out=45,in=down] (0.15,0) to[out=up,in=-45] (-0.2,0.4);
        \end{tikzpicture}
        \ =\
        \begin{tikzpicture}[centerzero]
            \draw[thick,red] (-0.2,-0.4) -- (-0.2,0.4);
            \draw[<-] (0.2,-0.4) -- (0.2,0.4);
        \end{tikzpicture}
        \ ,\qquad
        \begin{tikzpicture}[centerzero]
            \draw[<-] (-0.2,-0.4) to[out=45,in=down] (0.15,0) to[out=up,in=-45] (-0.2,0.4);
            \draw[wipe] (0.2,-0.4) to[out=135,in=down] (-0.15,0) to[out=up,in=225] (0.2,0.4);
            \draw[thick,red] (0.2,-0.4) to[out=135,in=down] (-0.15,0) to[out=up,in=225] (0.2,0.4);
        \end{tikzpicture}
        \ =\
        \begin{tikzpicture}[centerzero]
            \draw[<-] (-0.2,-0.4) -- (-0.2,0.4);
            \draw[thick,red] (0.2,-0.4) -- (0.2,0.4);
        \end{tikzpicture}
        \ ,\qquad
        \begin{tikzpicture}[centerzero]
            \draw[->] (0.4,-0.4) -- (-0.4,0.4);
            \draw[wipe] (0,-0.4) to[out=135,in=down] (-0.32,0) to[out=up,in=225] (0,0.4);
            \draw[->] (0,-0.4) to[out=135,in=down] (-0.32,0) to[out=up,in=225] (0,0.4);
            \draw[wipe] (-0.4,-0.4) -- (0.4,0.4);
            \draw[thick,red] (-0.4,-0.4) -- (0.4,0.4);
        \end{tikzpicture}
        =
        \begin{tikzpicture}[centerzero]
            \draw[->] (0.4,-0.4) -- (-0.4,0.4);
            \draw[wipe] (0,-0.4) to[out=45,in=down] (0.32,0) to[out=up,in=-45] (0,0.4);
            \draw[->] (0,-0.4) to[out=45,in=down] (0.32,0) to[out=up,in=-45] (0,0.4);
            \draw[wipe] (-0.4,-0.4) -- (0.4,0.4);
            \draw[thick,red] (-0.4,-0.4) -- (0.4,0.4);
        \end{tikzpicture}
        \ ,
    \end{equation}
    where
    \begin{equation} \label{reddown}
        \posdownredcross
        :=
        \begin{tikzpicture}[centerzero]
            \draw[->] (-0.4,0.3) -- (-0.4,0.1) to[out=down,in=left] (-0.2,-0.2) to[out=right,in=left] (0.2,0.2) to[out=right,in=up] (0.4,-0.1) -- (0.4,-0.3);
            \draw[wipe] (0.2,-0.3) \braidup (-0.2,0.3);
            \draw[thick,red] (0.2,-0.3) \braidup (-0.2,0.3);
        \end{tikzpicture}
        \qquad \text{and} \qquad
        \negdownredcross
        :=
        \begin{tikzpicture}[centerzero]
            \draw[->] (0.4,0.3) -- (0.4,0.1) to[out=down,in=right] (0.2,-0.2) to[out=left,in=right] (-0.2,0.2) to[out=left,in=up] (-0.4,-0.1) -- (-0.4,-0.3);
            \draw[wipe] (-0.2,-0.3) \braidup (0.2,0.3);
            \draw[thick,red] (-0.2,-0.3) \braidup (0.2,0.3);
        \end{tikzpicture}
        \ .
    \end{equation}
\end{defin}

Note that we do \emph{not} have morphisms corresponding to a red strand passing \emph{under} a black strand.  We also do not have red cups or caps.

\begin{lem}
    The following relations hold in $\Wcat(z)$:
    \begin{equation} \label{redslip}
        \begin{tikzpicture}[centerzero]
            \draw[<-] (-0.2,-0.3) -- (-0.2,-0.1) arc(180:0:0.2) -- (0.2,-0.3);
            \draw[wipe] (-0.3,0.3) \braiddown (0,-0.3);
            \draw[thick,red] (-0.3,0.3) \braiddown (0,-0.3);
        \end{tikzpicture}
        =
        \begin{tikzpicture}[centerzero]
            \draw[<-] (-0.2,-0.3) -- (-0.2,-0.1) arc(180:0:0.2) -- (0.2,-0.3);
            \draw[wipe] (0.3,0.3) \braiddown (0,-0.3);
            \draw[thick,red] (0.3,0.3) \braiddown (0,-0.3);
        \end{tikzpicture}
        \ ,\qquad
        \begin{tikzpicture}[centerzero]
            \draw[<-] (-0.2,0.3) -- (-0.2,0.1) arc(180:360:0.2) -- (0.2,0.3);
            \draw[wipe] (-0.3,-0.3) \braidup (0,0.3);
            \draw[thick,red] (-0.3,-0.3) \braidup (0,0.3);
        \end{tikzpicture}
        =
        \begin{tikzpicture}[centerzero]
            \draw[<-] (-0.2,0.3) -- (-0.2,0.1) arc(180:360:0.2) -- (0.2,0.3);
            \draw[wipe] (0.3,-0.3) \braidup (0,0.3);
            \draw[thick,red] (0.3,-0.3) \braidup (0,0.3);
        \end{tikzpicture}
        \ ,\qquad
        \begin{tikzpicture}[centerzero]
            \draw[->] (-0.2,-0.3) -- (-0.2,-0.1) arc(180:0:0.2) -- (0.2,-0.3);
            \draw[wipe] (-0.3,0.3) \braiddown (0,-0.3);
            \draw[thick,red] (-0.3,0.3) \braiddown (0,-0.3);
        \end{tikzpicture}
        =
        \begin{tikzpicture}[centerzero]
            \draw[->] (-0.2,-0.3) -- (-0.2,-0.1) arc(180:0:0.2) -- (0.2,-0.3);
            \draw[wipe] (0.3,0.3) \braiddown (0,-0.3);
            \draw[thick,red] (0.3,0.3) \braiddown (0,-0.3);
        \end{tikzpicture}
        \ ,\qquad
        \begin{tikzpicture}[centerzero]
            \draw[->] (-0.2,0.3) -- (-0.2,0.1) arc(180:360:0.2) -- (0.2,0.3);
            \draw[wipe] (-0.3,-0.3) \braidup (0,0.3);
            \draw[thick,red] (-0.3,-0.3) \braidup (0,0.3);
        \end{tikzpicture}
        =
        \begin{tikzpicture}[centerzero]
            \draw[->] (-0.2,0.3) -- (-0.2,0.1) arc(180:360:0.2) -- (0.2,0.3);
            \draw[wipe] (0.3,-0.3) \braidup (0,0.3);
            \draw[thick,red] (0.3,-0.3) \braidup (0,0.3);
        \end{tikzpicture}
        \ .
    \end{equation}
\end{lem}

\begin{proof}
    We compute
    \[
        \begin{tikzpicture}[centerzero]
            \draw[<-] (-0.2,-0.3) -- (-0.2,-0.1) arc(180:0:0.2) -- (0.2,-0.3);
            \draw[wipe] (-0.3,0.3) \braiddown (0,-0.3);
            \draw[thick,red] (-0.3,0.3) \braiddown (0,-0.3);
        \end{tikzpicture}
        \overset{\cref{reddown}}{=}
        \begin{tikzpicture}[centerzero]
            \draw[->] (0.8,-0.3) -- (0.8,-0.1) to[out=up,in=right] (0.6,0.2) to[out=left,in=right] (0.2,-0.2) to[out=left,in=right] (-0.2,0.2) to[out=left,in=up] (-0.4,-0.1) -- (-0.4,-0.3);
            \draw[wipe] (-0.2,-0.3) \braidup (0.2,0.3);
            \draw[thick,red] (-0.2,-0.3) \braidup (0.2,0.3);
        \end{tikzpicture}
        \overset{\cref{venom1}}{=}
        \begin{tikzpicture}[centerzero]
            \draw[<-] (-0.2,-0.3) -- (-0.2,-0.1) arc(180:0:0.2) -- (0.2,-0.3);
            \draw[wipe] (0.3,0.3) \braiddown (0,-0.3);
            \draw[thick,red] (0.3,0.3) \braiddown (0,-0.3);
        \end{tikzpicture}
        \ .
    \]
    The proofs of the remaining equalities are analogous.
\end{proof}

\begin{prop} \label{scurry}
    In $\Wcat(z)$, we have
    \begin{equation} \label{boxslide}
        \begin{tikzpicture}[anchorbase]
            \draw (0.3,0) -- (0.3,-0.8) \braiddown (0.6,-1.6);
            \draw (-0.3,0) -- (-0.3,-0.8) \braiddown (0,-1.6);
            \filldraw[draw=black,fill=white] (0.4,-0.4) rectangle (-0.4,-0.8);
            \node at (0,-0.6) {\strandlabel{f}};
            \node at (0.05,-0.2) {$\cdots$};
            \node at (0.05,-1) {$\cdots$};
            \draw[wipe] (0.6,0) -- (0.6,-0.8) \braiddown (-0.3,-1.6);
            \draw[thick,red] (0.6,0) -- (0.6,-0.8) \braiddown (-0.3,-1.6);
        \end{tikzpicture}
        =
        \begin{tikzpicture}[anchorbase]
            \draw (-0.3,0) -- (-0.3,0.8) \braidup (-0.6,1.6);
            \draw (0.3,0) -- (0.3,0.8) \braidup (0,1.6);
            \filldraw[draw=black,fill=white] (-0.4,0.4) rectangle (0.4,0.8);
            \node at (0,0.6) {\strandlabel{f}};
            \node at (0.05,0.2) {$\cdots$};
            \node at (0,1) {$\cdots$};
            \draw[wipe] (-0.6,0) -- (-0.6,0.8) \braidup (0.3,1.6);
            \draw[thick,red] (-0.6,0) -- (-0.6,0.8) \braidup (0.3,1.6);
        \end{tikzpicture}
        \qquad \text{and} \qquad
        \begin{tikzpicture}[anchorbase]
            \draw (0.3,0) -- (0.3,0.8) \braidup (0.6,1.6);
            \draw (-0.3,0) -- (-0.3,0.8) \braidup (0,1.6);
            \filldraw[draw=black,fill=white] (0.4,0.4) rectangle (-0.4,0.8);
            \node at (0,0.6) {\strandlabel{f}};
            \node at (0,0.2) {$\cdots$};
            \node at (0.05,1) {$\cdots$};
            \draw[wipe] (0.6,0) -- (0.6,0.8) \braidup (-0.3,1.6);
            \draw[thick,red] (0.6,0) -- (0.6,0.8) \braidup (-0.3,1.6);
        \end{tikzpicture}
        =
        \begin{tikzpicture}[anchorbase]
            \draw (-0.3,0) -- (-0.3,-0.8) \braiddown (-0.6,-1.6);
            \draw (0.3,0) -- (0.3,-0.8) \braiddown (0,-1.6);
            \filldraw[draw=black,fill=white] (-0.4,-0.4) rectangle (0.4,-0.8);
            \node at (0,-0.6) {\strandlabel{f}};
            \node at (0.05,-0.2) {$\cdots$};
            \node at (0.05,-1) {$\cdots$};
            \draw[wipe] (-0.6,0) -- (-0.6,-0.8) \braiddown (0.3,-1.6);
            \draw[thick,red] (-0.6,0) -- (-0.6,-0.8) \braiddown (0.3,-1.6);
        \end{tikzpicture}
        \ ,
    \end{equation}
    where $f$ is any string diagram in $\Qcat(z)$ not containing Clifford tokens.
\end{prop}

\begin{proof}
    First note that the second equality in \cref{boxslide} follows from the first after composing on the top and bottom with the appropriate red-black crossings and using the first four relations in \cref{rouge}.  Therefore, we prove only the first equality.  It suffices to prove it for $f$ equal to each of the generating morphisms $\posupcross$, $\negupcross$, $\negrightcross$, $\leftcap$, $\leftcup$.  Since
    \[
        \negrightcross =
        \begin{tikzpicture}[centerzero]
            \draw[<-] (0.4,-0.3) -- (0.4,-0.1) to[out=up,in=right] (0.2,0.2) to[out=left,in=right] (-0.2,-0.2) to[out=left,in=down] (-0.4,0.1) -- (-0.4,0.3);
            \draw[wipe] (0.2,-0.3) \braidup (-0.2,0.3);
            \draw[->] (0.2,-0.3) \braidup (-0.2,0.3);
        \end{tikzpicture}
        \ ,
    \]
    it is also enough to show it holds for $f \in \{\posupcross, \negupcross, \leftcap, \leftcup, \rightcap, \rightcup\}$.

    For $f = \posupcross$, the first equality in \cref{boxslide} follows from the last relation in \cref{rouge}.  Composing both sides of the last relation in \cref{rouge} on the top with $\negupcross\, \upstrand$ and on the bottom with $\upstrand\, \negupcross$ shows that the first equality in \cref{boxslide} also holds with $f = \negupcross$.

    To prove that the first equality in \cref{boxslide} holds with $f = \leftcap$ and $f = \leftcup$, we must show that
    \begin{equation} \label{camel}
        \begin{tikzpicture}[anchorbase]
            \draw[->] (0.3,-0.4) \braidup (0,0) arc(0:180:0.15) \braiddown (0,-0.4);
            \draw[wipe] (-0.3,-0.4) \braidup (0.3,0.4);
            \draw[thick,red] (-0.3,-0.4) \braidup (0.3,0.4);
        \end{tikzpicture}
        =
        \begin{tikzpicture}[anchorbase]
            \draw[->] (0.3,-0.4) -- (0.3,-0.2) arc(0:180:0.15) -- (0,-0.4);
            \draw[thick,red] (-0.3,-0.4) -- (-0.3,0.4);
        \end{tikzpicture}
        \qquad \text{and} \qquad
        \begin{tikzpicture}[anchorbase,rotate=180]
            \draw[<-] (0.3,-0.4) \braidup (0,0) arc(0:180:0.15) \braiddown (0,-0.4);
            \draw[wipe] (-0.3,-0.4) \braidup (0.3,0.4);
            \draw[thick,red] (-0.3,-0.4) \braidup (0.3,0.4);
        \end{tikzpicture}
        =
        \begin{tikzpicture}[anchorbase,rotate=180]
            \draw[<-] (0.3,-0.4) -- (0.3,-0.2) arc(0:180:0.15) -- (0,-0.4);
            \draw[thick,red] (-0.3,-0.4) -- (-0.3,0.4);
        \end{tikzpicture}
        \ .
    \end{equation}
    The first relation in \cref{camel} follows from the first relation in \cref{redslip} after composing on the bottom with $\posdownredcross\, \upstrand$ and using third relation in \cref{rouge}.  Similarly the second relation in \cref{camel} follows from the second relation in \cref{redslip} after composing on the top with $\upstrand\, \posdownredcross$ and using the fourth relation in \cref{rouge}.  The proofs for $f = \rightcap$ and $f = \rightcup$ are analogous, using the last two equalities in \cref{redslip}.
\end{proof}

In the remainder of this section we will be discussing connections to $U_q\smod$.  Thus, we now begin supposing that $\kk = \C(q)$ and $z = q-q^{-1}$.  Recall the definition of $L$ from \cref{Ldef}.  For a finite-dimensional $U_q$-supermodule $M$, we let $\rho_M \colon U_q \to \End_\kk(M)$ denote the corresponding representation.  We then define
\[
    L_M
    = (\rho_M \otimes 1_V)(L)
    = \sum_{\substack{i, j \in \tI \\ i \le j}} \rho_M(u_{ij}) \otimes E_{ij}
    \in \End_\kk(M) \otimes \End_\kk(V).
\]
In particular, we have $L_V = \Theta$; see \cref{rho}.

\begin{prop} \label{coin}
    For any $M \in U_q\smod$, the map
    \begin{equation} \label{burrito}
        T_{MV} := \flip \circ L_M \colon M \otimes V \to V \otimes M
    \end{equation}
    is an isomorphism of $U_q$-supermodules.  Furthermore, for all $f \in \Hom_{U_q}(M,N)$, we have
    \begin{equation} \label{gringo}
        T_{NV} \circ (f \otimes 1_V)
        = (1_V \otimes f) \circ T_{MV}.
    \end{equation}
\end{prop}

\begin{proof}
    It is clear that $T_{MV}$ is invertible, with
    \begin{equation}
        T_{MV}^{-1} = (\rho_M \otimes 1_V)(L^{-1}) \circ \flip
        \overset{\cref{antipode}}{=} \left( \sum_{i \le j} \rho_M(S(u_{ij})) \otimes E_{ij} \right) \circ \flip.
    \end{equation}
    To show that $T_{MV}$ is a homomorphism of $U_q$-supermodules, it suffices to prove that it commutes with the action of $u_{kl}$, $k,l \in \tI$, $k \le l$.  By \cref{dive}, it is enough to show that
    \begin{equation} \label{tadpole}
        \flip^{12} L_M^{12} L_M^{13} L_V^{23}
        = L_V^{13} L_M^{23} \flip^{12} L_M^{12}
        \quad \text{as maps } M \otimes V \otimes V \to V \otimes M \otimes V.
    \end{equation}
    Composing both sides of \cref{tadpole} on the left with the invertible map $\flip^{12}$, and using the fact that $L_V = \Theta$, we see that \cref{tadpole} is equivalent to
    \begin{equation} \label{frog}
        L_M^{12} L_M^{13} \Theta^{23}
        = \Theta^{23} L_M^{13} L_M^{12}
        \quad \text{as maps } M \otimes V \otimes V \to M \otimes V \otimes V,
    \end{equation}
    which follows from the last equality in \cref{QQdef}.

    It remains to prove \cref{gringo}.  For $f \in \Hom_{U_q}(M,N)$, $m \in M$, and $v \in V$, we have
    \begin{align*}
        T_{NV} \circ (f \otimes 1_V)(m \otimes v)
        &= \flip \circ \left( \sum_{i \le j} (-1)^{(\bar{f}+\bar{m})p(i,j)} u_{ij} f(m) \otimes E_{ij} v \right) \\
        &= \flip \circ \left( \sum_{i \le j} (-1)^{\bar{m}p(i,j)} f(u_{ij} m) \otimes E_{ij} v \right)  \\
        &= \flip \circ (f \otimes 1_V) \circ \left( \sum_{i \le j} u_{ij} \otimes E_{ij} \right) (m \otimes v) \\
        &= (1_V \otimes f) \circ T_{MV} (m \otimes v).
        \qedhere
    \end{align*}
\end{proof}

Note that $T_{VV} = T$ (see \cref{Tdef}), and so \cref{coin} is a generalization of \cref{coffee}.  Now, for $i,j \in \tI$, define
\[
    E_{ij}^* \colon V^* \to V^*, \qquad
    f \mapsto (-1)^{p(i,j)\bar{f}} f \circ E_{ij}.
\]
It follows that
\begin{equation} \label{flower}
    E_{ij}^* v_k^* = \delta_{ik} (-1)^{p(i)+p(i)p(j)} v_j^*
    \qquad \text{and} \qquad
    E_{ij}^* E_{kl}^* = \delta_{il} (-1)^{p(i,j)p(k,l)} E_{kj}^*,
\end{equation}
for all $i,j,k,l \in \tI$.

\begin{theo}
    For each $U_q$-supermodule $M$, the superfunctor $\bF_n$ of \cref{bread} extends to a unique monoidal superfunctor $\bF_n^M \colon \Wcat(z) \to U_q\smod$ such that
    \[
        \bF_n^M(\, \redobject\, ) = M
        \qquad \text{and} \qquad
        \bF_n^M(\posupredcross) = T_{MV}.
    \]
    Furthermore, $\bF_n^M(\negupredcross) = T_{MV}^{-1}$,
    \begin{equation} \label{rojo}
        \bF_n^M(\posdownredcross) = \flip \circ \left( \sum_{i \le j} S^{-1}(u_{ij}) \otimes E_{ij}^* \right)
        ,\quad \text{and} \quad
        \bF_n^M(\negdownredcross) = \left( \sum_{i \le j} u_{ij} \otimes E_{ij}^* \right) \circ \flip.
    \end{equation}
\end{theo}

\begin{proof}
    By \cref{bread}, to show that $\bF_n^M$ is well defined, it suffices to show that $\bF_n^M$ respects the relations \cref{rouge}.  First of all, uniqueness of the inverse implies that $\bF_n^M(\negupredcross) = T_{MV}^{-1}$, and then the first two relations in \cref{rouge} are satisfied.

    Next we show, using \cref{reddown}, that the equalities \cref{rojo} must hold.  For $m \in M$ and $k \in \tI$, we have
    \begin{align*}
        \bF_n^M(\posdownredcross) : m \otimes v_k^*
        &\xmapsto{\bF_n(\rightcup) \otimes 1_M \otimes 1_{V^*}}
        \sum_l (-1)^{p(l)} q^{2n-2|l|+1} v_l^* \otimes v_l \otimes m \otimes v_k^*
        \\
        &\xmapsto{1_{V^*} \otimes T_{MV}^{-1} \otimes 1_{V^*}}
        \sum_l \sum_{i \le j} (-1)^{\bar{m}(p(i,j)+p(l))+p(l)} q^{2n-2|l|+1} v_l^* \otimes S(u_{ij}) m \otimes E_{ij} v_l \otimes v_k^*
        \\
        &= \sum_{i \le j} (-1)^{\bar{m}p(i)+p(j)} q^{2n-2|j|+1} v_j^* \otimes S(u_{ij}) m \otimes v_i \otimes v_k^*
        \\
        &\xmapsto{1_{V^*} \otimes 1_M \otimes \bF_n(\rightcap)}
        \sum_{j : j \ge k} (-1)^{\bar{m}p(k)+p(j,k)} q^{2|k|-2|j|} v_j^* \otimes S(u_{kj}) m
        \\
        &\overset{\mathclap{\cref{flower}}}{\underset{\mathclap{\cref{Sinv}}}{=}}\
        \flip \circ \left( \sum_{i \le j} S^{-1}(u_{ij}) \otimes E_{ij}^* \right) (m \otimes v_k^*).
    \end{align*}
    Next, we compute
    \begin{align*}
        \bF_n^M(\negdownredcross) : v_k^* \otimes m
        &\xmapsto{1_{V^*} \otimes 1_M \otimes \coev} \sum_l v_k^* \otimes m \otimes v_l \otimes v_l^*
        \\
        &\xmapsto{1_{V^*} \otimes T_{MV} \otimes 1_{V^*}} \sum_l \sum_{i \le j} (-1)^{\bar{m} p(i,j)} v_k^* \otimes \flip (u_{ij} m \otimes E_{ij} v_l) \otimes v_l^*
        \\
        &= \sum_{i \le j} (-1)^{\bar{m} p(i,j)} v_k^* \otimes \flip (u_{ij} m \otimes v_i) \otimes v_j^*
        \\
        &= \sum_{i \le j} (-1)^{\bar{m}p(j)+p(i)+p(i)p(j)} v_k^* \otimes v_i \otimes u_{ij} m \otimes v_j^*
        \\
        &\xmapsto{\ev \otimes 1_M \otimes 1_{V^*}} \sum_{j : j \ge k} (-1)^{\bar{m}p(j)+p(k)+p(k)p(j)} u_{kj} m \otimes v_j^*
        \\
        &\overset{\mathclap{\cref{flower}}}{=}\ \left( \sum_{i \le j} u_{ij} \otimes E_{ij}^* \right) \circ \flip(v_k^* \otimes m).
    \end{align*}

    Now, for the third relation in \cref{rouge}, we compute
    \begin{multline*}
        \bF_n^M
        \left(
            \begin{tikzpicture}[centerzero]
                \draw[<-] (0.2,-0.4) to[out=135,in=down] (-0.15,0) to[out=up,in=225] (0.2,0.4);
                \draw[wipe] (-0.2,-0.4) to[out=45,in=down] (0.15,0) to[out=up,in=-45] (-0.2,0.4);
                \draw[thick,red] (-0.2,-0.4) to[out=45,in=down] (0.15,0) to[out=up,in=-45] (-0.2,0.4);
            \end{tikzpicture}
        \right)
        = \sum_{i \le j} \sum_{k \le l} (-1)^{p(i,j)p(k,l)} u_{ij} S^{-1}(u_{kl}) \otimes E_{ij}^* E_{kl}^*
        \overset{\cref{flower}}{=} \sum_{k \le i \le j} u_{ij} S^{-1}(u_{ki}) \otimes E_{kj}^*
        \\
        = \sum_{k \le j} S^{-1} \left( \sum_{i : k \le i \le j} u_{ki} S(u_{ij}) \right) \otimes E_{kj}^*
        = \sum_{k \le j} S^{-1}(\varepsilon(u_{kj})) \otimes E_{kj}^*
        \overset{\cref{counit}}{=} \sum_{k \le j} \delta_{kj} \otimes E_{kj}^*
        = 1_M \otimes 1_V.
    \end{multline*}
    \details{
        In the third equality above we use
        \[
            \sum_{k : k \le i \le j} u_{ij} S^{-1}(u_{ki})
            = S^{-1} \left( \sum_{i : k \le i \le j} (-1)^{p(i,j) p(i,k)} u_{ki} S(u_{ij}) \right)
            = S^{-1} \left( \sum_{i : k \le i \le j} u_{ki} S(u_{ij}) \right),
        \]
        since $i$ must either be of the same parity as $j$ or of the same parity as $k$.  In the fourth equality we use the property of a Hopf superalgebra that $\nabla \circ (1 \otimes S) = \varepsilon$, where $\nabla \colon U_q \otimes U_q \to U_q$ is the multiplication map.
    }
    The proof of the fourth relation in \cref{rouge} is analogous.
    \details{
        We compute
        \begin{multline*}
            \flip \circ
            \begin{tikzpicture}[centerzero]
                \draw[<-,black] (-0.2,-0.4) to[out=45,in=down] (0.15,0) to[out=up,in=-45] (-0.2,0.4);
                \draw[wipe] (0.2,-0.4) to[out=135,in=down] (-0.15,0) to[out=up,in=225] (0.2,0.4);
                \draw[thick,red] (0.2,-0.4) to[out=135,in=down] (-0.15,0) to[out=up,in=225] (0.2,0.4);
            \end{tikzpicture}
            \circ \flip
            = \sum_{i \le j} \sum_{k \le l} (-1)^{p(i,j)p(k,l)} S^{-1}(u_{ij}) u_{kl} \otimes E_{ij}^* E_{kl}^*
            \overset{\cref{flower}}{=} \sum_{k \le i \le j} S^{-1}(u_{ij}) u_{ki} \otimes E_{kj}^*
            \\
            = \sum_{k \le j} S^{-1} \left( \sum_{i : k \le i \le j} S(u_{ki}) u_{ij} \right) \otimes E_{kj}^*
            = \sum_{k \le j} S^{-1}(\varepsilon(u_{kj})) \otimes E_{kj}^*
            \overset{\cref{counit}}{=} \sum_{k \le j} \delta_{kj} \otimes E_{kj}^*
            = 1_M \otimes 1_V.
        \end{multline*}
    }

    For the last relation in \cref{rouge}, we compute that $\bF_n^M$ sends the left-hand side to the map $M \otimes V^{\otimes 2} \to V^{\otimes 2} \otimes M$ given by
    \[
       T_{VV}^{12} T_{MV}^{23} T_{MV}^{12}
       = \flip^{12} \Theta^{12} \flip^{23} L^{23} \flip^{12} L^{12} \\
       = \flip^{12} \flip^{23} \flip^{12} \Theta^{23} L^{13} L^{12}.
    \]
    On the other hand, $\bF_n^M$ sends the right-hand side to the map
    \[
        T_{MV}^{23} T_{MV}^{12} T_{VV}^{23}
        = \flip^{23} L^{23} \flip^{12} L^{12} \flip^{23} \Theta^{23}
        = \flip^{23} \flip^{12} \flip^{23} L^{12} L^{13} \Theta^{23}.
    \]
    Then the last relation in \cref{rouge} follows from \cref{QQdef} and the fact that $\flip^{12} \flip^{23} \flip^{12} = \flip^{23} \flip^{12} \flip^{23}$.
\end{proof}

\begin{rem} \label{word}
    There are natural superfunctors $\bF^\uparrow, \bF^\downarrow \colon \Wcat(z) \to \Qcat(z)$ sending $\,\redobject\,$ to $\uparrow$ and $\downarrow$, respectively.  It is then straightforward to very that following diagrams commute:
    \[
        \begin{tikzcd}
            \Wcat(z) \arrow[d, "\bF^\uparrow"'] \arrow[dr, "\bF_n^V"] & \\
            \Qcat(z) \arrow[r, "\bF_n"] & U_q\smod
        \end{tikzcd}
        \qquad \qquad \qquad
        \begin{tikzcd}
            \Wcat(z) \arrow[d, "\bF^\downarrow"'] \arrow[dr, "\bF_n^{V^*}"] & \\
            \Qcat(z) \arrow[r, "\bF_n"] & U_q\smod
        \end{tikzcd}
    \]
    In particular, we have
    \begin{gather*}
        \bF_n(\negrightcross) = T_{VV^*},\quad
        \bF_n(\posleftcross) = T_{VV^*}^{-1},\quad
        \bF_n(\negleftcross) = T_{V^*V},\quad
        \bF_n(\posrightcross) = T_{V^*V}^{-1},\quad
        \\
        \bF_n(\posdowncross) = T_{V^*V^*},\quad
        \bF_n(\negdowncross) = T_{V^*V^*}^{-1}.
    \end{gather*}
\end{rem}

For $M \in U_q\smod$, we will denote the image of a string diagram in $\Wcat(z)$ under $\bF_n^M$ by labeling the red strands by $M$.  Thus, for example
\[
    \begin{tikzpicture}[centerzero]
        \draw[->] (0.2,-0.2) -- (-0.2,0.2);
        \draw[wipe]  (-0.2,-0.2) -- (0.2,0.2);
        \draw[thick,red] (-0.2,-0.2) node[anchor=north] {\strandlabel{M}} -- (0.2,0.2);
    \end{tikzpicture}
    = \bF_n^M(\posupredcross) = T_{MV},\qquad
    \begin{tikzpicture}[centerzero]
        \draw[->] (-0.2,-0.2) -- (0.2,0.2);
        \draw[wipe]  (0.2,-0.2) -- (-0.2,0.2);
        \draw[thick,red] (0.2,-0.2) node[anchor=north] {\strandlabel{M}} -- (-0.2,0.2);
    \end{tikzpicture}
    = \bF_n^M(\negupredcross) = T_{MV}^{-1}.
\]
Then \cref{gringo} is equivalent to
\begin{equation} \label{rocks}
    \begin{tikzpicture}[centerzero]
        \draw[->] (0.4,-0.4) -- (-0.4,0.4);
        \draw[wipe] (-0.4,-0.4) -- (0.4,0.4);
        \draw[thick,red] (-0.4,-0.4) node[anchor=north] {\strandlabel{M}} -- (0.4,0.4) node[anchor=south] {\strandlabel{N}};
        \coupon{-0.2,-0.2}{f};
    \end{tikzpicture}
    =
    \begin{tikzpicture}[centerzero]
        \draw[->] (0.4,-0.4) -- (-0.4,0.4);
        \draw[wipe] (-0.4,-0.4) -- (0.4,0.4);
        \draw[thick,red] (-0.4,-0.4) node[anchor=north] {\strandlabel{M}} -- (0.4,0.4) node[anchor=south] {\strandlabel{N}};
        \coupon{0.2,0.2}{f};
    \end{tikzpicture}
    \quad \text{for all } f \in \Hom_{U_q}(M,N).
\end{equation}

\begin{lem}
    For any $U_q$-supermodules $M$ and $N$, we have
    \begin{equation} \label{twored}
        \begin{tikzpicture}[centerzero]
            \draw[->] (0.4,-0.4) -- (-0.4,0.4);
            \draw[wipe] (-0.4,-0.4) -- (0.4,0.4);
            \draw[thick,red] (-0.4,-0.4) node[anchor=north] {\strandlabel{M \otimes N}} -- (0.4,0.4);
        \end{tikzpicture}
        =
        \begin{tikzpicture}[centerzero]
            \draw[->] (0.4,-0.4) -- (-0.4,0.4);
            \draw[wipe] (-0.4,-0.4) -- (0,0.4);
            \draw[thick,red] (-0.4,-0.4) node[anchor=north] {\strandlabel{M}} -- (0,0.4);
            \draw[wipe] (0,-0.4) -- (0.4,0.4);
            \draw[thick,red] (0,-0.4) node[anchor=north] {\strandlabel{N}} -- (0.4,0.4);
        \end{tikzpicture}
        \ .
    \end{equation}
\end{lem}

\begin{proof}
    We have
    \begin{multline*}
        T_{M \otimes N,V}
        \overset{\cref{dive}}{=} \flip_{M \otimes N,V} \circ L^{13} L^{23}
        = (\flip_{MV} \otimes 1_N) \circ (1_M \otimes \flip_{NV}) \circ L^{13} L^{23}
        \\
        = (\flip_{MV} \otimes 1_N) \circ L^{12} \circ (1_M \otimes \flip_{NV}) \circ L^{23}
        = (T_M \otimes 1_N) \circ (1_M \otimes T_N).
        \qedhere
    \end{multline*}
\end{proof}

The relations \cref{rouge,boxslide,rocks,twored} show that $\posupredcross$, $\negupredcross$, $\posdownredcross$, and $\negdownredcross$, together with the crossings in $\Qcat(z)$, almost endow $\Wcat(z)$ with the structure of a braided monoidal category.  However, we do not truly have a braiding since, for example, we do not have a morphism corresponding to a red strand passing \emph{under} a black strand.  Furthermore, closed Clifford tokens do not pass under crossings.  In general, we can define a crossing for any sequence of strands in $\Wcat(z)$ passing over any sequence of strands in $\Qcat(z)$.  All morphisms in $\Wcat(z)$ pass over such crossings, but only \emph{some} morphisms in $\Qcat(z)$ pass under them.  In other words, the crossings are only natural in one argument.  Because of this asymmetry, we refer to this structure as a \emph{chiral braiding}.

We now restrict our attention to diagrams with a single red strand.  Let $\Wcat_1(z)$ denote the full sub-supercategory of $\Wcat(z)$ on objects that are tensor products of $\uparrow$, $\downarrow$, and $\redobject\,$, with \emph{exactly one} occurrence of $\,\redobject\,$.  Thus, objects of $\Wcat_1(z)$ are of the form $X \otimes \, \redobject\, \otimes Y$ for $X,Y \in \Qcat(z)$.  Note that $\Wcat_1(z)$ is not a \emph{monoidal} supercategory.

\begin{theo} \label{redfunctor}
    There is a unique superfunctor
    \[
        \redFunc_n \colon \Wcat_1(z) \to \SEnd(U_q\smod)
    \]
    defined as follows: On an object $X \in \Wcat_1(z)$, $\redFunc_n(X)$ is the superfunctor
    \[
        \redFunc_n(X) \colon U_q\smod \to U_q\smod,\quad M \mapsto \bF_n^M(X).
    \]
    On a morphism $f \in \Hom_{\Wcat_1(z)}(X,Y)$, $\redFunc_n(f)$ is the natural transformation $\redFunc_n(X) \to \redFunc_n(Y)$ whose $M$-component, for $M \in U_q\smod$, is
    \[
        \redFunc_n(f)_M = \bF_n^M(f).
    \]
\end{theo}

\begin{proof}
    It follows from \cref{rocks} that the given definition is natural in $M$.
\end{proof}

\section{The quantum affine isomeric supercategory}

In this section we introduce an affine version of the quantum isomeric supercategory and examine some of its properties.  Throughout this section, $\kk$ is an arbitrary commutative ring of characteristic not equal to two, and $z \in \kk$.

\begin{defin} \label{AQdef}
    The \emph{quantum affine isomeric supercategory} $\AQcat(z)$ is the strict monoidal supercategory obtained from $\Qcat(z)$ by adjoining an additional odd morphism
    \[
        \dotup \colon \uparrow\ \to\ \uparrow
    \]
    subject to the relations
    \begin{equation} \label{AQrels}
        \begin{tikzpicture}[centerzero]
            \draw[->] (0,-0.3) -- (0,0.3);
            \opendot{0,-0.1};
            \opendot{0,0.1};
        \end{tikzpicture}
        = -
        \begin{tikzpicture}[centerzero]
            \draw[->] (0,-0.3) -- (0,0.3);
        \end{tikzpicture}
        \ ,\quad
        \begin{tikzpicture}[centerzero]
            \draw[->] (0.3,-0.3) -- (-0.3,0.3);
            \draw[wipe] (-0.3,-0.3) -- (0.3,0.3);
            \draw[->] (-0.3,-0.3) -- (0.3,0.3);
            \opendot{0.17,-0.17};
        \end{tikzpicture}
        =
        \begin{tikzpicture}[centerzero]
            \draw[->] (0.3,-0.3) -- (-0.3,0.3);
            \draw[wipe] (-0.3,-0.3) -- (0.3,0.3);
            \draw[->] (-0.3,-0.3) -- (0.3,0.3);
            \opendot{-0.15,0.15};
        \end{tikzpicture}
        \ ,\quad
        \begin{tikzpicture}[centerzero]
            \bubright{0,0};
            \opendot{-0.2,0};
        \end{tikzpicture}
        = 0.
    \end{equation}
    We refer to $\dotup$ as an \emph{open Clifford token}.  To emphasize the difference, we will henceforth refer to $\tokup$ as a \emph{closed Clifford token}.
\end{defin}

It is important to note that we do \emph{not} impose a relation for sliding open Clifford tokens past closed ones.  It follows immediately from the defining relations that we have the following symmetry of $\AQcat(z)$.

\begin{lem} \label{bounce}
    There is a unique isomorphism of monoidal supercategories
    \[
        \AQcat(z) \to \AQcat(-z)
    \]
    determined on objects by $\uparrow\ \mapsto\ \uparrow$, $\downarrow\ \mapsto\ \downarrow$, and sending
    \[
        \tokup \mapsto \dotup,\quad
        \dotup \mapsto \tokup,\quad
        \posupcross \mapsto \negupcross,\quad
        \negupcross \mapsto \posupcross,\quad
        \rightcap \mapsto \rightcap,\quad
        \rightcup \mapsto \rightcup.
    \]
    On arbitrary diagrams, the isomorphism acts by interchanging open and closed Clifford tokens and flipping crossings.
\end{lem}

Define
\begin{equation} \label{dotdown}
    \begin{tikzpicture}[centerzero]
        \draw[<-] (0,-0.4) -- (0,0.4);
        \opendot{0,0};
    \end{tikzpicture}
    :=
    \begin{tikzpicture}[centerzero]
        \draw[->] (-0.3,0.4) -- (-0.3,0) arc(180:360:0.15) arc(180:0:0.15) -- (0.3,-0.4);
        \opendot{0,0};
    \end{tikzpicture}
    \ .
\end{equation}
It follows that
\begin{equation}
    \begin{tikzpicture}[centerzero]
        \draw[<-] (0,-0.3) -- (0,0.3);
        \opendot{0,-0.1};
        \opendot{0,0.1};
    \end{tikzpicture}
    =
    \begin{tikzpicture}[centerzero]
        \draw[<-] (0,-0.3) -- (0,0.3);
    \end{tikzpicture}
    \ .
\end{equation}

\begin{lem}
    The following relations hold in $\AQcat(z)$:
    \begin{gather} \label{ironopen}
        \begin{tikzpicture}[centerzero]
            \draw (-0.3,-0.3) -- (0.3,0.3);
            \draw[wipe] (0.3,-0.3) -- (-0.3,0.3);
            \draw (0.3,-0.3) -- (-0.3,0.3);
            \opendot{-0.18,-0.18};
        \end{tikzpicture}
        =
        \begin{tikzpicture}[centerzero]
            \draw (-0.3,-0.3) -- (0.3,0.3);
            \draw[wipe] (0.3,-0.3) -- (-0.3,0.3);
            \draw (0.3,-0.3) -- (-0.3,0.3);
            \opendot{0.18,0.18};
        \end{tikzpicture}
        \ ,\quad
        \begin{tikzpicture}[centerzero]
            \draw (0.3,-0.3) -- (-0.3,0.3);
            \draw[wipe] (-0.3,-0.3) -- (0.3,0.3);
            \draw (-0.3,-0.3) -- (0.3,0.3);
            \opendot{-0.18,0.18};
        \end{tikzpicture}
        =
        \begin{tikzpicture}[centerzero]
            \draw (0.3,-0.3) -- (-0.3,0.3);
            \draw[wipe] (-0.3,-0.3) -- (0.3,0.3);
            \draw (-0.3,-0.3) -- (0.3,0.3);
            \opendot{0.18,-0.18};
        \end{tikzpicture}
        \ ,\quad
        \begin{tikzpicture}[anchorbase]
            \draw (-0.2,0.2) -- (-0.2,0) arc (180:360:0.2) -- (0.2,0.2);
            \opendot{-0.2,0};
        \end{tikzpicture}
        \ =\
        \begin{tikzpicture}[anchorbase]
            \draw (-0.2,0.2) -- (-0.2,0) arc (180:360:0.2) -- (0.2,0.2);
            \opendot{0.2,0};
        \end{tikzpicture}
        \ ,\quad
        \begin{tikzpicture}[anchorbase]
            \draw (-0.2,-0.2) -- (-0.2,0) arc (180:0:0.2) -- (0.2,-0.2);
            \opendot{-0.2,0};
        \end{tikzpicture}
        \ =\
        \begin{tikzpicture}[anchorbase]
            \draw (-0.2,-0.2) -- (-0.2,0) arc (180:0:0.2) -- (0.2,-0.2);
            \opendot{0.2,0};
        \end{tikzpicture}
        \ ,\quad
        \begin{tikzpicture}[centerzero]
            \bubun{0,0};
            \opendot{-0.2,0};
        \end{tikzpicture}
        = 0,
        \\ \label{nickelopen}
        \begin{tikzpicture}[centerzero]
            \draw[->] (0.3,-0.3) -- (-0.3,0.3);
            \draw[wipe] (-0.3,-0.3) -- (0.3,0.3);
            \draw[->] (-0.3,-0.3) -- (0.3,0.3);
            \opendot{-0.15,-0.15};
        \end{tikzpicture}
        =
        \begin{tikzpicture}[centerzero]
            \draw[->] (0.3,-0.3) -- (-0.3,0.3);
            \draw[wipe] (-0.3,-0.3) -- (0.3,0.3);
            \draw[->] (-0.3,-0.3) -- (0.3,0.3);
            \opendot{0.15,0.15};
        \end{tikzpicture}
        + z\
        \left(
            \begin{tikzpicture}[centerzero]
                \draw[->] (-0.2,-0.3) -- (-0.2,0.3);
                \draw[->] (0.2,-0.3) -- (0.2,0.3);
                \opendot{-0.2,0};
            \end{tikzpicture}
            -
            \begin{tikzpicture}[centerzero]
                \draw[->] (-0.2,-0.3) -- (-0.2,0.3);
                \draw[->] (0.2,-0.3) -- (0.2,0.3);
                \opendot{0.2,0};
            \end{tikzpicture}
        \right),
        \qquad
        \begin{tikzpicture}[centerzero]
            \draw[->] (-0.3,-0.3) -- (0.3,0.3);
            \draw[wipe] (0.3,-0.3) -- (-0.3,0.3);
            \draw[->] (0.3,-0.3) -- (-0.3,0.3);
            \opendot{-0.15,0.15};
        \end{tikzpicture}
        =
        \begin{tikzpicture}[centerzero]
            \draw[->] (-0.3,-0.3) -- (0.3,0.3);
            \draw[wipe] (0.3,-0.3) -- (-0.3,0.3);
            \draw[->] (0.3,-0.3) -- (-0.3,0.3);
            \opendot{0.16,-0.16};
        \end{tikzpicture}
        + z\
        \left(
            \begin{tikzpicture}[centerzero]
                \draw[->] (-0.2,-0.3) -- (-0.2,0.3);
                \draw[->] (0.2,-0.3) -- (0.2,0.3);
                \opendot{0.2,0};
            \end{tikzpicture}
            -
            \begin{tikzpicture}[centerzero]
                \draw[->] (-0.2,-0.3) -- (-0.2,0.3);
                \draw[->] (0.2,-0.3) -- (0.2,0.3);
                \opendot{-0.2,0};
            \end{tikzpicture}
        \right),
    \end{gather}
    where, in \cref{ironopen}, the relations hold for all orientations of the strands.
\end{lem}

\begin{proof}
    This follows immediately from \cref{bounce,ironfull,nickelfull}.
\end{proof}

It follows from the above discussion that the isomorphisms $\Omega_-$, $\Omega_\updownarrow$, and $\Omega_\leftrightarrow$ defined in \cref{sec:Qdef} extend to isomorphisms of monoidal supercategories
\[
    \Omega_- \colon \AQcat(z) \xrightarrow{\cong} \AQcat(-z),\qquad
    \Omega_\updownarrow \colon \AQcat(z) \xrightarrow{\cong} \AQcat(z)^\op,\qquad
    \Omega_\leftrightarrow \colon \AQcat(z) \xrightarrow{\cong} \AQcat(z)^\rev.
\]
These are defined as in \cref{sec:Qdef} for the generators of $\Qcat(z)$ and, on the open Clifford token, are defined by
\[
    \Omega_-(\dotup) = \dotup,\qquad
    \Omega_\updownarrow  (\dotup) = \dotdown,\qquad
    \Omega_\leftrightarrow (\dotup) = \dotup.
\]
Furthermore, $\AQcat(z)$ is strictly pivotal, with duality superfunctor
\[
    \Omega_\leftrightarrow \circ \Omega_\updownarrow \colon \AQcat(z) \xrightarrow{\cong} (\AQcat(z)^\op)^\rev.
\]

\begin{rem}
    Note that, while $\Qcat(0)$ is isomorphic to the oriented Brauer--Clifford supercategory of \cite[Def.~3.2]{BCK19}, as described in \cref{degenerate}, the supercategory $\AQcat(z)$ does \emph{not} reduce to the definition \cite[Def.~3.2]{BCK19} of the degenerate affine oriented Brauer--Clifford supercategory when $z=0$.  This is analogous to the fact that the degenerate affine Hecke algebra of type $A$ is not simply the $q=1$ specialization of the affine Hecke algebra of type $A$.
\end{rem}

Define, for $k \in \Z$,
\begin{equation} \label{zebradef}
    \begin{tikzpicture}[centerzero]
        \draw[->] (0,-0.3) -- (0,0.3);
        \zebras{0,0}{east}{k};
    \end{tikzpicture}
    :=
    \begin{tikzpicture}[centerzero]
        \draw[->] (0,-0.6) -- (0,0.6);
        \opendot{0,-0.45};
        \token{0,-0.25};
        \opendot{0,-0.05};
        \draw (0,0.35) node[anchor=west] {$\vdots$};
    \end{tikzpicture}
    \Bigg\} \text{\scriptsize{$k$ tokens}}
    \quad \text{if } k \ge 0
    \qquad \text{and} \qquad
    \begin{tikzpicture}[centerzero]
        \draw[->] (0,-0.3) -- (0,0.3);
        \zebras{0,0}{east}{k};
    \end{tikzpicture}
    :=
    \begin{tikzpicture}[centerzero]
        \draw[->] (0,-0.6) -- (0,0.6);
        \token{0,-0.45};
        \opendot{0,-0.25};
        \token{0,-0.05};
        \draw (0,0.35) node[anchor=west] {$\vdots$};
    \end{tikzpicture}
    \Bigg\} \text{\scriptsize{$-k$ tokens}}
    \quad \text{if } k < 0.
\end{equation}
Note that both morphisms in \cref{zebradef} are of parity $k \pmod 2$.  We then define, for $k \in \Z$,
\begin{equation} \label{zebradown}
    \begin{tikzpicture}[centerzero]
        \draw[<-] (0,-0.4) -- (0,0.4);
        \zebras{0,0}{east}{k};
    \end{tikzpicture}
    :=
    \begin{tikzpicture}[centerzero]
        \draw[->] (0.3,0.4) -- (0.3,0) arc(360:180:0.15) arc(0:180:0.15) -- (-0.3,-0.4);
        \zebras{0,0}{south}{k};
    \end{tikzpicture}
    \ ,\qquad
    \begin{tikzpicture}[centerzero]
        \draw[->] (0,-0.3) -- (0,0.3);
        \zebra{0,0};
    \end{tikzpicture}
    :=
    \begin{tikzpicture}[centerzero]
        \draw[->] (0,-0.3) -- (0,0.3);
        \zebras{0,0}{west}{2};
    \end{tikzpicture}
    \ ,\qquad
    \begin{tikzpicture}[centerzero]
        \draw[<-] (0,-0.3) -- (0,0.3);
        \zebra{0,0};
    \end{tikzpicture}
    :=
    \begin{tikzpicture}[centerzero]
        \draw[<-] (0,-0.3) -- (0,0.3);
        \zebras{0,0}{west}{2};
    \end{tikzpicture}
    \ .
\end{equation}
We refer to the decorations
$
    \begin{tikzpicture}
        \zebra{0,0};
    \end{tikzpicture}
$
as \emph{zebras}.  We have colored them and their labels mahogany to help distinguish these labels from coefficients in linear combinations of diagrams.  The morphism $\zebraup$ should be thought of as a quantum analogue of the even morphism $\tokup$ of \cite{BCK19}.

Recall our convention
\[
    \begin{tikzpicture}[centerzero]
        \draw[->] (-0.2,-0.3) -- (-0.2,0.3);
        \draw[->] (0.2,-0.3) -- (0.2,0.3);
        \zebras{-0.2,0}{east}{k};
        \zebras{0.2,0}{west}{l};
    \end{tikzpicture}
    =
    \begin{tikzpicture}[centerzero]
        \draw[->] (-0.2,-0.3) -- (-0.2,0.3);
        \draw[->] (0.2,-0.3) -- (0.2,0.3);
        \zebras{-0.2,0.1}{east}{k};
        \zebras{0.2,-0.1}{west}{l};
    \end{tikzpicture}
    \ .
\]
That is, when zebras appear at the same height, the entire zebra on the left should be considered as above the entire zebra on the right.  Note that composition of zebras is a bit subtle, since the labels do not add in general.
We have a homomorphism of superalgebras
\[
    \kk[x^{\pm 1}]\langle y \rangle / (yx=x^{-1}y,\, y^2=-1)
    \to \End_{\Qcat(z)}(\uparrow),
    \qquad
    x^k \mapsto \zebramultup{2k}, \quad
    y x^k \mapsto (-1)^{\delta_{k<0}}\ \zebramultup{2k+1}.
\]
\Cref{affinebasis} below would imply that this map is injective.

\begin{lem}
    The following relations hold in $\AQcat(z)$ for all $k \in \Z$:
    \begin{gather} \label{kimchi}
        \begin{tikzpicture}[anchorbase]
            \draw (-0.2,0.2) -- (-0.2,0) arc (180:360:0.2) -- (0.2,0.2);
            \zebras{-0.2,0}{east}{k};
        \end{tikzpicture}
        \ =\
        \begin{tikzpicture}[anchorbase]
            \draw (-0.2,0.2) -- (-0.2,0) arc (180:360:0.2) -- (0.2,0.2);
            \zebras{0.2,0}{west}{k};
        \end{tikzpicture}
        \ ,\quad
        \begin{tikzpicture}[anchorbase]
            \draw (-0.2,-0.2) -- (-0.2,0) arc (180:0:0.2) -- (0.2,-0.2);
            \zebras{-0.2,0}{east}{k};
        \end{tikzpicture}
        \ =\
        \begin{tikzpicture}[anchorbase]
            \draw (-0.2,-0.2) -- (-0.2,0) arc (180:0:0.2) -- (0.2,-0.2);
            \zebras{0.2,0}{west}{k};
        \end{tikzpicture}
        \ ,
        \\ \label{tapioca}
        \rightzebrabub{2k+1} = 0 = \leftzebrabub{2k+1}, \quad
        \rightzebrabub{2k} = - \left( \rightzebrabub{-2k} \right),\quad
        \leftzebrabub{2k} = - \leftzebrabub{-2k},
    \end{gather}
    where, in \cref{kimchi}, the relations hold for both orientations of the strands.
\end{lem}

\begin{proof}
    It follows from \cref{zebradown,iron,ironopen,leftadj} that
    \[
        \begin{tikzpicture}[centerzero]
            \draw[<-] (0,-0.4) -- (0,0.4);
            \zebras{0,0}{east}{k};
        \end{tikzpicture}
        = (-1)^{\binom{k}{2}}\,
        \begin{tikzpicture}[centerzero]
            \draw[<-] (0,-0.6) -- (0,0.6);
            \opendot{0,0.45};
            \token{0,0.25};
            \opendot{0,0.05};
            \draw (0,-0.2) node[anchor=west] {$\vdots$};
        \end{tikzpicture}
        \Bigg\} \text{\scriptsize{$k$ tokens}}
        \quad \text{if } k \ge 0
        \qquad \text{and} \qquad
        \begin{tikzpicture}[centerzero]
            \draw[<-] (0,-0.3) -- (0,0.3);
            \zebras{0,0}{east}{k};
        \end{tikzpicture}
        = (-1)^{\binom{k}{2}}\,
        \begin{tikzpicture}[centerzero]
            \draw[<-] (0,-0.6) -- (0,0.6);
            \token{0,0.45};
            \opendot{0,0.25};
            \token{0,0.05};
            \draw (0,-0.2) node[anchor=west] {$\vdots$};
        \end{tikzpicture}
        \Bigg\} \text{\scriptsize{$-k$ tokens}}
        \quad \text{if } k < 0.
    \]
    Then relations \cref{kimchi} follow from \cref{iron,ironopen}.

    For $k \ge 0$, we have
    \[
        \rightzebrabub{2k+1}
        =
        \begin{tikzpicture}[centerzero]
            \draw[->] (0,0.3) arc(90:0:0.2) -- (0.2,-0.1) arc(360:180:0.2) -- (-0.2,0.1) arc(180:90:0.2);
            \zebras{-0.2,-0.1}{east}{2k};
            \opendot{-0.2,0.1};
        \end{tikzpicture}
        \overset{\cref{ironopen}}{=}
        \begin{tikzpicture}[centerzero]
            \draw[->] (0,0.3) arc(90:0:0.2) -- (0.2,-0.1) arc(360:180:0.2) -- (-0.2,0.1) arc(180:90:0.2);
            \zebras{-0.2,-0.1}{east}{2k};
            \opendot{0.2,0.1};
        \end{tikzpicture}
        =
        \begin{tikzpicture}[centerzero]
            \draw[->] (0,0.3) arc(90:0:0.2) -- (0.2,-0.1) arc(360:180:0.2) -- (-0.2,0.1) arc(180:90:0.2);
            \zebras{-0.2,0.1}{east}{2k};
            \opendot{0.2,-0.1};
        \end{tikzpicture}
        \overset{\cref{ironopen}}{=}
        \begin{tikzpicture}[centerzero]
            \draw[->] (0,0.3) arc(90:0:0.2) -- (0.2,-0.1) arc(360:180:0.2) -- (-0.2,0.1) arc(180:90:0.2);
            \zebras{-0.2,0.1}{east}{2k};
            \opendot{-0.2,-0.1};
        \end{tikzpicture}
        = -
        \rightzebrabub{1-2k}
        =
        \dotsb
        = \pm \rightzebrabub{\pm 1}
        \overset{\cref{tokrel}}{\underset{\cref{ironopen}}{=}} 0.
    \]
    The case $k<0$, as well as the proof of the second equality in \cref{tapioca}, are analogous.

    For the third equality in \cref{tapioca}, it suffices to consider the case $k>0$.  In this case, we have
    \[
        \rightzebrabub{2k}
        =
        \begin{tikzpicture}[centerzero]
            \draw[->] (0,0.3) arc(90:0:0.2) -- (0.2,-0.1) arc(360:180:0.2) -- (-0.2,0.1) arc(180:90:0.2);
            \zebras{-0.2,-0.1}{east}{2k-1};
            \token{-0.2,0.1};
        \end{tikzpicture}
        \overset{\cref{ironopen}}{=}
        \begin{tikzpicture}[centerzero]
            \draw[->] (0,0.3) arc(90:0:0.2) -- (0.2,-0.1) arc(360:180:0.2) -- (-0.2,0.1) arc(180:90:0.2);
            \zebras{-0.2,-0.1}{east}{2k-1};
            \token{0.2,0.1};
        \end{tikzpicture}
        = -
        \begin{tikzpicture}[centerzero]
            \draw[->] (0,0.3) arc(90:0:0.2) -- (0.2,-0.1) arc(360:180:0.2) -- (-0.2,0.1) arc(180:90:0.2);
            \zebras{-0.2,0.1}{east}{2k-1};
            \token{0.2,-0.1};
        \end{tikzpicture}
        \overset{\cref{ironopen}}{=} -
        \begin{tikzpicture}[centerzero]
            \draw[->] (0,0.3) arc(90:0:0.2) -- (0.2,-0.1) arc(360:180:0.2) -- (-0.2,0.1) arc(180:90:0.2);
            \zebras{-0.2,0.1}{east}{2k-1};
            \token{-0.2,-0.1};
        \end{tikzpicture}
        =
        - \left( \rightzebrabub{-2k} \right).
    \]
    The proof of the last equality in \cref{tapioca} is similar.
\end{proof}

\begin{lem}
    The following relations hold in $\AQcat(z)$:
    \begin{equation} \label{zebracrossing}
        \begin{tikzpicture}[centerzero]
            \draw[->] (0.3,-0.3) -- (-0.3,0.3);
            \draw[wipe] (-0.3,-0.3) -- (0.3,0.3);
            \draw[->] (-0.3,-0.3) -- (0.3,0.3);
            \zebra{-0.16,-0.16};
        \end{tikzpicture}
        =
        \begin{tikzpicture}[centerzero]
            \draw[->] (-0.3,-0.3) -- (0.3,0.3);
            \draw[wipe] (0.3,-0.3) -- (-0.3,0.3);
            \draw[->] (0.3,-0.3) -- (-0.3,0.3);
            \zebra{0.15,0.15};
        \end{tikzpicture}
        - z\
        \begin{tikzpicture}[centerzero]
            \draw[->] (-0.2,-0.3) -- (-0.2,0.3);
            \draw[->] (0.2,-0.3) -- (0.2,0.3);
            \opendot{-0.2,0};
            \token{0.2,0};
        \end{tikzpicture}
        ,
        \qquad
        \begin{tikzpicture}[centerzero]
            \draw[->] (-0.3,-0.3) -- (0.3,0.3);
            \draw[wipe] (0.3,-0.3) -- (-0.3,0.3);
            \draw[->] (0.3,-0.3) -- (-0.3,0.3);
            \zebra{0.16,-0.16};
        \end{tikzpicture}
        =
        \begin{tikzpicture}[centerzero]
            \draw[->] (0.3,-0.3) -- (-0.3,0.3);
            \draw[wipe] (-0.3,-0.3) -- (0.3,0.3);
            \draw[->] (-0.3,-0.3) -- (0.3,0.3);
            \zebra{-0.15,0.15};
        \end{tikzpicture}
        - z\
        \begin{tikzpicture}[centerzero]
            \draw[->] (-0.2,-0.3) -- (-0.2,0.3);
            \draw[->] (0.2,-0.3) -- (0.2,0.3);
            \token{-0.2,0};
            \opendot{0.2,0};
        \end{tikzpicture}
        .
    \end{equation}
\end{lem}

\begin{proof}
    For the first relation, we have
    \[
        \begin{tikzpicture}[centerzero]
            \draw[->] (0.3,-0.3) -- (-0.3,0.3);
            \draw[wipe] (-0.3,-0.3) -- (0.3,0.3);
            \draw[->] (-0.3,-0.3) -- (0.3,0.3);
            \zebra{-0.16,-0.16};
        \end{tikzpicture}
        \overset{\cref{tokrel}}{=}
        \begin{tikzpicture}[centerzero]
            \draw[->] (0.3,-0.3) -- (-0.3,0.3);
            \draw[wipe] (-0.3,-0.3) -- (0.3,0.3);
            \draw[->] (-0.3,-0.3) -- (0.3,0.3);
            \opendot{-0.15,-0.15};
            \token{0.15,0.15};
        \end{tikzpicture}
        \overset{\cref{nickelopen}}{=}
        \begin{tikzpicture}[centerzero]
            \draw[->] (0.3,-0.3) -- (-0.3,0.3);
            \draw[wipe] (-0.3,-0.3) -- (0.3,0.3);
            \draw[->] (-0.3,-0.3) -- (0.3,0.3);
            \zebra{0.15,0.15};
        \end{tikzpicture}
        + z
        \left(
            \begin{tikzpicture}[centerzero]
                \draw[->] (-0.2,-0.3) -- (-0.2,0.3);
                \draw[->] (0.2,-0.3) -- (0.2,0.3);
                \opendot{-0.2,-0.1};
                \token{0.2,0.1};
            \end{tikzpicture}
            -
            \begin{tikzpicture}[centerzero]
                \draw[->] (-0.2,-0.3) -- (-0.2,0.3);
                \draw[->] (0.2,-0.3) -- (0.2,0.3);
                \zebra{0.2,0};
            \end{tikzpicture}
        \right)
        \overset{\cref{skein}}{=}
        \begin{tikzpicture}[centerzero]
            \draw[->] (-0.3,-0.3) -- (0.3,0.3);
            \draw[wipe] (0.3,-0.3) -- (-0.3,0.3);
            \draw[->] (0.3,-0.3) -- (-0.3,0.3);
            \zebra{0.15,0.15};
        \end{tikzpicture}
        - z\
        \begin{tikzpicture}[centerzero]
            \draw[->] (-0.2,-0.3) -- (-0.2,0.3);
            \draw[->] (0.2,-0.3) -- (0.2,0.3);
            \opendot{-0.2,0};
            \token{0.2,0};
        \end{tikzpicture}
        .
    \]
    The proof of the second relation is analogous.
\end{proof}

\begin{cor}
    For all $k \in \Z_{>0}$, the following relations hold in $\AQcat(z)$:
    \begin{align} \label{flipflop}
        \begin{tikzpicture}[centerzero]
            \draw[->] (0.4,-0.4) -- (-0.4,0.4);
            \draw[wipe] (-0.4,-0.4) -- (0.4,0.4);
            \draw[->] (-0.4,-0.4) -- (0.4,0.4);
            \zebras{-0.16,-0.16}{east}{2k};
        \end{tikzpicture}
        &=
        \begin{tikzpicture}[centerzero]
            \draw[->] (-0.4,-0.4) -- (0.4,0.4);
            \draw[wipe] (0.4,-0.4) -- (-0.4,0.4);
            \draw[->] (0.4,-0.4) -- (-0.4,0.4);
            \zebras{0.18,0.18}{west}{2k};
        \end{tikzpicture}
        - z \sum_{r=1}^{k-1}
        \begin{tikzpicture}[centerzero]
            \draw[->] (-0.2,-0.3) -- (-0.2,0.3);
            \draw[->] (0.2,-0.3) -- (0.2,0.3);
            \zebras{-0.2,0}{east}{2k-2r};
            \zebras{0.2,0}{west}{2r};
        \end{tikzpicture}
        - z \sum_{r=1}^k
        \begin{tikzpicture}[centerzero]
            \draw[->] (-0.2,-0.3) -- (-0.2,0.3);
            \draw[->] (0.2,-0.3) -- (0.2,0.3);
            \zebras{-0.2,0}{east}{2k-2r+1};
            \zebras{0.2,0}{west}{1-2r};
        \end{tikzpicture}
        ,
        \\ \label{flipflop2}
        \begin{tikzpicture}[centerzero]
            \draw[->] (-0.4,-0.4) -- (0.4,0.4);
            \draw[wipe] (0.4,-0.4) -- (-0.4,0.4);
            \draw[->] (0.4,-0.4) -- (-0.4,0.4);
            \zebras{-0.18,-0.18}{east}{-2k};
        \end{tikzpicture}
        &=
        \begin{tikzpicture}[centerzero]
            \draw[->] (0.4,-0.4) -- (-0.4,0.4);
            \draw[wipe] (-0.4,-0.4) -- (0.4,0.4);
            \draw[->] (-0.4,-0.4) -- (0.4,0.4);
            \zebras{0.16,0.16}{west}{-2k};
        \end{tikzpicture}
        + z \sum_{r=1}^{k-1}
        \begin{tikzpicture}[centerzero]
            \draw[->] (-0.2,-0.3) -- (-0.2,0.3);
            \draw[->] (0.2,-0.3) -- (0.2,0.3);
            \zebras{-0.2,0}{east}{-2r};
            \zebras{0.2,0}{west}{2r-2k};
        \end{tikzpicture}
        + z \sum_{r=1}^k
        \begin{tikzpicture}[centerzero]
            \draw[->] (-0.2,-0.3) -- (-0.2,0.3);
            \draw[->] (0.2,-0.3) -- (0.2,0.3);
            \zebras{-0.2,0}{east}{1-2r};
            \zebras{0.2,0}{west}{2k-2r+1};
        \end{tikzpicture}
        .
    \end{align}
\end{cor}

\begin{proof}
    We prove \cref{flipflop} by induction on $k$.  The case $k=1$ is \cref{zebracrossing}.  Then, for $k \ge 2$, we have
    \begin{align*}
        \begin{tikzpicture}[centerzero]
            \draw[->] (0.4,-0.4) -- (-0.4,0.4);
            \draw[wipe] (-0.4,-0.4) -- (0.4,0.4);
            \draw[->] (-0.4,-0.4) -- (0.4,0.4);
            \zebras{-0.16,-0.16}{east}{2k};
        \end{tikzpicture}
        &=
        \begin{tikzpicture}[centerzero]
            \draw[->] (-0.4,-0.4) -- (0.4,0.4);
            \draw[wipe] (0.4,-0.4) -- (-0.4,0.4);
            \draw[->] (0.4,-0.4) -- (-0.4,0.4);
            \zebras{0.18,0.18}{west}{2k-2};
            \zebra{-0.16,-0.16};
        \end{tikzpicture}
        - z \sum_{r=1}^{k-2}
        \begin{tikzpicture}[centerzero]
            \draw[->] (-0.2,-0.3) -- (-0.2,0.3);
            \draw[->] (0.2,-0.3) -- (0.2,0.3);
            \zebras{-0.2,0}{east}{2k-2r};
            \zebras{0.2,0}{west}{2r};
        \end{tikzpicture}
        - z \sum_{r=1}^{k-1}
        \begin{tikzpicture}[centerzero]
            \draw[->] (-0.2,-0.3) -- (-0.2,0.3);
            \draw[->] (0.2,-0.3) -- (0.2,0.3);
            \zebras{-0.2,0}{east}{2k-2r+1};
            \zebras{0.2,0}{west}{1-2r};
        \end{tikzpicture}
        \\
        &\overset{\mathclap{\cref{skein}}}{=}\
        \begin{tikzpicture}[centerzero]
            \draw[->] (0.4,-0.4) -- (-0.4,0.4);
            \draw[wipe] (-0.4,-0.4) -- (0.4,0.4);
            \draw[->] (-0.4,-0.4) -- (0.4,0.4);
            \zebras{0.18,0.18}{west}{2k-2};
            \zebra{-0.16,-0.16};
        \end{tikzpicture}
        - z \sum_{r=1}^{k-1}
        \begin{tikzpicture}[centerzero]
            \draw[->] (-0.2,-0.3) -- (-0.2,0.3);
            \draw[->] (0.2,-0.3) -- (0.2,0.3);
            \zebras{-0.2,0}{east}{2k-2r};
            \zebras{0.2,0}{west}{2r};
        \end{tikzpicture}
        - z \sum_{r=1}^{k-1}
        \begin{tikzpicture}[centerzero]
            \draw[->] (-0.2,-0.3) -- (-0.2,0.3);
            \draw[->] (0.2,-0.3) -- (0.2,0.3);
            \zebras{-0.2,0}{east}{2k-2r+1};
            \zebras{0.2,0}{west}{1-2r};
        \end{tikzpicture}
        \\
        &\overset{\mathclap{\cref{zebracrossing}}}{=}\
        \begin{tikzpicture}[centerzero]
            \draw[->] (-0.4,-0.4) -- (0.4,0.4);
            \draw[wipe] (0.4,-0.4) -- (-0.4,0.4);
            \draw[->] (0.4,-0.4) -- (-0.4,0.4);
            \zebras{0.18,0.18}{west}{2k};
        \end{tikzpicture}
        - z \sum_{r=1}^{k-1}
        \begin{tikzpicture}[centerzero]
            \draw[->] (-0.2,-0.3) -- (-0.2,0.3);
            \draw[->] (0.2,-0.3) -- (0.2,0.3);
            \zebras{-0.2,0}{east}{2k-2r};
            \zebras{0.2,0}{west}{2r};
        \end{tikzpicture}
        - z \sum_{r=1}^k
        \begin{tikzpicture}[centerzero]
            \draw[->] (-0.2,-0.3) -- (-0.2,0.3);
            \draw[->] (0.2,-0.3) -- (0.2,0.3);
            \zebras{-0.2,0}{east}{2k-2r+1};
            \zebras{0.2,0}{west}{1-2r};
        \end{tikzpicture}
        ,
    \end{align*}
    where we used the induction hypothesis in the first equality.  Relation \cref{flipflop2} then follows by composing \cref{flipflop} on the bottom with
    $
        \begin{tikzpicture}[centerzero]
            \draw[->] (0,-0.2) -- (0,0.2);
            \zebras{0,-0.02}{east}{-2k};
        \end{tikzpicture}
        \ \upstrand
    $ and on the top with $\upstrand\ \zebramultup{-2k}$.
\end{proof}

\begin{lem}
    For all $k > 0$, we have
    \begin{equation} \label{eyez}
        \rightzebrabub{2k} - \leftzebrabub{2k}
        = z \sum_{r=1}^{k-1} \rightzebrabub{2r}\, \leftzebrabub{2k-2r}.
    \end{equation}
\end{lem}

\begin{proof}
    We have
    \begin{multline*}
        \leftzebrabub{2k}
        =
        \begin{tikzpicture}[centerzero]
            \draw (0.7,0) arc(0:-90:0.3) to[out=180,in=0] (-0.4,0.3);
            \draw[wipe] (-0.4,-0.3) to[out=0,in=180] (0.4,0.3);
            \draw[->] (-0.4,0.3) arc(90:270:0.3) to[out=0,in=180] (0.4,0.3) arc(90:0:0.3);
            \zebras{-0.16,-0.2}{north}{2k};
        \end{tikzpicture}
        \overset{\cref{flipflop}}{=}
        \begin{tikzpicture}[centerzero]
            \draw[->] (-0.4,0.3) arc(90:270:0.3) to[out=0,in=180] (0.4,0.3) arc(90:0:0.3);
            \draw[wipe] (0.7,0) arc(0:-90:0.3) to[out=180,in=0] (-0.4,0.3);
            \draw (0.7,0) arc(0:-90:0.3) to[out=180,in=0] (-0.4,0.3);
            \zebras{0.16,0.2}{south}{2k};
        \end{tikzpicture}
        - z \sum_{r=1}^{k-1} \rightzebrabub{2r}\, \leftzebrabub{2k-2r}
        - z \sum_{r=1}^k \rightzebrabub{1-2r}\, \leftzebrabub{2k-2r+1}
        \\
        \overset{\cref{tapioca}}{=} \rightzebrabub{2k} - z \sum_{r=1}^{k-1} \rightzebrabub{2r}\, \leftzebrabub{2k-2r}.
        \qedhere
    \end{multline*}
\end{proof}

Let $\Sym$ denote the $\kk$-algebra of symmetric functions over $\kk$.  For $r \ge 0$, let $e_r$ and $h_r$ denote the degree $r$ elementary and complete homogeneous symmetric functions, respectively, with the convention that $e_0 = h_0 = 1$.

\begin{prop}
    We have a homomorphism of rings
    \[
        \beta \colon \Sym \to \End_{\AQcat(z)}(\one),\quad
        e_r \mapsto (-1)^{r-1} z\, \leftzebrabub{2r},\quad
        h_r \mapsto  z \rightzebrabub{2r},\quad r \ge 1.
    \]
\end{prop}

\begin{proof}
    The $\kk$-algebra $\Sym$ is generated by $e_r,h_r$, $r>0$, modulo the identities
    \[
        \sum_{r=0}^k (-1)^r e_{k-r} h_r = 0,\quad k > 0,
    \]
    where $h_0=e_0=1$.  The map $\beta$ sends the left-hand side of this identity to
    \[
        (-1)^{k-1} z\, \leftzebrabub{2k}
        + (-1)^{k-1} z^2 \sum_{r=1}^{k-1} \rightzebrabub{2r}\, \leftzebrabub{2k-2r}
        + (-1)^k z \rightzebrabub{2k},
    \]
    which is equal to zero by \cref{eyez}.
\end{proof}

\Cref{affinespan} below implies that the map $\beta$ is surjective, while \Cref{affinebasis} would imply it is an isomorphism.  We next deduce a bubble slide relation.

\begin{lem}
    For all $k \ge 0$, we have
    \[
        \leftzebrabub{2k}
        \begin{tikzpicture}[centerzero]
            \draw[->] (0,-0.35) -- (0,0.35);
        \end{tikzpicture}
        =
        \begin{tikzpicture}[centerzero]
            \draw[->] (0,-0.35) -- (0,0.35);
        \end{tikzpicture}
        \leftzebrabub{2k}
        - kz
        \left(
            \begin{tikzpicture}[centerzero]
                \draw[->] (0,-0.35) -- (0,0.35);
                \zebras{0,0}{west}{2k};
            \end{tikzpicture}
            +
            \begin{tikzpicture}[centerzero]
                \draw[->] (0,-0.35) -- (0,0.35);
                \zebras{0,0}{west}{-2k};
            \end{tikzpicture}
        \!\!\right)
        + z^2 \sum_{r=1}^{k-1} r\ \leftzebrabub{2k-2r}
        \left(
            \begin{tikzpicture}[centerzero]
                \draw[->] (0,-0.35) -- (0,0.35);
                \zebras{0,0}{west}{2r};
            \end{tikzpicture}
            +
            \begin{tikzpicture}[centerzero]
                \draw[->] (0,-0.35) -- (0,0.35);
                \zebras{0,0}{west}{-2r};
            \end{tikzpicture}
        \!\!\right).
    \]
\end{lem}

\begin{proof}
    The case $k=0$ follows immediately from \cref{braid}.  Thus, we suppose $k>0$.  We first compute
    \[
        \leftzebrabub{2k}
        \begin{tikzpicture}[centerzero]
            \draw[->] (0,-0.5) -- (0,0.5);
        \end{tikzpicture}
        \overset{\cref{braid}}{=}
        \begin{tikzpicture}[centerzero]
            \draw[->] (0.2,-0.5) to[out=110,in=down] (-0.1,0) to[out=up,in=250] (0.2,0.5);
            \draw[wipe] (-0.6,0) to[out=down,in=left] (-0.4,-0.4) to[out=right,in=225] (-0.2,-0.3) to[out=45,in=down] (0.2,0) to[out=up,in=right] (-0.3,0.4) to[out=left,in=up] (-0.6,0);
            \draw[->] (-0.6,0) to[out=down,in=left] (-0.4,-0.4) to[out=right,in=225] (-0.2,-0.3) to[out=45,in=down] (0.2,0) to[out=up,in=right] (-0.3,0.4) to[out=left,in=up] (-0.6,0);
            \zebras{-0.2,-0.3}{north}{2k};
        \end{tikzpicture}
        \overset{\cref{flipflop}}{=}
        \begin{tikzpicture}[centerzero]
            \draw (-0.3,0) arc(180:360:0.3);
            \draw[wipe] (0.3,-0.5) to[out=110,in=down] (0,0) to[out=up,in=250] (0.3,0.5);
            \draw[->] (0.3,-0.5) to[out=110,in=down] (0,0) to[out=up,in=250] (0.3,0.5);
            \draw[wipe] (0.3,0) arc(0:180:0.3);
            \draw[->] (0.3,0) arc(0:180:0.3);
            \zebras{0.3,0}{west}{2k};
        \end{tikzpicture}
        - z \sum_{r=1}^{k-1}
        \begin{tikzpicture}[centerzero]
            \draw[->] (-0.6,0.2) -- (-0.6,-0.2) arc(180:360:0.2) \braidup (0.2,0.2) -- (0.2,0.5);
            \draw[wipe] (0.2,-0.5) -- (0.2,-0.2) \braidup (-0.2,0.2) arc(0:180:0.2);
            \draw (0.2,-0.5) -- (0.2,-0.2) \braidup (-0.2,0.2) arc(0:180:0.2);
            \zebra{-0.2,-0.2};
            \node at (-0.25,-0.6) {\strandlabel{\color{Mahogany} 2k-2r}};
            \zebras{0.2,-0.2}{west}{2r};
        \end{tikzpicture}
        -  z\sum_{r=1}^k
        \begin{tikzpicture}[centerzero]
            \draw[->] (-0.6,0.2) -- (-0.6,-0.2) arc(180:360:0.2) \braidup (0.2,0.2) -- (0.2,0.5);
            \draw[wipe] (0.2,-0.5) -- (0.2,-0.2) \braidup (-0.2,0.2) arc(0:180:0.2);
            \draw (0.2,-0.5) -- (0.2,-0.2) \braidup (-0.2,0.2) arc(0:180:0.2);
            \zebra{-0.2,-0.2};
            \node at (-0.25,-0.6) {\strandlabel{\color{Mahogany} 2k-2r+1}};
            \zebras{0.2,-0.2}{west}{1-2r};
        \end{tikzpicture} .
    \]
    Next, note that, for $1 \le r \le k-1$,
    \[
        \begin{tikzpicture}[centerzero]
            \draw[->] (-0.6,0.2) -- (-0.6,-0.2) arc(180:360:0.2) \braidup (0.2,0.2) -- (0.2,0.5);
            \draw[wipe] (0.2,-0.5) -- (0.2,-0.2) \braidup (-0.2,0.2) arc(0:180:0.2);
            \draw (0.2,-0.5) -- (0.2,-0.2) \braidup (-0.2,0.2) arc(0:180:0.2);
            \zebra{-0.2,-0.2};
            \node at (-0.25,-0.6) {\strandlabel{\color{Mahogany} 2k-2r}};
            \zebras{0.2,-0.2}{west}{2r};
        \end{tikzpicture}
        \overset{\cref{skein}}{=}
        \begin{tikzpicture}[centerzero]
            \draw (0.2,-0.5) -- (0.2,-0.2) \braidup (-0.2,0.2) arc(0:180:0.2);
            \draw[wipe] (-0.6,0.2) -- (-0.6,-0.2) arc(180:360:0.2) \braidup (0.2,0.2) -- (0.2,0.5);
            \draw[->] (-0.6,0.2) -- (-0.6,-0.2) arc(180:360:0.2) \braidup (0.2,0.2) -- (0.2,0.5);
            \zebra{-0.2,-0.2};
            \node at (-0.25,-0.6) {\strandlabel{\color{Mahogany} 2k-2r}};
            \zebras{0.2,-0.2}{west}{2r};
        \end{tikzpicture}
        - z\ \leftzebrabub{2k-2r}
        \begin{tikzpicture}[centerzero]
            \draw[->] (0,-0.5) -- (0,0.5);
            \zebras{0,0}{west}{2r};
        \end{tikzpicture}
        \overset{\cref{flipflop}}{\underset{\cref{tapioca}}{=}}
        \begin{tikzpicture}[centerzero]
            \draw[->] (0,-0.5) -- (0,0.5);
            \zebras{0,0}{west}{2k};
        \end{tikzpicture}
        - z \sum_{s=0}^{k-r-1} \leftzebrabub{2k-2r-2s}
        \begin{tikzpicture}[centerzero]
            \draw[->] (0,-0.5) -- (0,0.5);
            \zebras{0,0}{west}{2r+2s};
        \end{tikzpicture} .
    \]
    Similarly, for $1 \le r \le k$,
    \[
        \begin{tikzpicture}[centerzero]
            \draw[->] (-0.6,0.2) -- (-0.6,-0.2) arc(180:360:0.2) \braidup (0.2,0.2) -- (0.2,0.5);
            \draw[wipe] (0.2,-0.5) -- (0.2,-0.2) \braidup (-0.2,0.2) arc(0:180:0.2);
            \draw (0.2,-0.5) -- (0.2,-0.2) \braidup (-0.2,0.2) arc(0:180:0.2);
            \zebra{-0.2,-0.2};
            \node at (-0.25,-0.6) {\strandlabel{\color{Mahogany} 2k-2r+1}};
            \zebras{0.2,-0.2}{west}{1-2r};
        \end{tikzpicture}
        =
        \begin{tikzpicture}[centerzero]
            \draw[->] (-0.6,0.2) -- (-0.6,-0.2) arc(180:360:0.2) \braidup (0.2,0.2) -- (0.2,0.5);
            \draw[wipe] (0.2,-0.5) -- (0.2,-0.2) \braidup (-0.2,0.2) arc(0:180:0.2);
            \draw (0.2,-0.6) -- (0.2,-0.2) \braidup (-0.2,0.2) arc(0:180:0.2);
            \zebra{-0.2,-0.2};
            \node at (-0.25,-0.6) {\strandlabel{\color{Mahogany} 2k-2r}};
            \zebras{0.2,-0.2}{west}{2r-2};
            \opendot{0.2,0.3};
            \token{0.2,-0.4};
        \end{tikzpicture}
        =
        \begin{tikzpicture}[centerzero]
            \draw[->] (0,-0.5) -- (0,0.5);
            \zebras{0,0}{west}{-2k};
        \end{tikzpicture}
        - z \sum_{s=0}^{k-r-1} \leftzebrabub{2k-2r-2s}
        \begin{tikzpicture}[centerzero]
            \draw[->] (0,-0.5) -- (0,0.5);
            \zebras{0,0}{west}{-2r-2s};
        \end{tikzpicture} ,
    \]
    where the last sum is zero when $r=k$.  We also have
    \[
        \begin{tikzpicture}[centerzero]
            \draw (-0.3,0) arc(180:360:0.3);
            \draw[wipe] (0.3,-0.5) to[out=110,in=down] (0,0) to[out=up,in=250] (0.3,0.5);
            \draw[->] (0.3,-0.5) to[out=110,in=down] (0,0) to[out=up,in=250] (0.3,0.5);
            \draw[wipe] (0.3,0) arc(0:180:0.3);
            \draw[->] (0.3,0) arc(0:180:0.3);
            \zebras{0.3,0}{west}{2k};
        \end{tikzpicture}
        \overset{\cref{skein}}{=}
        \begin{tikzpicture}[centerzero]
            \draw[->] (0.3,-0.5) to[out=110,in=down] (0,0) to[out=up,in=250] (0.3,0.5);
            \draw[wipe] (-0.3,0) arc(-180:180:0.3);
            \draw[->] (-0.3,0) arc(-180:180:0.3);
            \zebras{0.3,0}{west}{2k};
        \end{tikzpicture}
        - z\
        \begin{tikzpicture}[anchorbase]
            \draw[->] (-0.4,0) to[out=down,in=left] (-0.25,-0.15) to[out=right,in=down] (0,0.4);
            \draw[wipe] (0,-0.4) to[out=up,in=right] (-0.25,0.15);
            \draw (0,-0.6) -- (0,-0.4) to[out=up,in=right] (-0.25,0.15) to[out=left,in=up] (-0.4,0);
            \zebras{0,-0.4}{west}{2k};
        \end{tikzpicture}
        =
        \begin{tikzpicture}[centerzero]
            \draw[->] (0,-0.5) -- (0,0.5);
        \end{tikzpicture}
        \leftzebrabub{2k}
        - z\
        \begin{tikzpicture}[centerzero]
            \draw[->] (0,-0.5) -- (0,0.5);
            \zebras{0,0}{west}{2k};
        \end{tikzpicture}
        .
    \]
    Therefore,
    \begin{align*}
        \leftzebrabub{2k}
        \begin{tikzpicture}[centerzero]
            \draw[->] (0,-0.5) -- (0,0.5);
        \end{tikzpicture}
        &=
        \begin{tikzpicture}[centerzero]
            \draw[->] (0,-0.5) -- (0,0.5);
        \end{tikzpicture}
        \leftzebrabub{2k}        - k z
        \left(
            \begin{tikzpicture}[centerzero]
                \draw[->] (0,-0.5) -- (0,0.5);
                \zebras{0,0}{west}{2k};
            \end{tikzpicture}
            +
            \begin{tikzpicture}[centerzero]
                \draw[->] (0,-0.5) -- (0,0.5);
                \zebras{0,0}{west}{-2k};
            \end{tikzpicture}
        \!\!\right)
        + z^2 \sum_{r=1}^{k-1} \sum_{s=0}^{k-r-1}  \leftzebrabub{2k-2r-2s}
        \left(
            \begin{tikzpicture}[centerzero]
                \draw[->] (0,-0.5) -- (0,0.5);
                \zebras{0,0}{west}{2r+2s};
            \end{tikzpicture}
            +
            \begin{tikzpicture}[centerzero]
                \draw[->] (0,-0.5) -- (0,0.5);
                \zebras{0,0}{west}{-2r-2s};
            \end{tikzpicture}
        \!\!\right)
        \\
        &=
        \begin{tikzpicture}[centerzero]
            \draw[->] (0,-0.35) -- (0,0.35);
        \end{tikzpicture}
        \leftzebrabub{2k}
        - kz
        \left(
            \begin{tikzpicture}[centerzero]
                \draw[->] (0,-0.35) -- (0,0.35);
                \zebras{0,0}{west}{2k};
            \end{tikzpicture}
            +
            \begin{tikzpicture}[centerzero]
                \draw[->] (0,-0.35) -- (0,0.35);
                \zebras{0,0}{west}{-2k};
            \end{tikzpicture}
        \!\!\right)
        + z^2 \sum_{r=1}^{k-1} r\ \leftzebrabub{2k-2r}
        \left(
            \begin{tikzpicture}[centerzero]
                \draw[->] (0,-0.35) -- (0,0.35);
                \zebras{0,0}{west}{2r};
            \end{tikzpicture}
            +
            \begin{tikzpicture}[centerzero]
                \draw[->] (0,-0.35) -- (0,0.35);
                \zebras{0,0}{west}{-2r};
            \end{tikzpicture}
        \!\!\right).
        \qedhere
    \end{align*}
\end{proof}

For any two objects $X,Y \in \AQcat(z)$, the morphism space $\Hom_{\AQcat(z)}(X,Y)$ is a right $\Sym$-supermodule with action given by
\[
    f \cdot a := f \otimes \beta(a),\qquad
    f \in \Hom_{\AQcat(z)}(X,Y),\quad a \in \Sym.
\]
As in \cref{sec:Qdef}, for each $(X,Y)$, fix a set $B(X,Y)$ consisting of a choice of positive reduced lift for each $(X,Y)$-matching.  Then let $B_{\zebralone}(X,Y)$ denote the set of all morphisms that can be obtained from the elements of $B(X,Y)$ by adding a zebra, labelled by some integer (possibly zero) near the terminus of each string.  We require that all zebras occurring on strands whose terminus is at the top of the diagram to be at the same height; similarly we require that all zebras occurring on strands whose terminus is at the bottom of the diagram to be at the same height, and below those zebras on strands whose terminus is at the top of the diagram.

\begin{prop} \label{affinespan}
    For any objects $X,Y$ of $\AQcat(z)$, the set $B_{\zebralone}(X,Y)$ spans the morphism space $\Hom_{\AQcat(z)}(X,Y)$ as a right $\Sym$-supermodule.
\end{prop}

\begin{proof}
    Since this type of argument is standard in categorical representation theory, we only give a sketch of the proof.  (See also the proof of \cref{nonaffinespan}.)  We have the Reidemeister relations, a skein relation, and bubble and zebra sliding relations.  These allow diagrams for morphisms in $\AQcat(z)$ to be transformed in a way similar to the way oriented tangles are simplified in skein categories.  Hence, there is a straightening algorithm to rewrite any diagram representing a morphism $X \to Y$ as a linear combination of the ones in $B_{\zebralone}(X,Y)$.
\end{proof}

\begin{conj} \label{affinebasis}
    For any objects $X,Y$ of $\AQcat(z)$, the morphism space $\Hom_{\AQcat(z)}(X,Y)$ is a free right $\Sym$-supermodule with basis $B_{\zebralone}(X,Y)$.
\end{conj}

As noted in the introduction, we expect that \cref{affinebasis} could be proved using the categorical comultiplication technique of \cite{BSW-qheis}, after introducing the more general \emph{quantum isomeric Heisenberg supercategory}.

\begin{prop} \label{collapse}
    There is a unique monoidal superfunctor
    \[
        \bC \colon \AQcat(z) \to \SEnd(\Qcat(z))
    \]
    defined as follows.  On objects $X \in \AQcat(z)$ and morphisms $f \in \{\posupcross,\negupcross,\posrightcross,\leftcap,\leftcup,\tokup\}$,
    \[
        \bC(X) = X \otimes -
        \qquad \text{and} \qquad
        \bC(f) = f \otimes -.
    \]
    In addition, $\bC(\dotup)$ is the natural transformation $\uparrow \otimes - \to\ \uparrow \otimes -$ whose $X$-component, $X \in \Qcat(z)$, is
    \begin{equation} \label{cbatch}
        \bC(\dotup)_X :=
        \begin{tikzpicture}[centerzero]
            \draw[->] (-0.2,-0.4) \braidup (0.2,0) \braidup (-0.2,0.4);
            \draw[wipe] (0.2,-0.4) \braidup (-0.2,0) \braidup (0.2,0.4);
            \draw[thick] (0.2,-0.4) node[anchor=north] {\strandlabel{X}} \braidup (-0.2,0) \braidup (0.2,0.4);
            \token{0.2,0};
        \end{tikzpicture},
    \end{equation}
    where the thick strand labelled $X$ is the identity morphism $1_X$ of $X$.
\end{prop}

\begin{proof}
    Naturality of $\bC(f)$ is clear for $f \in \{\posupcross,\negupcross,\posrightcross,\leftcap,\leftcup,\tokup\}$.  For $f = \dotup$, it follows from the fact that the generating morphisms \cref{Qgen1,Qgen2} slide over crossings.

    All the relations appearing in \cref{Qdef} are clearly respected by $\bC$.  It remains to verify the relations \cref{AQrels}.  The first relation is straightforward.  For the second relation, we compute (dropping the label $X$ on the thick strand)
    \[
        \bC
        \left(
            \begin{tikzpicture}[centerzero]
                \draw[->] (0.3,-0.3) -- (-0.3,0.3);
                \draw[wipe] (-0.3,-0.3) -- (0.3,0.3);
                \draw[->] (-0.3,-0.3) -- (0.3,0.3);
                \opendot{0.17,-0.17};
            \end{tikzpicture}
        \right)
        = \bC
        \left(
            \begin{tikzpicture}[anchorbase]
                \draw[->] (0,0) \braidup (0.3,0.4) \braidup (-0.3,1.2);
                \draw[wipe] (-0.3,0) \braidup (0,1.2);
                \draw[->] (-0.3,0) \braidup (0,1.2);
                \draw[wipe] (0.3,0) \braidup (0,0.4) \braidup (0.3,1) -- (0.3,1.2);
                \draw[thick] (0.3,0) \braidup (0,0.4) \braidup (0.3,1) -- (0.3,1.2);
                \token{0.3,0.4};
            \end{tikzpicture}
        \right)
        = \bC
        \left(
            \begin{tikzpicture}[anchorbase]
                \draw[->] (0,0) \braidup (-0.3,0.4) \braidup (0.3,0.8) \braidup (-0.3,1.2);
                \draw[wipe] (-0.3,0) \braidup (0,0.4) \braidup (-0.3,0.8) \braidup (0,1.2);
                \draw[->] (-0.3,0) \braidup (0,0.4) \braidup (-0.3,0.8) \braidup (0,1.2);
                \draw[wipe] (0.3,0) -- (0.3,0.4) \braidup (0.1,0.8) \braidup (0.3,1.2);
                \draw[thick] (0.3,0) -- (0.3,0.4) \braidup (0.1,0.8) \braidup (0.3,1.2);
                \token{0.3,0.8};
            \end{tikzpicture}
        \right)
        =
        \bC
        \left(
            \begin{tikzpicture}[centerzero]
                \draw[->] (0.3,-0.3) -- (-0.3,0.3);
                \draw[wipe] (-0.3,-0.3) -- (0.3,0.3);
                \draw[->] (-0.3,-0.3) -- (0.3,0.3);
                \opendot{-0.15,0.15};
            \end{tikzpicture}
        \right).
    \]
    Finally, for the last relation in \cref{AQrels}, we compute
    \[
        \bC
        \left(
            \begin{tikzpicture}[centerzero]
                \bubright{0,0};
                \opendot{-0.2,0};
            \end{tikzpicture}
        \right)
        =
        \bC
        \left(
            \begin{tikzpicture}[anchorbase]
                \draw (0,0.4) arc(360:180:0.15) \braidup (0.3,0.8) \braidup (-0.3,1.2) arc(180:0:0.15);
                \draw[wipe] (0,0.4) \braidup (-0.3,0.8) \braidup (0,1.2);
                \draw[->] (0,1.2) \braiddown (-0.3,0.8);
                \draw (-0.3,0.8) \braiddown (0,0.4);
                \draw[wipe] (0.3,0.1) \braidup (0.1,0.8) \braidup (0.3,1.5);
                \draw[thick] (0.3,0.1) \braidup (0.1,0.8) \braidup (0.3,1.5);
                \token{0.3,0.8};
            \end{tikzpicture}
        \right)
        =
        \bC
        \left(
            \begin{tikzpicture}[centerzero]
                \draw[->] (-0.3,0) arc(-180:180:0.3);
                \draw[wipe] (0.3,-0.4) \braidup (0.1,0) \braidup (0.3,0.4);
                \draw[thick] (0.3,-0.7) -- (0.3,-0.4) \braidup (0.1,0) \braidup (0.3,0.4) -- (0.3,0.7);
                \token{0.3,0};
            \end{tikzpicture}
        \right)
        \overset{\cref{redslip}}{=}
        \bC
        \left(
            \begin{tikzpicture}[centerzero]
                \draw[->] (0,0.3) arc(90:450:0.3);
                \draw[wipe] (-0.3,-0.4) \braidup (-0.1,0) \braidup (-0.3,0.4);
                \draw[thick] (-0.3,-0.7) -- (-0.3,-0.4) \braidup (-0.1,0) \braidup (-0.3,0.4) -- (-0.3,0.7);
                \token{0.3,0};
            \end{tikzpicture}
        \right)
        \overset{\cref{rouge}}{=}
        \bC
        \left(\,
            \begin{tikzpicture}[centerzero]
                \draw[->] (0,0.3) arc(90:450:0.3);
                \draw[thick] (-0.4,-0.7) -- (-0.4,0.7);
                \token{0.3,0};
            \end{tikzpicture}
        \right)
        \overset{\cref{iron}}{=} 0.
        \qedhere
    \]
\end{proof}

The superfunctor $\bC$, which we will call the \emph{collapsing superfunctor}, should viewed as an odd analogue of the one appearing in \cite[Th.~3.2]{MS21}, which describes actions of the \emph{affinization} of a braided monoidal category.  In that setting, the analogue of the open Clifford token is the affine dot, which acts as
\[
    \begin{tikzpicture}[centerzero]
        \draw[thick] (0,-0.6) -- (0,0);
        \draw (0.75,0) to[out=-90,in=0] (0.65,-0.1) to[out=180,in=-90] (0.5,0.1);
        \draw[wipe] (-0.5,-0.6) to[out=45,in=down] (0.5,-0.1) to[out=90,in=180] (0.65,0.1) to[out=0,in=90] (0.75,0);
        \draw (-0.5,-0.6) to[out=45,in=down] (0.5,-0.1) to[out=90,in=180] (0.65,0.1) to[out=0,in=90] (0.75,0);
        \draw[->] (0.5,0.1) to[out=up,in=-45] (-0.5,0.6);
        \draw[wipe] (0,0) -- (0,0.6);
        \draw[thick] (0,0) -- (0,0.6);
    \end{tikzpicture}
    \ .
\]
See \cref{oddaffinization} for additional discussion.

\section{Affine endomorphism superalgebras\label{sec:affineEnd}}

In this section we describe the relationship between the endomorphism superalgebras in the quantum affine isomeric supercategory and affine Hecke--Clifford superalgebras.  We also use the collapsing superfunctor of \cref{collapse} to explain how the Jucys--Murphy elements in the Hecke--Clifford superalgebra arise naturally in this context.  Throughout this section, $\kk$ is an arbitrary commutative ring of characteristic not equal to two, and $z \in \kk$.

\begin{defin} \label{AHCdef}
    For $r \in \Z_{>0}$ and $z \in \kk$, let $\AHC_r(z)$ denote the associative superalgebra generated by
    \begin{center}
        even elements $t_1,\dotsc,t_{r-1}$ and odd elements $\cg_1,\dotsc,\cg_r,\varpi_1,\dotsc,\varpi_r$,
    \end{center}
    satisfying the following relations (for $i,j$ in the allowable range):
    \begin{align}
        t_i^2 &= z t_i + 1,
        \\
        t_i t_{i+1} t_i &= t_{i+1} t_i t_{i+1},&
        t_i t_j &= t_j t_i,& |i-j| > 1,
        \\
        \cg_i^2 &= -1,&
        \cg_i \cg_j &= - \cg_j \cg_i,& i \ne j,
        \\ \label{toad}
        \varpi_i^2 &= -1,&
        \varpi_i \varpi_j &= - \varpi_j \varpi_i,& i \ne j,
        \\ \label{bike}
        t_i \cg_i &= \cg_{i+1} t_i,&
        t_i \cg_j &= \cg_j t_i,& j \ne i,i+1,
        \\ \label{stool}
        t_i \varpi_{i+1} &= \varpi_i t_i,&
        t_i \varpi_j &= \varpi_j t_i,& j \ne i,i+1.
    \end{align}
    Equivalently, $\AHC_r(z)$ is the associative superalgebra generated by $\HC_r(z)$, together with odd elements $\varpi_1,\dotsc,\varpi_r$, subject to the relations \cref{toad,stool}.
\end{defin}

Multiplying both sides of the first relation in \cref{stool} on the left and right by $t_i^{-1} = t_i-z$ gives
\begin{equation} \label{mushy}
    t_i \varpi_i = \varpi_{i+1} t_i + z(\varpi_i - \varpi_{i+1}).
\end{equation}
The next result shows that $\AHC_r(z)$ is isomorphic to the \emph{affine Hecke--Clifford superalgebra}, which was first introduced in \cite[\S3]{JN99}, where it was called the \emph{affine Sergeev algebra}.

\begin{lem} \label{milk}
    For $r \in \Z_{>0}$ and $z \in \kk$, $\AHC_r(z)$ is isomorphic to the associative superalgebra  $\AHC_r'(z)$ generated by $\HC_r(z)$, together with pairwise-commuting invertible even elements $x_1,\dotsc,x_r$, subject to the following relations (for $i,j$ in the allowable range):
    \begin{align*}
        t_i x_i &= x_{i+1} t_i - z (x_{i+1} - \cg_i \cg_{i+1} x_i), \\
        t_i x_{i+1} &= x_i t_i + z(1 + \cg_i \cg_{i+1}) x_{i+1}, \\
        t_i x_j &= x_j t_i,& j \ne i,i+1, \\
        \cg_i x_i &= x_i^{-1} \cg_i, \\
        x_i x_j &= x_j x_i,& j \ne i.
    \end{align*}
    The isomorphism is given by
    \begin{equation} \label{thwomp}
        \AHC'_r(z) \xrightarrow{\cong} \AHC_r(z), \qquad
        t_i \mapsto t_i,\quad
        \cg_i \mapsto \cg_i,\quad
        x_i \mapsto \cg_i \varpi_i.
    \end{equation}
\end{lem}

\begin{proof}
    It is a straightforward exercise to verify that \cref{thwomp} respects the defining relations of $\AHC'_r(z)$.
    \details{
        Let $\tilde{x}_i = \cg_i \varpi_i$, so that $\tilde{x}_i^{-1} = \varpi_i \cg_i$.  We have
        \begin{gather*}
            t_i \tilde{x}_i
            = \cg_{i+1} t_i \varpi_i
            \overset{\cref{mushy}}{=}  \tilde{x}_{i+1} t_i + z \cg_{i+1}(\varpi_i-\varpi_{i+1})
            = \tilde{x}_{i+1} t_i - z (\tilde{x}_{i+1} - \cg_i \cg_{i+1} \tilde{x}_i)
            \\
            t_i \tilde{x}_{i+1}
            \overset{\cref{upside}}{=} \cg_i t_i \varpi_{i+1} + z(\cg_{i+1} - \cg_i) \varpi_{i+1}
            \overset{\cref{stool}}{=} \tilde{x}_i t_i + z(1+\cg_i\cg_{i+1}) \tilde{x}_{i+1},
            \\
            \cg_i \tilde{x}_i
            = - \varpi_i
            = \tilde{x}_i^{-1} \cg.
        \end{gather*}
        The remaining relations are easy to check.
    }
    Thus the map \cref{thwomp} is well-defined homomorphism of superalgebras.  It is invertible, with inverse
    \[
        \AHC_r(z) \to \AHC'_r(z),\qquad
        t_i \mapsto t_i,\quad
        \cg_i \mapsto \cg_i,\quad
        \varpi_i \mapsto - \cg_i x_i.
        \qedhere
    \]
\end{proof}

In light of \cref{milk}, we will simply refer to $\AHC_r(z)$ as the affine Hecke--Clifford superalgebra.

\begin{prop} \label{crossbow}
    For $r \in \N$, we have a homomorphism of associative superalgebras
    \[
        \AHC_r(z) \to \End_{\AQcat(z)}(\uparrow^{\otimes r})
    \]
    given by
    \begin{align*}
        t_i &\mapsto\, \uparrow^{\otimes (i-1)} \otimes \posupcross \otimes \uparrow^{\otimes (r-i-1)},& 1 \le i \le r-1,
        \\
        \cg_i &\mapsto\, \uparrow^{\otimes (i-1)} \otimes\, \tokup\, \otimes \uparrow^{\otimes (r-i)},& 1 \le i \le r,
        \\
        \varpi_i &\mapsto\, \uparrow^{\otimes (i-1)} \otimes\, \dotup\, \otimes \uparrow^{\otimes (r-i)},& 1 \le i \le r.
    \end{align*}
\end{prop}

\begin{proof}
    It is a straightforward computation to verify that the given map is well-defined, i.e.\ that it respects the relations in \cref{AHCdef}.
\end{proof}

Note that, under the homomorphism of \cref{crossbow}, we have
\[
    x_i \mapsto\ \uparrow^{\otimes (i-1)} \otimes\, \zebraup\, \otimes \uparrow^{\otimes (r-i)},\qquad 1 \le i \le r.
\]
The difference between the new presentation of the affine Hecke--Clifford superalgebra given in \cref{AHCdef} and the one in \cref{milk} that has appeared previously in the literature is that the former presentation involves the odd generators $\varpi_i$, whereas the latter involves the even generator $x_i = \cg_i \varpi_i$.  We prefer the presentation of \cref{AHCdef} since the relations are simpler and a natural symmetry of $\AHC_r(z)$ becomes apparent.  In particular, we have an automorphism of $\AHC_r(z)$ given by
\begin{equation} \label{yinyang}
    \cg_i \mapsto \varpi_i,\qquad
    \varpi_i \mapsto \cg_i,\qquad
    t_j \mapsto - t_j^{-1},\qquad
    1 \le i \le r,\ 1 \le j \le r-1.
\end{equation}

For the remainder of this section, we reverse our numbering convention for strands in diagrams; see \cref{numbering}.  The composite of the map of \cref{crossbow} with the automorphism of $\End_{\AQcat(z)}(\uparrow^{\otimes r})$ induced by the superfunctor $\Omega_\leftrightarrow$ yields a homomorphism of associative superalgebras
\[
    \imath_r \colon \AHC_r(z) \to \End_{\AQcat(z)}(\uparrow^{\otimes r})
\]
given by
\begin{align*}
    t_i &\mapsto\, -\uparrow^{\otimes (r-i-1)} \otimes \negupcross \otimes \uparrow^{\otimes (i-1)},& 1 \le i \le r-1,
    \\
    \cg_i &\mapsto\, \uparrow^{\otimes (r-i)} \otimes\, \tokup\, \otimes \uparrow^{\otimes (i-1)},& 1 \le i \le r,
    \\
    \varpi_i &\mapsto\, \uparrow^{\otimes (r-i)} \otimes\, \dotup\, \otimes \uparrow^{\otimes (i-1)},& 1 \le i \le r.
\end{align*}
This is also equal to the automorphism \cref{yinyang} followed by the map of \cref{crossbow}.  We have
\[
    \imath_r(x_i) =\ \uparrow^{\otimes (r-i)} \otimes\, \zebraup\, \otimes \uparrow^{\otimes (i-1)},\qquad 1 \le i \le r.
\]

The \emph{Jucys--Murphy elements} $J_1,\dotsc,J_r \in \HC_r(z)$ were defined recursively in \cite[(3.10)]{JN99} by
\[
    J_i :=
    \begin{cases}
        1 & \text{for } i=1, \\
        (t_{i-1} - z \cg_{i-1} \cg_i) J_{i-1} t_{i-1} & \text{for } i=2,\dotsc,n.
    \end{cases}
\]
For $1 \le i \le r$, define the \emph{odd Jucys--Murphy elements}
\begin{equation} \label{JModd}
    J_i^\odd = t_{i-1}^{-1} \dotsm t_2^{-1} t_1^{-1} \cg_1 t_1 t_2 \dotsm t_{i-1} \in \HC_r(z),
\end{equation}
where, by convention, we have $J_1^\odd = \cg_1$.  The following result gives a direct (i.e.\ non-recursive) expression for the even Jucys--Murphy elements.

\begin{lem}
    For all $1 \le i \le r$, we have $J_i = -\cg_i J_i^\odd$.
\end{lem}

\begin{proof}
    We prove the result by induction on $i$.  Since $- \cg_1 J_1^\odd = -\cg_1^2 = 1 = J_1$, the result holds for $i=1$.  Now suppose that $i>1$ and that $J_{i-1} = - \cg_{i-1} J_{i-1}^\odd$.  First note that
    \[
        \cg_i t_{i-1}^{-1}
        = \cg_i (t_{i-1}-z)
        \overset{\cref{bike}}{=} t_{i-1} \cg_{i-1} - z \cg_i.
    \]
    Thus, we have
    \[
        - \cg_i J_i^\odd
        = - \cg_i t_{i-1}^{-1} J_{i-1}^\odd t_{i-1}
        = - (t_{i-1} \cg_{i-1} - z \cg_i) J_{i-1}^\odd t_{i-1}
        = (t_{i-1} - z \cg_{i-1} \cg_i) J_{i-1} t_{i-1}
        = J_i.
        \qedhere
    \]
\end{proof}

Evaluation on the unit object $\one$ yields a superfunctor
\[
    \Ev_\one \colon \SEnd(\Qcat(z)) \to \Qcat(z).
\]
Note that this is \emph{not} a monoidal superfunctor.  Recall the collapsing superfunctor $\bC$ of \cref{collapse}.

\begin{prop} \label{Guiness}
    For $1 \le i \le r$, we have
    \begin{align*}
        \Ev_\one \circ \bC \left( \uparrow^{\otimes (r-i)} \otimes\, \dotup\, \otimes \uparrow^{\otimes (i-1)} \right)
        &= \imath_r \left( J_i^\odd \right),
        \\
        \Ev_\one \circ \bC \left( \uparrow^{\otimes (r-i)} \otimes\, \zebraup\, \otimes \uparrow^{\otimes (i-1)} \right)
        &= - \imath_r \left( J_i \right).
    \end{align*}
\end{prop}

\begin{proof}
    We have
    \[
        \Ev_\one \circ \bC \left( \uparrow^{\otimes (r-i)} \otimes\, \dotup\, \otimes \uparrow^{\otimes (i-1)} \right)
        \overset{\cref{cbatch}}{=}
        \begin{tikzpicture}[centerzero]
            \draw[->] (0,-0.5) \braidup (1.3,0) \braidup (0,0.5);
            \token{1.3,0};
            \node at (0.67,0) {$\cdots$};
            \draw[wipe] (0.3,-0.5) -- (0.3,0.5);
            \draw[->] (0.3,-0.5) -- (0.3,0.5);
            \draw[wipe] (1,-0.5) -- (1,0.5);
            \draw[->] (1,-0.5) -- (1,0.5);
            \draw[->] (-0.3,-0.5) -- (-0.3,0.5);
            \draw[->] (-1,-0.5) -- (-1,0.5);
            \node at (-0.65,0) {$\cdots$};
        \end{tikzpicture}
        = \imath_r(J_i^\odd),
    \]
    where it is the $i$-th strand from the right that passes under other strands.  The proof of the second equality in the statement of the proposition follows after adding a closed Clifford token to the top of the $i$-th strand from the right.
\end{proof}

\begin{rem} \label{oddaffinization}
    Recall that the $i$-th Jucys--Murphy element in the Iwahori--Hecke algebra of type $A$ is given, in terms of string diagrams, by
    \[
        \begin{tikzpicture}[centerzero]
            \draw[->] (1.3,0) \braidup (0,0.5);
            \node at (0.67,0) {$\cdots$};
            \draw[wipe] (0.3,-0.5) -- (0.3,0.5);
            \draw[->] (0.3,-0.5) -- (0.3,0.5);
            \draw[wipe] (1,-0.5) -- (1,0.5);
            \draw[->] (1,-0.5) -- (1,0.5);
            \draw[->] (-0.3,-0.5) -- (-0.3,0.5);
            \draw[->] (-1,-0.5) -- (-1,0.5);
            \node at (-0.65,0) {$\cdots$};
            \draw[wipe] (0,-0.5) \braidup (1.3,0);
            \draw (0,-0.5) \braidup (1.3,0);
        \end{tikzpicture}
        \ ,
    \]
    where it is the $i$-th strand from the right that loops around other strands.  See \cite[\S 6]{MS21} for a discussion of Jucys--Murphy elements in a more general setting, related to the \emph{affinization} of braided monoidal categories.  The above discussion suggests there may be a general notion of \emph{odd affinization}, where the above diagram is replaced by the one appearing in the proof of \cref{Guiness}.
\end{rem}

The next result shows that we can naturally view $\Qcat(z)$ as a subcategory of $\AQcat(z)$.

\begin{prop} \label{feline}
    The superfunctor $\Qcat(z) \to \AQcat(z)$ that is the identity on objects and sends each generating morphism in $\Qcat(z)$ to the morphism in $\AQcat(z)$ depicted by the same string diagram is faithful.
\end{prop}

\begin{proof}
    It is straightforward to verify that $\Ev_\one \circ \bC$ is left inverse to the superfunctor in the statement of the proposition.
\end{proof}

\section{The affine action superfunctor}

In this final section, we define an action of $\AQcat(z)$ on the category of $U_q$-supermodules.  We then use this action to define a sequence of elements in the center of $U_q$.  Throughout this section we assume that $\kk = \C(q)$ and $z=q-q^{-1}$.

The image of $\uparrow \redobject\,$ under the superfunctor $\redFunc_n$ of \cref{redfunctor} is the superfunctor
\[
    \redFunc_n(\uparrow \redobject\,) = V \otimes - \colon U_q\smod \to U_q\smod
\]
of tensoring on the left with $V$.  Define the natural transformation
\begin{equation} \label{carrot}
    K := \redFunc_n
    \left(
        \begin{tikzpicture}[centerzero]
            \draw[->] (-0.2,-0.4) \braidup (0.2,0) \braidup (-0.2,0.4);
            \draw[wipe] (0.2,-0.4) \braidup (-0.2,0) \braidup (0.2,0.4);
            \draw[thick,red] (0.2,-0.4) \braidup (-0.2,0) \braidup (0.2,0.4);
            \token{0.2,0};
        \end{tikzpicture}
    \right)
    \colon V \otimes - \to V \otimes -.
\end{equation}
Thus the $M$-component of $K$, for $M \in U_q\smod$, is the $U_q$-supermodule homomorphism
\[
    K_M =
    \begin{tikzpicture}[centerzero]
        \draw[->] (-0.2,-0.4) \braidup (0.2,0) \braidup (-0.2,0.4);
        \draw[wipe] (0.2,-0.4) \braidup (-0.2,0) \braidup (0.2,0.4);
        \draw[thick,red] (0.2,-0.4) node[anchor=north] {\strandlabel{M}} \braidup (-0.2,0) \braidup (0.2,0.4);
        \token{0.2,0};
    \end{tikzpicture}
    = T_{MV} \circ (1 \otimes J) \circ T_{MV}^{-1}
    \colon V \otimes M \to V \otimes M.
\]
It is straightforward to verify that $K^2=-\id$, where $\id$ is the identity natural transformation.

\begin{theo} \label{butter}
    There is a unique monoidal superfunctor
    \[
        \bFh_n \colon \AQcat(z) \to \SEnd(U_q\smod),
    \]
    such that
    \[
        \bFh_n|_{\Qcat(z)} = \redFunc_n(- \otimes \, \redobject\,)
        \qquad \text{and} \qquad
        \bFh_n(\dotup) = \redFunc_n
        \left(
            \begin{tikzpicture}[centerzero]
                \draw[->] (-0.2,-0.4) \braidup (0.2,0) \braidup (-0.2,0.4);
                \draw[wipe] (0.2,-0.4) \braidup (-0.2,0) \braidup (0.2,0.4);
                \draw[thick,red] (0.2,-0.4) \braidup (-0.2,0) \braidup (0.2,0.4);
                \token{0.2,0};
            \end{tikzpicture}
        \right)
        \overset{\cref{carrot}}{=} K.
    \]
\end{theo}

\begin{proof}
    The proof is almost identical to that of \cref{collapse}; one merely replaces the thick black strand there (representing the identity morphism $1_X$) with a thick red strand.
    \details{
        All the relations appearing in \cref{Qdef} are respected by $\bFh_n$ since they are respected by $\bF_n$, hence by $\redFunc_n$.  It remains to verify the relations \cref{AQrels}.  The first relation holds since $K^2=-\id$.  For the second relation in \cref{AQrels}, we compute
        \[
            \bFh_n
            \left(
                \begin{tikzpicture}[centerzero]
                    \draw[->] (0.3,-0.3) -- (-0.3,0.3);
                    \draw[wipe] (-0.3,-0.3) -- (0.3,0.3);
                    \draw[->] (-0.3,-0.3) -- (0.3,0.3);
                    \opendot{0.17,-0.17};
                \end{tikzpicture}
            \right)
            = \redFunc_n
            \left(
                \begin{tikzpicture}[anchorbase]
                    \draw[->] (0,0) \braidup (0.3,0.4) \braidup (-0.3,1.2);
                    \draw[wipe] (-0.3,0) \braidup (0,1.2);
                    \draw[->] (-0.3,0) \braidup (0,1.2);
                    \draw[wipe] (0.3,0) \braidup (0,0.4) \braidup (0.3,1) -- (0.3,1.2);
                    \draw[thick,red] (0.3,0) \braidup (0,0.4) \braidup (0.3,1) -- (0.3,1.2);
                    \token{0.3,0.4};
                \end{tikzpicture}
            \right)
            = \redFunc_n
            \left(
                \begin{tikzpicture}[anchorbase]
                    \draw[->] (0,0) \braidup (-0.3,0.4) \braidup (0.3,0.8) \braidup (-0.3,1.2);
                    \draw[wipe] (-0.3,0) \braidup (0,0.4) \braidup (-0.3,0.8) \braidup (0,1.2);
                    \draw[->] (-0.3,0) \braidup (0,0.4) \braidup (-0.3,0.8) \braidup (0,1.2);
                    \draw[wipe] (0.3,0) -- (0.3,0.4) \braidup (0.1,0.8) \braidup (0.3,1.2);
                    \draw[thick,red] (0.3,0) -- (0.3,0.4) \braidup (0.1,0.8) \braidup (0.3,1.2);
                    \token{0.3,0.8};
                \end{tikzpicture}
            \right)
            =
            \bFh_n
            \left(
                \begin{tikzpicture}[centerzero]
                    \draw[->] (0.3,-0.3) -- (-0.3,0.3);
                    \draw[wipe] (-0.3,-0.3) -- (0.3,0.3);
                    \draw[->] (-0.3,-0.3) -- (0.3,0.3);
                    \opendot{-0.15,0.15};
                \end{tikzpicture}
            \right).
        \]
        Finally, for the last relation in \cref{AQrels}, we compute
        \[
            \bFh_n
            \left(
                \begin{tikzpicture}[centerzero]
                    \bubright{0,0};
                    \opendot{-0.2,0};
                \end{tikzpicture}
            \right)
            =
            \redFunc_n
            \left(
                \begin{tikzpicture}[anchorbase]
                    \draw (0,0.4) arc(360:180:0.15) \braidup (0.3,0.8) \braidup (-0.3,1.2) arc(180:0:0.15);
                    \draw[wipe] (0,0.4) \braidup (-0.3,0.8) \braidup (0,1.2);
                    \draw[->] (0,1.2) \braiddown (-0.3,0.8);
                    \draw (-0.3,0.8) \braiddown (0,0.4);
                    \draw[wipe] (0.3,0.1) \braidup (0.1,0.8) \braidup (0.3,1.5);
                    \draw[thick,red] (0.3,0.1) \braidup (0.1,0.8) \braidup (0.3,1.5);
                    \token{0.3,0.8};
                \end{tikzpicture}
            \right)
            =
            \redFunc_n
            \left(
                \begin{tikzpicture}[centerzero]
                    \draw[->] (-0.3,0) arc(-180:180:0.3);
                    \draw[wipe] (0.3,-0.4) \braidup (0.1,0) \braidup (0.3,0.4);
                    \draw[thick,red] (0.3,-0.7) -- (0.3,-0.4) \braidup (0.1,0) \braidup (0.3,0.4) -- (0.3,0.7);
                    \token{0.3,0};
                \end{tikzpicture}
            \right)
            \overset{\cref{redslip}}{=}
            \redFunc_n
            \left(
                \begin{tikzpicture}[centerzero]
                    \draw[->] (0,0.3) arc(90:450:0.3);
                    \draw[wipe] (-0.3,-0.4) \braidup (-0.1,0) \braidup (-0.3,0.4);
                    \draw[thick,red] (-0.3,-0.7) -- (-0.3,-0.4) \braidup (-0.1,0) \braidup (-0.3,0.4) -- (-0.3,0.7);
                    \token{0.3,0};
                \end{tikzpicture}
            \right)
            \overset{\cref{rouge}}{=}
            \redFunc_n
            \left(\,
                \begin{tikzpicture}[centerzero]
                    \draw[->] (0,0.3) arc(90:450:0.3);
                    \draw[thick,red] (-0.4,-0.7) -- (-0.4,0.7);
                    \token{0.3,0};
                \end{tikzpicture}
            \right)
            \overset{\cref{iron}}{=} 0.
        \]
    }
\end{proof}

We call the superfunctor $\bFh_n$ the \emph{affine action superfunctor}.  It endows $U_q\smod$ with the structure of an $\AQcat(z)$-supermodule category. Note that
\[
    \bFh_n(\uparrow) = V \otimes -
    \qquad \text{and} \qquad
    \bFh_n(\downarrow) = V^* \otimes -
\]
are the translation endosuperfunctors of $U_q\smod$ given by tensoring on the left with $V$ and $V^*$, respectively.  Thus, combining \cref{crossbow,butter}, we have a homomorphism of associative superalgebras
\[
    \AHC_r(z) \to \End_{U_q}(V^{\otimes r} \otimes M)
\]
for any $U_q$-supermodule $M$ and $r,s \in \N$.  This is a quantum analogue of \cite[Th.~7.4.1]{HKS11}.

Let
\[
    Z_q := \{x \in U_q : xy = (-1)^{\bar{x}\bar{y}} yx \text{ for all } y \in U_q\}
\]
be the center of $U_q$. Evaluation on the identity element of the regular representation defines a canonical superalgebra isomorphism
\[
    \End(\id_{U_q\smod}) \xrightarrow{\cong} Z_q,
\]
where $\id_\cC$ denotes the identity endosuperfunctor of a supercategory $\cC$.  Consider the composite superalgebra homomorphism
\begin{equation} \label{canon}
    \End_{\AQcat(z)}(\one) \xrightarrow{\bFh_n} \End(\id_{U_q\smod}) \xrightarrow{\cong} Z_q.
\end{equation}
Our goal is now to compute the image of this homomorphism.  By \cref{affinespan}, it suffices to compute the image of the zebra bubbles $\leftzebrabub{2k}$, $k > 0$.

We begin with a simplifying computation.  Using \cref{boxslide} and the relations in $\Qcat(z)$ we have, for $k > 0$,
\begin{equation} \label{goomba}
    \bFh_n \left( \leftzebrabub{2k} \right)
    =
    \begin{tikzpicture}[anchorbase]
        \draw (-1,1.5) to[out=down,in=left] (-0.5,0.2) to[out=right,in=down] (0.3,0.6) \braidup (-0.3,1) \braidup (0.3,1.4);
        \node at (0.3,1.8) {$\vdots$};
        \draw[->] (0.3,2) -- (0.3,2.2) \braidup (-0.3,2.6) to[out=up,in=right] (-0.5,2.8) to[out=left,in=up] (-1,1.5);
        \token{0.3,0.6};
        \token{-0.3,1};
        \token{0.3,2.2};
        \token{-0.3,2.6};
        \draw[wipe] (0,0) to (0,3);
        \draw[thick,red] (0,0) to (0,3);
    \end{tikzpicture}
    =
    \begin{tikzpicture}[anchorbase]
        \draw (0.3,0.6) \braidup (-0.3,1) \braidup (0.3,1.4);
        \node at (0.3,1.8) {$\vdots$};
        \draw[->] (-0.5,3.2) to[out=left,in=up] (-0.65,3) to[out=down,in=up] (0.7,2.6) -- (0.7,1.55);
        \draw[wipe] (-0.3,2.6) to[out=up,in=right] (-0.5,3.2);
        \draw (0.3,2) -- (0.3,2.2) \braidup (-0.3,2.6) to[out=up,in=right] (-0.5,3.2);
        \draw (-0.65,0.1) to[out=up,in=down] (0.7,0.1) -- (0.7,1.55);
        \draw[wipe] (0.3,0.6) \braiddown (-0.3,0.1);
        \draw (0.3,0.6) \braiddown (-0.3,0.1) to[out=down,in=right] (-0.5,-0.1) to[out=left,in=down] (-0.65,0.1);
        \token{0.3,0.6};
        \token{-0.3,1};
        \token{0.3,2.2};
        \token{-0.3,2.6};
        \draw[wipe] (0,-0.3) to (0,3.4);
        \draw[thick,red] (0,-0.3) to (0,3.4);
    \end{tikzpicture}
    =
    \begin{tikzpicture}[anchorbase]
        \draw (0.6,1.5) -- (0.6,0.6) to[out=down,in=down,looseness=1.5] (0.3,0.6) \braidup (-0.3,1) \braidup (0.3,1.4);
        \node at (0.3,1.8) {$\vdots$};
        \draw[->] (0.3,2) -- (0.3,2.2) \braidup (-0.3,2.6) to[out=up,in=up] (0.6,2.6) -- (0.6,1.5);
        \token{0.3,0.6};
        \token{-0.3,1};
        \token{0.3,2.2};
        \token{-0.3,2.6};
        \draw[wipe] (0,0.3) to (0,3.1);
        \draw[thick,red] (0,0.3) to (0,3.1);
    \end{tikzpicture}
    \ ,
\end{equation}
where there are a total of $2k$ closed Clifford tokens on alternating sides of the red strand.

For $i,j \in \tI$, define
\begin{equation}
    y_{ij} := - \sum_{k = \max(i,j)}^n (-1)^{p(i,k)p(j,k)} S(u_{ik}) u_{-k,-j} \in U_q.
\end{equation}
Note that $y_{ij}$ is of parity $p(i,j)$.  Next, for $i,j \in \tI$, and $m > 0$, define
\[
    y_{ij}^{(m)} := (-1)^{p(i)p(j)} \sum_{i=i_0,i_1,\dotsc,i_{m-1},i_m=j} (-1)^{\sum_{k=1}^{m-1} p(i_k) + \sum_{k=0}^{m-1} p(i_k)p(i_{k+1})} y_{i_0,i_1} \dotsm y_{i_{m-1},i_m}.
\]
These are isomeric analogues of the elements defined in \cite[(5.15)]{BSW-qheis}.

\begin{lem}
    We have
    \begin{equation} \label{yoshi}
        \redFunc_n
        \left(
            \begin{tikzpicture}[anchorbase]
                \draw[->] (0.2,-0.2) -- (0.2,0) \braidup (-0.2,0.4) \braidup (0.2,0.8);
                \draw[wipe] (-0.2,-0.2) -- (-0.2,0) \braidup (0.2,0.4) \braidup (-0.2,0.8);
                \draw[thick,red] (-0.2,-0.2) -- (-0.2,0) \braidup (0.2,0.4) \braidup (-0.2,0.8);
                \token{0.2,0};
                \token{-0.2,0.4};
            \end{tikzpicture}
        \right)
        = \sum_{i,j \in \tI} y_{ij} \otimes E_{ij}.
    \end{equation}
    where we interpret the right-hand side as a natural transformation whose $M$-component, for a $U_q$-supermodule $M$ with corresponding representation $\rho_M$, is $\sum_{i,k \in I} \rho_M (y_{ij}) \otimes E_{ij}$.
\end{lem}

\begin{proof}
    We have
    \begin{align*}
        \redFunc_n
        \left(
            \begin{tikzpicture}[anchorbase]
                \draw[->] (0.2,-0.2) -- (0.2,0) \braidup (-0.2,0.4) \braidup (0.2,0.8);
                \draw[wipe] (-0.2,-0.2) -- (-0.2,0) \braidup (0.2,0.4) \braidup (-0.2,0.8);
                \draw[thick,red] (-0.2,-0.2) -- (-0.2,0) \braidup (0.2,0.4) \braidup (-0.2,0.8);
                \token{0.2,0};
                \token{-0.2,0.4};
            \end{tikzpicture}
        \right)
        &= L^{-1} \circ \flip \circ (J \otimes 1) \circ \flip \circ L \circ (1 \otimes J)
        \\
        &= \left( \sum_{i \le k} S(u_{ik}) \otimes E_{ik} \right) (1 \otimes J) \left( \sum_{j \le l} u_{jl} \otimes E_{jl} \right) (1 \otimes J)
        \\
        &= - \left( \sum_{i \le k} S(u_{ik}) \otimes E_{ik} \right) \left( \sum_{j \le l} u_{jl} \otimes E_{-j,-l} \right)
        \\
        &= - \sum_{i \le k \ge -l} (-1)^{p(i,k)p(-k,l)} S(u_{ik}) u_{-k,l} \otimes E_{i,-l}.
    \end{align*}
    Replacing $l$ by $-j$ yields \cref{yoshi}.
\end{proof}

\begin{cor}
    We have
    \begin{equation} \label{koopa}
        \redFunc_n
        \left( \left(
            \begin{tikzpicture}[anchorbase]
                \draw[->] (0.2,-0.2) -- (0.2,0) \braidup (-0.2,0.4) \braidup (0.2,0.8);
                \draw[wipe] (-0.2,-0.2) -- (-0.2,0) \braidup (0.2,0.4) \braidup (-0.2,0.8);
                \draw[thick,red] (-0.2,-0.2) -- (-0.2,0) \braidup (0.2,0.4) \braidup (-0.2,0.8);
                \token{0.2,0};
                \token{-0.2,0.4};
            \end{tikzpicture}
        \right)^{\circ m} \right)
        = \sum_{i,j \in \tI} y_{ij}^{(m)} \otimes E_{ij}.
    \end{equation}
\end{cor}

\begin{prop} \label{shadow}
    For $m>0$, the image of $\leftzebrabub{2m}$ under \cref{canon} is
    \begin{equation} \label{dark}
        \sum_{i \in \tI} (-1)^{p(i)} q^{2|i|-2n-1} y_{ii}^{(m)}.
    \end{equation}
\end{prop}

\begin{proof}
    It suffices to compute the action of $\bFh_n(\leftzebrabub{2k})$ on the element $1$ of the regular representation.  Using \cref{goomba}, this is given by
    \[
        1
        \xmapsto{\bFh_n(\leftcup)} \sum_{k \in \tI} 1 \otimes v_k \otimes v_k^*
        \xmapsto{\cref{koopa}} \sum_{i,k \in \tI} y_{ik}^{(m)} \otimes v_i \otimes v_k^*
        \mapsto{\bF_n(\rightcap)} \sum_{i \in \tI} (-1)^{p(i)} q^{2|i|-2n-1} y_{ii}^{(m)}.
        \qedhere
    \]
\end{proof}

\Cref{shadow} is an isomeric analogue of \cite[(5.29)]{BSW-qheis}, giving the image of the analogous diagrams for the affine HOMFLYPT skein category, which is the $U_q(\mathfrak{gl_n})$-analogue of $\Qcat(z)$.  On the other hand, \Cref{shadow} can also be viewed as a quantum analogue of \cite[Th.~4.5]{BCK19}, which treats the degenerate (i.e.\ non-quantum) case.  In particular, the elements \cref{dark} are quantum analogues of central elements in $U(\fq_n)$ introduced by Sergeev in \cite{Ser83}; see \cite[Prop.~4.6]{BCK19}.  In the degenerate case, these elements generate the center of $U(\fq_n)$; see \cite[Prop.~1.1]{NS06}.  It seems likely that the elements \cref{dark} do not quite generate the center of $U_q(\fq_n)$, by analogy with the case of $U_q(\mathfrak{gl}_n)$, where one needs to add one additional generator; see \cite[Cor.~5.11]{BSW-qheis} and \cref{donut}\cref{donut2}.


\bibliographystyle{alphaurl}
\bibliography{QuantumIsomeric}

\end{document}